%% file: sample-distance-multivariance-arXiv.tex
\documentclass[10pt,a4paper,final]{article} %

\usepackage{enumerate}
\usepackage[usenames, dvipsnames]{xcolor} 
\usepackage{amsfonts} 
\usepackage{dsfont} 
\usepackage{amsmath} 
\usepackage{amsthm} 
\usepackage[myheadings]{fullpage} 
\usepackage[final]{graphicx}
\usepackage{bm} 
\usepackage{ifthen} 
\usepackage{mathtools} 
\usepackage{multirow}
\usepackage[section]{placeins} 
\usepackage{float} 
\usepackage{url}

\usepackage[nottoc,numbib]{tocbibind}

\usepackage[notcite,notref]{showkeys}

\makeatletter
\def\parsetime#1#2#3#4#5\empty{#1#2:#3#4}
\def\parsedate#1:20#2#3#4#5#6#7#8\empty{20#2#3/#4#5/#6#7-\parsetime#8\empty}
\def\moddate#1{\expandafter\parsedate\pdffilemoddate{#1}\empty}
\newcommand{\timestamp}{\today}
\makeatother

\usepackage[draft=false]{scrlayer-scrpage} 
\ihead{\textsc{Berschneider \& B\"ottcher / Quadratic forms \& sample distance multivariance}}

\newtheorem{theorem}{Theorem}[section]
\newtheorem*{theorem*}{Theorem}
\newtheorem{lemma}[theorem]{Lemma}
\newtheorem{remark}[theorem]{Remark}

\newtheorem{example}[theorem]{Example}

\newtheorem{corollary}[theorem]{Corollary}
\newtheorem{proposition}[theorem]{Proposition}
\newtheorem{ass}[theorem]{Assumption}

\newtheorem{rem}[theorem]{Remark}
\newtheorem*{test*}{Test}
\numberwithin{theorem}{section}
\numberwithin{equation}{section}

\newcommand{\A}{\mathbb{A}}
\newcommand{\B}{\mathbb{B}}
\newcommand{\C}{\mathbb{C}}
\newcommand{\E}{\mathbb{E}}
\newcommand{\G}{\mathbb{G}}
\newcommand{\pH}{\mathbb{H}}
\newcommand{\N}{\mathbb{N}}
\newcommand{\R}{\mathbb{R}}
\newcommand{\V}{\mathbb{V}}
\newcommand{\W}{\mathbb{W}}

\newcommand{\Prob}{\mathbb{P}}

\newcommand{\Askript}{\mathcal{A}}
\newcommand{\Bskript}{\mathcal{B}}

\newcommand{\Mskript}{\mathcal{M}}

\newcommand{\eqd}{\stackrel{d}{=}}

\newcommand{\One}{\mathds 1}

\newcommand{\Cov}{\operatorname{Cov}}

\newcommand{\ii}{\mathrm{i}}
\newcommand{\ee}{\mathrm{e}}
\newcommand{\vonen}{\vect{1}_n}

\newcommand{\hN}{\mbox{}^{\scriptscriptstyle N}\kern-1.5pt}

\newcommand{\pG}{\G}
\newcommand{\meanG}{\mu} 
\newcommand{\covG}{K} 
\newcommand{\pcovG}{C}  
\newcommand{\edag}[1]{#1^*}

\newcommand{\eye}{I}
\newcommand{\matB}{B}
\newcommand{\mat}[1]{#1} 
\newcommand{\vect}[1]{\mathbf{\bm{#1}}} 
\newcommand{\covUV}{\Sigma}
\newcommand{\itint}[1]{t^{(#1)}}
\newcommand{\sumalpha}[1]{\mu^{(#1)}}

\DeclareMathOperator{\skw}{skew}
\DeclareMathOperator{\exkurt}{exkurt}
\DeclareMathOperator{\range}{Range}
\DeclareMathOperator{\tr}{trace}

\renewcommand{\Re}{\operatorname{Re}}
\renewcommand{\Im}{\operatorname{Im}}

\newcommand{\pairslr}[3]{\left#1 #2 \right#3}
\newcommand{\pairscs}[4]{{\csname#1l\endcsname #2} #3 {\csname#1r\endcsname #4}}

\newcommand{\pairs}[4][lr]{%
        \ifthenelse{\equal{#1}{}}{%
                #2 #3 #4}{%
                \ifthenelse{\equal{#1}{lr}}{\pairslr{#2}{#3}{#4}}{%
                        \pairscs{#1}{#2}{#3}{#4}}}}
\newcommand{\abs}[2][lr]{\pairs[#1]{\lvert}{#2}{\rvert}}
\newcommand{\norm}[2][lr]{\pairs[#1]{\lVert}{#2}{\rVert}}
\newcommand{\innerp}[3][lr]{\pairs[#1]{\langle}{#2\,,#3}{\rangle}}
\newcommand{\Expect}[2][lr]{\E\ifthenelse{\equal{#1}{lr}}{\!}{}\pairs[#1]{(}{#2}{)}}
\newcommand{\Var}[2][lr]{\V\ifthenelse{\equal{#1}{lr}}{\!}{}\pairs[#1]{(}{#2}{)}}
\newcommand{\cov}[3][lr]{\Cov\ifthenelse{\equal{#1}{lr}}{\!}{}\pairs[#1]{(}{#2,#3}{)}}
\newcommand{\kronecker}[2]{\delta_{#1, #2}}

\newcommand{\conj}[1]{\overline{#1}} 
\newcommand{\cl}[1]{\overline{#1}} 
\DeclareMathOperator{\lin}{span}

\DeclareMathOperator{\colsum}{cs} 
\definecolor{darkgreen}{rgb}{0.0, 0.4, 0.13}

\begin{document}

\title{On complex Gaussian random fields, Gaussian quadratic forms and sample distance multivariance}
\author{Georg Berschneider\thanks{%
    Otto-von-Guericke-Universit\"at Magdeburg, Fakult\"at f\"ur Mathematik,
    Universit\"atsplatz 2, 39106 Magdeburg, Germany, email: \texttt{georg.berschneider@ovgu.de}}    
\and 
Bj\"orn B\"ottcher\thanks{TU Dresden, Fakult\"at Mathematik, Institut f\"ur
  Mathematische Stochastik, 01062 Dresden, Germany, email:
  \texttt{bjoern.boettcher@tu-dresden.de}, corresponding author}}
\date{\timestamp}
\maketitle

\begin{abstract}
The paper contains results in three areas: First we present a general estimate
for tail probabilities of Gaussian quadratic forms with known expectation and
variance. Thereafter we analyze the distribution of norms of complex Gaussian
random fields (with possibly dependent real and complex part) and derive
representation results, which allow to find efficient estimators for the
moments of the associated Gaussian quadratic form. Finally, we apply these
results to sample distance multivariance, which is the test statistic
corresponding to distance multivariance -- a recently introduced multivariate
dependence measure. The results yield new tests for independence of
multiple random vectors. These are less conservative than the classical
tests based on a general quadratic form estimate and they are (much) faster
than tests based on a resampling approach. As a special case this also
improves independence tests based on distance covariance, i.e., tests for
independence of two random vectors.
\end{abstract}


\tableofcontents

\section{Introduction}
\input{intro}

\section{Gaussian quadratic forms}\label{sec:quadratic-forms}
\input{quadratic-forms}

\section{Complex Gaussian random fields }\label{sec:gaussianprocess}
\input{random-fields}

\section{Sample distance multivariance}\label{sec:distance-multivariance}
\input{distance-multivariance}

\section{Examples, simulations and discussions}

\input{comparison}


\input{sample-distance-multivariance-arXiv.bbl}
{\small
\noindent\textsc{G. Berschneider:\\ Otto-von-Guericke-Universit\"at Magdeburg, Fakult\"at f\"ur Mathematik, Universit\"atsplatz 2, 39106 Magdeburg, Germany}, \url{georg.berschneider@ovgu.de} \\

\noindent\textsc{B. B\"ottcher:\\ TU Dresden, Fakult\"at Mathematik, Institut f\"ur Mathematische Stochastik, 01062 Dresden, Germany}, \url{bjoern.boettcher@tu-dresden.de}
}
\newpage
\section{Appendix}
\input{appendix}

\end{document}

%% file: intro.tex
Distance multivariance, total distance multivariance and $m$-multivariance were
recently introduced as measures of dependence for multiple random vectors
\cite{BoetKellSchi2018, BoetKellSchi2018a, Boet2019}, these extend the concept
of distance covariance introduced by Sz\'ekely, Rizzo and Bakirov
\cite{SzekRizzBaki2007}. Moreover, distance multivariance also extends the
approach to test multivariate (in)dependence of Bilodeau and Guetsop Nangue
\cite{BiloNang2017}. For random variables $X_i$ with values in $\R^{d_i}$,
$1\leq i\leq n$, the total distance multivariance $\overline{M}(\vect{X})$ of
$\vect{X} := (X_1,\ldots,X_n)$ is zero if and only if the random variables are
independent. See \cite{Boet2019} for a concise introduction to distance
multivariance, as well as many examples and a comparison to other dependence
measures. In \cite{BoetKellSchi2018a} conservative distribution-free tests of
independence based on distance multivariance were presented. Using a
resampling approach less conservative but computationally more expensive tests
are possible, \cite{Boet2019} (see also \cite{BiloNang2017}). In this paper we
analyze the proposed test statistics in more detail, in order to derive less
conservative tests without the computational expense of the resampling
approach. On the one hand we consider the limit distribution (under $H_0$)
using a spectral approach, this is related to known results: For
the special case $n=2$, e.g., Gretton et al.\ \cite{GretFukuTeoSongScho2008},
and to results of Fan et al.\ \cite{FanMichPeneSalo2017} for a different
multivariate estimator which goes back at least to Kankainen
\cite{Kank1995}. On the other hand we also analyze the finite sample
distribution, here some of our results look similar to an approach for
permutation tests by Bilodeau and Guetsop Nangue \cite[Chapter 3, Section
3.4.]{Guet2017}. But note that in \cite{Guet2017} the moments of the finite
sample permutation statistic (with possibly dependent components) are
considered, while, in contrast, we consider the moments of the
finite sample estimator (without permutations) under $H_0$.
Similar to the moment method proposed in \cite[Chapter 2, Section
2.5]{Guet2017} our results also improve the (performance of the) tests for the
basic case ($n=2$) of distance covariance \cite{SzekRizzBaki2007}, this is due
to the fact that the methods provide in this case p-value
estimates without the use of a resampling technique, see, e.g., Examples
\ref{ex:mv_bern} and \ref{ex:speed}. The derived methods are implemented in
the \texttt{R} package \texttt{multivariance} \cite{Boet2019R-2.0.0}.

Along the way at least two results which are also of general interest --
without the context of distance multivariance -- are proved:
\begin{itemize}
\item A general tail estimate for Gaussian quadratic forms: Let $0 \leq
  \alpha_i \leq \alpha \leq 1$ for $i\in \N$ with $\sum_{i\in\N} \alpha_i
  = 1$ and $Z_i$ be independent standard normal random variables then 
\begin{equation} \label{eq:introtail}
\Prob\left(\sum_{i\in\N} \alpha_i Z_i^2 \geq x\right) \leq \Prob\left(\alpha
  Y_{\frac{1}{\alpha}} \geq x\right) \text{ for all }x\geq x_0,
\end{equation} 
where $Y_{\frac{1}{\alpha}}$ is a chi-squared distributed random variable with
(fractional) $\frac{1}{\alpha}$ degree of freedom and $x_0$ depends on
$\alpha$ and is bounded from above by $1.5365$, see Theorem
\ref{thm:tail}. This result extends the special case $\alpha = 1$ which was
treated in \cite{SzekBaki2003}. Their estimate was based on the first moment
only, other estimates require at least the knowledge of moments up to order 3,
e.g.\ \cite{LiuTangZhan2009}. Thus in a sense our estimate fills the gap,
since it can be applied if only the first two moments are known, cf.\ Remark
\ref{rem:tail}.\ref{rem:tail:var}. For further references and a comparison
with other estimates for Gaussian quadratic forms see Equation
\eqref{eq:qf-estimates}.
\item Based on a representation of the squared $L^2(\rho)$-norm
  $\norm{\pG}^2_\rho$ of a complex Gaussian random field $\pG$ -- with
  possibly dependent real and imaginary parts -- as positive Gaussian
  quadratic form
 \begin{equation} \label{eq:introKLex}
 	\norm{\pG}_\rho^2 = \sum_{i\in I} \alpha_i Z_i^2 \qquad \Prob\text{-a.s.},
 \end{equation}
 where $\alpha_i$, $i\in I\subseteq\N$,  are the eigenvalues of the
 covariance operator associated to the random field and $Z_i$ are independent
 standard normally distributed random variables  (see, e.g.,
  \cite{Duchetal2016,FanMichPeneSalo2017,Kuo1975}) we present  -- for
  applications most important -- explicit representations of the sums
  $\sumalpha{k}=\sum_{i\in I} \alpha_i^k$ of the coefficients in terms of the
  covariance kernel of the random field $\pG$, see Proposition
  \ref{prop:pnormintegral}.
\end{itemize}

In the setting of distance multivariance (see Section \ref{sec:multivariance}
for the underlying definitions) it is known that the limit (as the sample size
$N$ increases to $\infty$) of sample distance multivariance $N \cdot \hN
M_\rho^2$ is distributed as a Gaussian quadratic form which is obtained as the
$L^2$-norm of a Gaussian random field. It is denoted by $\norm{\G}^2_\rho$ and it
(obviously) fits into the framework of \eqref{eq:introtail} and
\eqref{eq:introKLex}. So far its exact distribution was
intangible. In Section \ref{sec:multivariance} we will derive many useful
results about it, e.g., an explicit representation in Proposition
\ref{prop:replimit} which simplifies the previously known representation
\cite[Eq.\ (S.15)]{BoetKellSchi2018a-supp}. But here in the introduction we
just want to mention the results which seem of major practical
importance. First of all we derive explicit
estimators for the moments of the limit distribution (analogous
results for total- and $m$-multivariance are also proved in Section
\ref{sec:total}). 

\begin{theorem*}[Corollary \ref{cor:estlimitmom}] 
Let $\E(\psi_i(X_i))<\infty$ for $1\leq i\leq n$,
$(\vect{x}^{(1)},\dots,\vect{x}^{(N)})$ be samples of $(X_1,\dots,X_n)$ (with
possibly dependent components!) and let $\norm{\G}^2_\rho$ be the
distributional limit of the test statistic $N \cdot \hN M_\rho^2$ under the
hypothesis of independence. Set $S := \{1,\ldots,n\},$ then
\begin{eqnarray*}
 \hN\mu_S^{(1)}&\xrightarrow{N \to \infty}& \E(\norm{\pG}^2_\rho),\\
 2\cdot \hN\mu_S^{(2)}&\xrightarrow{N \to \infty}& \V(\norm{\pG}^2_\rho),\\
 8 \cdot \hN\mu_S^{(3)}&\xrightarrow{N \to \infty}& \E\left[(\norm{\pG}^2_\rho -
                                                    \E(\norm{\pG}^2_\rho))^3\right],
   \\
 48\hN\mu_S^{(4)}+ 12 \cdot\left( \hN\mu_S^{(2)}\right)^2&\xrightarrow{N \to
                                                           \infty}&
                                                                    \E\left[(\norm{\pG}^2_\rho
                                                                    -
                                                                    \E(\norm{\pG}^2_\rho))^4\right],
\end{eqnarray*}
where the $\hN\mu_S^{(k)}$ can be computed directly from the (possibly
dependent) samples. Here the convergence is in the sense of the strong law
of large numbers (SLLN), i.e., almost sure convergence if in the estimators
the samples are replaced by the corresponding random variables. The formulas
for $\hN\mu_S^{(k)}$ are given in Corollary \ref{cor:estlimitmom} on page
\pageref{eq:muestimators}. (See also Remark \ref{rem:unbiased} for the
corresponding unbiased estimators.)
\end{theorem*}

It turns out that for the case of $n=2$ and $n=3$ the above provides also good
estimates for the distribution of sample distance multivariance $N \cdot \hN
M_\rho^2$ for small sample sizes $N$ (cf.\ Example \ref{ex:jointmom}). But
beware, in general for larger $n$ the parameters of the distribution of $N
\cdot \hN M_\rho^2$ are not well approximated by the above result -- this is
true even for reasonable sample sizes like $N=100$. In particular the variance
of the limit can be much lower than the variance of the finite sample
estimator $N \cdot \hN M_\rho^2$ (cf. Remark
\ref{rem:finitemoms}.\ref{rem:finitemoms:var} and Example
\ref{ex:jointmom}). Nevertheless, also for the finite sample case we derive by
careful analysis the following explicit formulas (which again can be estimated
directly from possibly dependent samples). 

\begin{theorem*}[Theorem \ref{thm:finitemoms}]
Let $X_1,\ldots,X_n$ be independent random variables with
$\E(\psi_i(X_i)^2)<\infty$ for $1\leq i\leq n$. Then
\begin{align*}
&\E( N \cdot \hN M_\rho^2 (\vect{X}^{(1)},\ldots, \vect{X}^{(N)}))  =
  \frac{(N-1)^n + (-1)^n (N-1)}{N^n} \prod_{i=1}^n \E(\psi_i(X_i-X_i')),\\
&\E( [N\cdot \hN M_\rho^2(\vect{X}^{(1)},\ldots,\vect{X}^{(N)} )]^2) =\\
&\hspace{3.1cm}
\frac{1}{N^2}\sum_{k=1}^7C(N,k)\cdot
\prod_{i=1}^n \left[\frac{b(N,k) b_i + c(N,k) c_i + d(N,k)d_i}{N^4}\right],
\end{align*}
where $b_i:=\E(\psi_i(X_i-X_i')^2)$,
$c_i:=\E(\psi_i(X_i-X_i')\psi_i(X_i'-X_i''))$, $d_i:=[\E(\psi_i(X_i-X_i'))]^2$
and the coefficients $C(N,k),b(N,k),c(N,k),d(N,k)$  are given in
Table \ref{tab:coef} on page \pageref{tab:coef}. (For the estimators see
Remarks \ref{rem:finitemoms} and \ref{rem:unbiased}.)
\end{theorem*}

Finally, recall that in \cite{Boet2019,BoetKellSchi2018a,SzekRizzBaki2007} for
the independence tests the so-called normalized sample distance multivariance
$N \cdot \hN \Mskript_\rho^2$ was used (which is just a scaled version of
$N\cdot  \hN M_\rho^2$, see Section \ref{sec:multivariance}). But note that
the scaling therein also depends on the sample and thus for finite samples the
distribution has to be analyzed jointly with the scaling factor. This is done
in Theorem \ref{thm:finitemomsnormalized}.

The methods can also be extended to total- and $m$-multivariance, see Section
\ref{sec:total}. Moreover the explicit knowledge of the expectation and
variances can also be used to construct further (new) tests of
$m$-independence using the central limit theorem (see Remark
\ref{rem:clttest}). 

For readers with an interest to apply our results in the context of distance
multivariance we recommend the overview of the methods in Section
\ref{sec:tests} and the examples in Section \ref{sec:examples}. Based on the
examples, especially using the extensive study in Example \ref{ex:all}, we
have the following remarks and recommendations for performing independence
tests based on distance multivariance:

\begin{itemize}
\item Estimate \eqref{eq:introtail} provides more powerful tests than the
  classical estimate \eqref{eq:c1}. Moreover, the tests show (at least in all
  of our examples) conservative behavior when used with the unbiased finite
  sample estimators for the parameters.
\item Using any of the proposed methods to estimate the p-value is faster than
  the resampling approach (Example \ref{ex:speed}).
\item Using the unbiased finite sample estimators for the mean and variance of
  normalized (total-, $m$-) multivariance (Theorem
  \ref{thm:finitemomsnormalized}, Corollary \ref{cor:mfinitemomnormal})
  together with the unbiased estimator for the skewness (Corollaries
  \ref{cor:estlimitmom}, \ref{cor:estlimitmmom}) in Pearson's estimate
  \eqref{eq:pearson} provides p-value estimates which have the smallest
  relative mean squared error in comparison to the benchmark. This holds for
  the samples and methods studied in Example \ref{ex:all}, which are so rich
  that it seems to be reasonable to generalize this statement. 
\item Also for the special case of (generalized) distance covariance
  (\cite{SzekRizzBaki2007,BoetKellSchi2018}; distance multivariance with
  $n=2$) the above method is recommended.
\item For the case of analyzing pairwise or triple dependence of
  identically distributed marginals with $m$-multivariance the central
  limit theorem based method \eqref{eq:clt} performs very similarly to the
  above method (cf.\ Figure \ref{fig:all-nlarge} in the
  Appendix), while it requires only estimates of the first and second moment.
\item If the marginal distributions are known, one can compute (before
  testing) the required parameters, thereby speeding up further
  computations and achieving higher accuracy (Example
  \ref{ex:momknownmarginals}).
\item Using normalized multivariance instead of multivariance without
  normalization is always recommended (see the discussion before Remark
  \ref{rem:bias_vs_conservative} and their performance in Example
  \ref{ex:all}).
\item All methods discussed rely on the existence of certain moments, if this
  existence is in doubt there is no theoretical support for their use. In this
  case, the classical estimate or a resampling approach are the fundamental
  options, but also these require some basic moment conditions (Remark
  \ref{rem:momcond}). Alternatively a transformation can be used to overcome
  any moment condition \cite[Remark 2.7.2]{Boet2019}. Nevertheless there is also
  some indication for robust behavior of the presented methods (Example
  \ref{ex:robust}), but in our opinion this requires further
  investigation.
\end{itemize}


%% file: quadratic-forms.tex
We start with a general tail estimate for Gaussian quadratic forms. See
\eqref{eq:qf-estimates} for a comparison with other existing methods.

\begin{theorem}[Tail estimate for Gaussian quadratic forms]\label{thm:tail}
Let $Z_i$, $i\in\N$, be independent standard normal random variables and
$(\alpha_i)_{i\in\N}$ be a sequence in $\R$ with $0\leq \alpha_i \leq \alpha
\leq 1$ and $\sum_{i\in\N} \alpha_i = 1$. Then there exists a smallest $x_0 =
x_0(\alpha,(\alpha_i)_{i\in\N})  \geq 0$ such that
\begin{equation} \label{eq:tail}
\Prob\left(\sum_{i\in\N} \alpha_i Z_i^{2}\geq x \right) \leq \Prob\left(\alpha
  Y_\frac{1}{\alpha} \geq x\right) \quad \text{ for all }x\geq x_0,
\end{equation}
where $Y_\frac{1}{\alpha}$ is a chi-squared distributed random variable with
(fractional) $\frac{1}{\alpha}$ degree of freedom.

The value $x_0^\star(\alpha):= \sup \left\{ x_0(\alpha,(\alpha_i)_{i\in\N})
  \mid 0\leq \alpha_i \leq \alpha \leq 1,\ \sum_{i\in\N} \alpha_i = 1
\right\}$ is bounded from above by $1.5365$, i.e., inequality \eqref{eq:tail}
holds uniformly for all $x\geq 1.5365$, all $\alpha <1$ and all $\alpha_i$
with $0\leq \alpha_i \leq \alpha$ and $\sum_{i\in\N} \alpha_i = 1$.
\end{theorem}
\begin{remark} \label{rem:tail}
\begin{enumerate}
\item \label{rem:tail:var}To use the bound \eqref{eq:tail} a suitable $\alpha$
  has to be known or estimated. A simple choice is
  \begin{equation*}
    \alpha : = \sqrt{\frac{1}{2} \V \left( \sum_{i\in\N} \alpha_i
        Z_i^{2}\right)},
  \end{equation*}
  cf.\ Lemma \ref{lem:Q4}. Note that $\alpha = \max_{i\in\N} \alpha_i$
  is the optimal choice -- this is nothing but the spectral radius of the
  associated integral operator $T_\covG$, see \eqref{eq:operator}. 
\item \label{rem:tail:x0} For $\alpha \in (0,1]$ let $x_0(\alpha)$ be such
  that $0< \Prob( \frac{1}{\lceil 1 / \alpha \rceil }Y_{\lceil
    \frac{1}{\alpha}\rceil} \leq x_0(\alpha)) = \Prob(\frac{1}{\lceil
    1/\alpha\rceil +1 }Y_{\lceil \frac{1}{\alpha}\rceil +1} \leq x_0(\alpha))
  < 1$. By \cite[Prop.\ 1.(i)]{SzekBaki2003} $x_0(\alpha)$ is unique, and in
  our setting $x_0(\alpha)$ is an upper bound for $x_0^\star(\alpha)$ of
  Theorem \ref{thm:tail}.  For all $\alpha \in (0,1]$ the corresponding tail
  probability  $1-\Prob(\alpha Y_\frac{1}{\alpha} \geq x_0(\alpha))$ is larger
  than 0.215, which is well above any commonly used significance level (for
  details see Section \ref{sec:x0} in the Appendix, especially Figure
  \ref{fig:x0alphaprob-1}). Thus the estimate is proper for hypothesis
  testing.
\item 
The special case $\alpha = 1$ of \eqref{eq:tail}, i.e.,
\begin{equation} \label{eq:qfSB}
\Prob\left(\sum_{i\in\N} \alpha_i Z_i^{2}\geq x \right) \leq \Prob(Y_1 \geq x)
\quad \text{ for all }x\geq x_0,
\end{equation}
was proved in \cite{SzekBaki2003}. Therein also bounds for smaller $x$ can be
found. These bounds are of the form $\Prob\left(\frac{1}{n} Y_n \geq x\right)$
for some $n = n(x) \in \N$.

\item \label{rem:tail:chi} Recall that for $\alpha, \beta > 0$ and
  $Y_\frac{1}{\alpha} \sim \chi^2(\frac{1}{\alpha})$ (i.e., chi-squared
  distributed with parameter $\frac{1}{\alpha}$) the density and
  characteristic function of $\beta Y_\frac{1}{\alpha} $ are given by 
\begin{equation*}
 p_{\beta Y_\frac{1}{\alpha}}(x)=\frac{x^{\frac{1}{2\alpha}-1}
   e^{-\frac{x}{2\beta}}}{(2\beta)^{\frac{1}{2\alpha}}\Gamma(\frac{1}{2\alpha})} 
 \One_{(0,\infty)}(x)  \quad\text{ and }\quad f_{\beta
  Y_\frac{1}{\alpha}}(t)=(1-2\ii\beta t)^{-\frac{1}{2\alpha}},
\end{equation*}
respectively. In other words, $\beta Y_\frac{1}{\alpha}$ is
$(\frac{1}{2\alpha},2\beta)$-gamma distributed or has the Pearson Type III
distribution $p_{\text{III}}(\frac{1}{2\alpha}, 2\beta,0)$, see Section
\ref{sec:pearson} in the Appendix. The density of $\beta Y_\frac{1}{\alpha}$
is strictly monotonically decreasing for $\frac{1}{\alpha}\leq 2$, and for
independent $Y_\frac{1}{\alpha_i}^{(i)} \sim \chi^2(\frac{1}{\alpha_i})$ with
$\frac{1}{\alpha_i} >0$
\begin{equation}\label{eq:chisum} 
  \sum_{i\in\N} Y_{\frac{1}{\alpha_i}}^{(i)} \eqd Y_{\sum_{i\in\N} \frac{1}{\alpha_i}}
\end{equation}
holds.
\end{enumerate}
\end{remark}
\begin{proof}[Proof of Theorem \ref{thm:tail}]
Let the assumptions of the theorem hold. Without loss of generality we assume
that $(\alpha_i)_{i\in\N}$ is monotonically decreasing and denote
by $n_1$ the largest $n \in \N$ such that $\alpha_1 = \alpha_n$. 

If $\alpha_1 = \frac{1}{n_1}$ then the statement is trivial, since in this
case $\sum_{i\in\N} \alpha_i Z_i^2 \eqd \alpha_1 Y_{n_1} = \alpha_1
Y_\frac{1}{\alpha_1}$ by \eqref{eq:chisum}, and its tail distribution is
dominated by $\alpha Y_\frac{1}{\alpha}$ for any $\alpha > \alpha_1$. The
latter is a direct consequence of  $\frac{1}{\alpha_1} > \frac{1}{\alpha}$ and
thus
\begin{equation*}
\frac{p_{\alpha_1Y_\frac{1}{\alpha_1}}(x)}{p_{\alpha Y_\frac{1}{\alpha}}(x)} =
c_{\alpha_1,\alpha} x^{\frac{1}{2} \left(\frac{1}{\alpha_1} -
    \frac{1}{\alpha}\right)} e^{-{\frac{x}{2} \left(\frac{1}{\alpha_1} -
      \frac{1}{\alpha}\right)}} \xrightarrow{x\to \infty} 0.
\end{equation*}
It remains to consider the case  $\alpha_1 < \frac{1}{n_1}.$ By \cite[Equation
(5)]{Zolo1961} the density of $\sum_{i\in\N} \alpha_i Z_i^2$ has the form
\begin{equation*}
p(x) = c \frac{ \left(\frac{x}{2\alpha_1}\right)^{\frac{n_1}{2}-1}
  e^{-\frac{x}{2\alpha_1}}}{\Gamma(\frac{n_1}{2})} (1+\varepsilon(x))
\end{equation*} 
with $\varepsilon(x)\xrightarrow{x\to \infty} 0$ and some constant $c$. Thus
\begin{equation}\label{eq:densityqoutient}
\frac{p(x)}{p_{\alpha Y_\frac{1}{\alpha}}(x)} = \tilde c
x^{\frac{1}{2}(n_1-\frac{1}{\alpha})} e^{-\frac{x}{2} (\frac{1}{\alpha_1} -
  \frac{1}{\alpha})} (1+\varepsilon(x)) \xrightarrow{x \to \infty} 0
\end{equation}
since either $\alpha_1 < \alpha$, and hence $\frac{1}{\alpha_1}-
\frac{1}{\alpha} >0$ implies the limit in \eqref{eq:densityqoutient}, or
$\alpha_1 = \alpha$ and thus the exponential term is equal to 1 and $n_1 -
\frac{1}{\alpha} = n_1 - \frac{1}{\alpha_1} < 0$ implies the limit in
\eqref{eq:densityqoutient}. 

Finally, note that \eqref{eq:densityqoutient} implies that the density of
$\sum_{i\in\N} \alpha_i Z_i^2$ is dominated by the density of $\alpha
Y_\frac{1}{\alpha}$ for sufficiently large values, and thus \eqref{eq:tail}
holds.  This method proves the existence of $x_0$ but it does not provide a
bound.

To get a bound for $x_0$ one can follow  exactly the proof of \cite[Theorem
1]{SzekBaki2003} with the additional restriction: $\lambda_i \leq \beta \leq
1$ for all $i$ (where the $\lambda_i$ are the coefficients of the quadratic
form in their notation). Their proof is technical, long and uses many
auxiliary results. It seems reasonable to omit a replication of the details
here. In essence, differential calculus yields
\begin{multline*}
\inf_{\substack{0 \leq \alpha_i \leq \beta\\ \sum_{i\in\N} \alpha_i = 1}}
\Prob\Bigl(\sum_{i\in\N} \alpha_i Z_i^{2} \leq x\Bigr)  =\\
\min\left\{\inf_{n\in\N, \frac{1}{n}\leq \beta} \Prob(\textstyle \frac{1}{n}
  Y_n \leq x),\ \Prob(\beta Y_{\lfloor \frac{1}{\beta}\rfloor}^{(1)} +
  (1-\beta \lfloor \frac{1}{\beta}\rfloor ) Y_1^{(2)} \leq x) \right\},
\end{multline*}
where $Y_n \sim \chi^2(n)$ and $Y_{r_i}^{(i)}\sim \chi^2(r_i)$ for $i=1,2$ are
independent. 
Furthermore, let $x_0(\alpha)$ be such that $0< \Prob( \frac{1}{\lceil 1 /
  \alpha \rceil }Y_{\lceil \frac{1}{\alpha}\rceil} \leq x_0(\alpha)) =
\Prob(\frac{1}{\lceil 1/\alpha\rceil +1 }Y_{\lceil \frac{1}{\alpha}\rceil +1}
\leq x_0(\alpha)) < 1$. Then by \cite[Prop. 1', p. 189]{SzekBaki2003} the
function $\alpha \mapsto x_0(\alpha)$ is increasing on $(0,1]$, bounded by
$x_0(1)\approx 1.536404$ and for any $\beta \in (0,1]$
\begin{equation*}
\Prob(\textstyle \frac{1}{\lceil 1/\beta \rceil} Y_{\lceil 1/\beta \rceil}
\leq x) = \inf_{n\in\N, \frac{1}{n}\leq \beta} \Prob(\textstyle \tfrac{1}{n}
Y_n \leq x) \text{ for all } x \geq x_0(\beta).
\end{equation*} 
Finally, it can be verified (numerically; for details and discussion see
Section \ref{sec:x0} in the Appendix), that with $\overline \beta:= (1-\beta
\lfloor \frac{1}{\beta}\rfloor )$ 
\begin{equation}
\label{eq:finalx0}
\min\left\{\Prob(\textstyle \frac{1}{\lceil 1/\beta \rceil} Y_{\lceil 1/\beta
    \rceil} \leq x),\ \Prob(\beta Y_{\lfloor \frac{1}{\beta}\rfloor}^{(1)} +
   \overline \beta Y_1^{(2)} \leq x) \right\} \ \geq \ 
\Prob(\alpha Y_\frac{1}{\alpha} \leq x)
\end{equation}
for all $x \geq x_0(\alpha)$ and all $\beta \leq \alpha$. 
\end{proof} 

On can summarize the available options to estimate the tails of positive
(i.e., all coefficients $\alpha_i\geq 0$) Gaussian quadratic forms
as follows (here $Y_r \sim \chi^2(r)$ and the random variables $Z_i$ are
independent standard normally distributed):
 \begin{equation}\label{eq:qf-estimates}
\Prob\Bigl(\sum_{i\in\N} \alpha_i Z_i^2\ge x\Bigr)  
\begin{cases}
  \leq \Prob(Y_1
  \geq x) & \text{for } x\geq x_0,\ \sum_{i\in\N} \alpha_i = 1,  \text{
    cf. \cite{SzekBaki2003}};\\ 
  \leq \Prob\bigl(\alpha Y_{\frac{1}{\alpha}}\geq x\bigr) & \text{for } x\geq x_0,\
  \sum_{i\in\N} \alpha_i=1,\ \alpha_i\le\alpha,\\
  &  \text{see
    Theorem~\ref{thm:tail}};\\
  \approx \Prob\bigl(Y_{g((\alpha_i)_i)}\ge & \text{if } \sum_{i\in\N}\alpha_i^k
  \text{  are known for } k\leq 4,\\
  \phantom{\approx\Prob\bigl(Y} h((\alpha_i)_i,x)\bigr)       & \text{i.e.,
    first four moments are known,}\\
  & \text{see Lemma~\ref{lem:Q4} and cf.~\cite{LiuTangZhan2009};}\\
  \approx \text{Series } & \text{if some of the } \alpha_i \text{ are
    known,}\\
  & \text{e.g., \cite{Kotz1967,Rube1962};}\\
  \approx \text{Fourier inversion} & \text{if some of the } \alpha_i \text{ are
    known}\\
  &\text{e.g., \cite{Davies1980,Imhof1961};}\\
  = \text{explicit } & \text{if all } \alpha_i \text{ are  known, e.g.,
    \cite{Tzir1987}.}
\end{cases}
\end{equation}
Implementations of some of the methods are available for \texttt{R} in the
 package \texttt{CompQuadForm} \cite{DuchdeM2010}. Furthermore, quantile
asymptotics have also been investigated by Jaschke et al.\
\cite{Jaschke2004}. For further representations of the distribution we refer
to \cite[Chapter 4]{Mathai1992} and the references within. 
In Section \ref{sec:examples} these estimates will be compared for the
Gaussian quadratic forms related to distance multivariance. In the next
section  the $\alpha_i$ are computed explicitly for the case that the Gaussian
quadratic form is the $L^2$-norm of a complex Gaussian random field.

For the approximation by Liu et al.\ \cite{LiuTangZhan2009} the
first four moments of the Gaussian quadratic form are required. Formulas for
arbitrary moments of general quadratic forms are available, cf.~{\cite[Theorem
  3.2b.2]{Mathai1992}}. In our setting these reduce to the following Lemma,
 a direct proof can be found in Section~\ref{sec:proof:Q4} in the
Appendix. 

\begin{lemma}[Moments of Gaussian quadratic forms]\label{lem:Q4}
  With the notation from Theorem \ref{thm:tail} let $Q := \sum_{i\in\N}
  \alpha_i Z_i^{2}$. Then
  \begin{align*}
    \Expect{Q} &= \sum_{i\in\N} \alpha_i,  &  \Expect{(Q-\Expect{Q})^3}  &= 8
    \sum_{i\in\N} \alpha_i^3,\\ 
    \Var{Q} = \Expect{(Q-\Expect{Q})^2}&= 2\sum_{i\in\N} \alpha_i^2,  &
    \Expect{(Q-\Expect{Q})^4} &= 48 \sum_{i\in\N} \alpha_i^4 + 3\Var{Q}^2.
  \end{align*}
 Moreover, if $\alpha_i \geq 0$ then $\Expect{Q}<\infty$ implies that all
 moments are finite.
\end{lemma}


%% file: random-fields.tex
This section is devoted to establishing a connection between complex
Gaussian random fields and positive Gaussian quadratic forms. Main tools
will be a Mercer representation of the covariance kernel and an associated
Karhunen-Lo\`eve type expansion of the random field. Results linking
moments of quadratic forms to integrals of the covariance kernel allow to
connect the estimates of Section~\ref{sec:quadratic-forms} to our applications
in the setting of distance multivariance in the upcoming
Section~\ref{sec:distance-multivariance}. 

\medskip
We start with a reminder on general second-order random
fields. Throughout, all vector spaces are interpreted as $\C$-linear
spaces. For $d\in\N$ and a probability space $(\Omega,\Askript,\Prob)$
(allowing for the existence of normally distributed random variables) we call 
$\pG:\R^d\to L^2(\Omega,\Askript,\Prob)$ 
a complex second-order random field, i.e., $\pG(x)$, $x\in\R^d$, are
square-integrable complex-valued random variables;
$\meanG(t):=\Expect{\pG(t)}$, $t\in\R^d$, is its mean, its 
covariance kernel $\covG$ and pseudo-covariance kernel $C$ are given by,
\begin{align*}
  \covG(s,t) &:=
  \Expect[big]{(\pG(s)-\meanG(s))\conj{(\pG(t)-\meanG(t))}},\\ \pcovG(s,t)
  &:=\Expect[big]{(\pG(s)-\meanG(s))(\pG(t)-\meanG(t))},
\end{align*}
for $s,t\in\R^d$, respectively.

A second-order random field is called
complex Gaussian if all its finite-dimensional distributions are
complex normal distributions, see Section~\ref{sec:complexnormal} in the
Appendix for the definition and basic properties of the complex normal
distribution. Note that we do not assume the pseudo-covariance to vanish
identically, and thus, allow for dependent real and imaginary parts.

We recall some basic facts about second-order random fields.
\begin{proposition}[Properties of the covariance kernel; cf.~{\cite[Chapter
	2]{Sas2013}}]\label{prop:randomfields}\hfill

Let $\pG$ be a complex second-order random field with mean $\meanG$, covariance
kernel $\covG$ and pseudo-covariance kernel $\pcovG$.
\begin{enumerate}
\item The covariance kernel is positive definite, i.e., for all choices of
  finitely many points $t_1,\dotsc,t_m\in\R^d$ and coefficients
  $c_1,\dotsc,c_m\in\C$
  \begin{equation*}
    \sum_{i,j=1}^m c_i\conj{c_j} \covG(t_i,t_j) =
    \Expect{\abs{\sum_{i=1}^m c_i(\pG(t_i)-\meanG(t_i))}^2}\geq 0.
  \end{equation*}
  This implies that $\covG$ is hermitian, i.e.,
  $\covG(s,t)=\conj{\covG(t,s)}$.\par
  The pseudo-covariance kernel is symmetric, i.e., $\pcovG(s,t)=\pcovG(t,s)$.
\item If $\meanG \equiv 0$ and $\pG(-t) = \conj{\pG(t)}$
  (or the weaker requirement $\covG(s,t) = \pcovG(s,-t)$ for all $s,t\in\R^d$)
  then the field is a hermitian, centered complex
  second-order field. Note that in this case
\begin{equation*}
   \covG(s,t) = \pcovG(s,-t) = \pcovG(-t,s) = \covG(-t,-s) =
   \conj{\covG(-s,-t)},\qquad s,t\in\R^d.
\end{equation*}
\item A complex Gaussian random field is uniquely determined by its
mean function and its covariance and pseudo-covariance kernels.
\item Values of linear functionals of complex Gaussian random fields are
  (scalar) complex normally distributed random variables.
\end{enumerate}
\end{proposition}

In order to link random fields to quadratic forms discussed in Section
\ref{sec:quadratic-forms}, we assume $\rho$ to be a non-negative, symmetric,
$\sigma$-finite Borel measure on $\R^d$. For notational simplicity we write
$L^2(\rho):=L^2(\R^d,\Bskript(\R^d),\rho)$, denote the inner product in
$L^2(\rho)$ by
\begin{equation*}
  \innerp{u}{v} = \innerp{u}{v}_\rho:= \int_{\R^d}
  u(t)\conj{v(t)}\,\rho(dt),
\end{equation*}
and the norm by $\norm{u}_\rho =\sqrt{\innerp{u}{u}}$, for all $u,v\in
L^2(\rho)$.
\medskip

All statements in this section are made under the following assumptions.
\begin{ass}\label{ass:general}
Let $d\in\N$ and $\pG:\R^d\to L^2(\Omega,\Askript,\Prob)$ be a hermitian,
centered, measurable complex second-order random field. Denote by
$\covG$ its covariance kernel and let $\rho$ be a non-negative, symmetric,
$\sigma$-finite Borel measure satisfying
\begin{equation}\label{eq:integrability}
  \int_{\R^d} \covG(t,t)\,\rho(dt)<\infty.
\end{equation}
\end{ass}%
Note that in the statements of the results, below, all necessary
assumptions will be repeated to avoid ambiguity.

Positive-definiteness of the covariance kernel $\covG$ implies the
Cauchy-Schwarz type inequality $\abs{\covG(s,t)}^2\leq
\covG(s,s)\covG(t,t)$. Therefore, the integrability
condition \eqref{eq:integrability} directly gives
\begin{equation*}\label{eq:integrability2}
   \int_{\R^d\times\R^d} \abs{\covG(s,t)}^2\,(\rho\otimes\rho)(ds,dt)\leq
   \biggl(\int_{\R^d} \covG(t,t)\,\rho(dt)\biggr)^2<\infty,
\end{equation*}
making 
\begin{equation}\label{eq:operator}
  T_\covG: L^2(\rho)\to L^2(\rho),\ u\mapsto \int_{\R^d}
  \covG(\bullet,t)u(t)\,\rho(dt),
\end{equation}
a well-defined, positive, compact and self-adjoint operator on
$L^2(\rho)$. Using Mercer's theorem (see, e.g., \cite{SteiSco2012}), there
exists an index set $I\subseteq\N$, a monotonically decreasing sequence
$(\alpha_i)_{i\in I}$ of positive -- not necessarily distinct -- real numbers
and an orthonormal system $(e_i)_{i\in I}$ in $L^2(\rho)$ such that
\begin{equation*}
  T_\covG e_i = \alpha_ie_i,\ i\in I,\quad \range T_\covG = \cl{\lin\{e_i:\
    i\in I\}},
\end{equation*}
and Mercer's representation holds:
\begin{equation}\label{eq:mercer}
  \covG(s,t) = \sum_{i\in I} \alpha_i e_i(s)\conj{e_i(t)},\qquad \text{for
    } \rho\otimes\rho\text{-almost all } (s,t)\in\R^d\times\R^d.
\end{equation}
Under the integrability condition \eqref{eq:integrability}, we obtain
\begin{equation*}
   \sum_{i\in I} \alpha_i = \sum_{i\in I} \alpha_i \int_{\R^d}
   e_i(t)\conj{e_i(t)}\,\rho(dt) = \int_{\R^d} \covG(t,t)\,\rho(dt)<\infty,
 \end{equation*}
which implies that $T_\covG$ is a nuclear -- or trace-class -- operator (see,
e.g., \cite{ReedSimon1980}).

Define $\edag{e_i}(t):=\conj{e_i(-t)}$, $t\in\R^d$, $i\in I$. Symmetry of
$\rho$ and hermiticity of $\pG$ allow to deduce that with $(e_i)_{i\in I}$
also $(\edag{e_i})_{i\in I}$ is an orthonormal system of $\range T_\covG$. In
particular, we find
\begin{equation*}
  T_\covG \edag{e_i}  = \alpha_i\edag{e_i},\qquad i\in I.
\end{equation*}

The following observation is a special case of \cite[Remark
2.9]{Lyons2013}. 
\begin{rem}[Eigenvalues for product
  structures]\label{rem:operator:product}
In the situation of matching product structures of the measure $\rho$ and the
kernel $\covG$, i.e., $\rho=\bigotimes_{j=1}^n \rho_j$ and $\covG(s,t) =
\prod_{j=1}^n \covG_j(s_j,t_j)$ on $\R^d=\R^{d_1}\times\dotsb\times\R^{d_n}$,
the eigenvalues and eigenvectors of the operator $T_\covG$ arise as products
of the eigenvalues and eigenvectors of the operators $T_{\covG_j}$, $1\leq j\leq
n$, respectively. In fact, from the
orthogonal eigensystems $(e_i^{(j)})_{i\in I_j}$ of $T_{\covG_j}$ in
$L^2(\rho_j)$ associated to the eigenvalues $\alpha_i^{(j)}>0$, $i\in I_j$,
$1\leq j\leq n$, we obtain by setting for
each multi-index $\mathbf{i}\in I_1\times\dotsb\times I_n$
\begin{equation*}
  \alpha_{\mathbf{i}} = \prod_{j=1}^n \alpha_{i_j}^{(j)},\qquad
  e_{\mathbf{i}}(t) = \prod_{j=1}^n e_{i_j}^{(j)}(t_j),\quad
  t=(t_j)_{j=1,\dotsc,n}\in\R^d,
\end{equation*}
an orthonormal eigensystem $(e_{\mathbf{i}})_{\mathbf{i}\in I_1\times\dotsb
 \times I_n}$ of $T_\covG$ associated to the eigenvalues $\alpha_{\mathbf{i}}>0$,
$\mathbf{i}\in I_1\times\dotsb\times I_n$.
\end{rem}

A first direct consequence of Mercer's representation is summarized in the
following lemma which gives a generalization of the Karhunen-Lo\`eve
decomposition and is typically only stated for compact domains, see, e.g.,
\cite[Theorem 2.5.5]{Sas2013}. The result is (partly) hidden in the proof
of Theorem~2 in \cite{FanMichPeneSalo2017}.

\begin{lemma}[Karhunen-Lo\`eve representation of the process]\label{lem:repG}
Let $\pG$ be a hermitian, centered, measurable second-order random
field such that its covariance kernel $\covG$ satisfies
\eqref{eq:integrability}. Let $I\subseteq\N$ and $(e_i)_{i\in I}$ be as in
Mercer's representation \eqref{eq:mercer} of $\covG$. Then $\pG\in L^2(\rho)$
almost surely,
\begin{equation}\label{eq:repG}
   \pG = \sum_{i\in I} \innerp{\pG}{e_i}e_i\qquad \Prob\text{-a.s. in }
   L^2(\rho),
\end{equation}
and $\innerp{\pG}{e_i}$, $i\in I$, are mutually uncorrelated, centered,
complex-valued random variables. If $\pG$ is, additionally, complex
Gaussian then the random variables $\innerp{\pG}{e_i}$, $i\in I$, are
independent and complex normally distributed.
\end{lemma}
\begin{proof}
We know from Mercer's representation \eqref{eq:mercer} that $(e_i)_{i\in I}$
is an orthonormal basis of $\range T_\covG$. Thus, for the representation
\eqref{eq:repG} to hold it
remains to show that $\pG$ is almost surely contained in $\range
T_\covG\subseteq L^2(\rho)$. Tonelli's theorem allows to deduce from
\eqref{eq:integrability}
\begin{equation*}
  \Expect[]{\norm{\pG}_\rho^2} = \Expect[Big]{\int_{\R^d}
    \abs{\pG(t)}^2\rho(dt)} =
  \int_{\R^d} \Expect[]{\abs{\pG(t)}^2}\rho(dt) = \int_{\R^d}
  \covG(t,t)\,\rho(dt) < \infty,
\end{equation*}
i.e., $\pG\in L^2(\rho)$. To show that $\pG\in\range T_\covG$, we choose
$u\in\ker T_\covG$. Applying Fubini's theorem\footnote{%
$\Expect{\iint \abs{\pG(s)\conj{u(s)}\conj{\pG(t)}u(t)}
  \rho(ds)\rho(dt)} = \Expect{\Bigl(\int
  \abs{\pG(s)\conj{u(s)}}\rho(ds)\Bigr)^2} \leq
  \Expect{\norm{\pG}^2_\rho\cdot\norm{u}^2_\rho}<\infty$}
we find
\begin{align*}
 \Expect{\abs{\innerp{\pG}{u}}^2} &= \Expect[bigg]{\int_{\R^d} \int_{\R^d}
    \pG(s)\conj{u(s)}\cdot\conj{\pG(t)}u(t)\,\rho(ds)\,\rho(dt)} \\ 
  &= \int_{\R^d} \int_{\R^d}
  \Expect{\pG(s)\conj{\pG(t)}}u(t)\,\rho(dt)\conj{u(s)}\,\rho(ds) =
  \innerp{T_\covG u}{u} = 0.
\end{align*}
Since $u\in\ker T_\covG$ was arbitrary, this implies $\pG\in\range
T_\covG$. Using that $(e_i)_{i\in I}$ is an orthonormal basis of $\range
T_\covG$ the representation \eqref{eq:repG} follows. Since $\pG$ is
centered, $\Expect{\innerp{\pG}{e_i}}=0$, $i\in I$. Using Fubini's
theorem\footnote{$\Expect{\iint \abs{\G(s)\conj{e_i(s)}\conj{\G(t)}e_j(t)}
    \rho(ds)\rho(dt)}  \leq \Expect{\norm{\G}^2_\rho}<\infty$} we find 
\begin{align}\label{eq:cov:inner}
\cov{\innerp{\pG}{e_i}}{\innerp{\pG}{e_j}} &=
  \Expect{\innerp{\G}{e_i}\cdot\conj{\innerp{\G}{e_j}}} \notag\\ &=
  \Expect{\int_{\R^d}\int_{\R^d}
    \G(s)\conj{e_i(s)}\cdot\conj{\G(t)}e_j(t)\,\rho(ds)\,\rho(dt)}\notag\\
  &= \int_{\R^d} \int_{\R^d}
  \Expect{\G(s)\conj{\G(t)}}e_j(t)\conj{e_i(s)}\,\rho(ds)\,\rho(dt)\notag\\
  &= \innerp{T_\covG e_j}{e_i} = \alpha_j\kronecker{i}{j},  
\end{align}
where $\kronecker{i}{j}$ denotes the Kronecker symbol. Thus, the
$\innerp{\pG}{e_i}$ are uncorrelated.

For the complex Gaussian case note that $\innerp{\pG}{e_i}$ are (as values of
linear functionals of a complex Gaussian random field) jointly complex
normally distributed random variables and, therefore, the uncorrelation
implies their independence.
\end{proof}

Although the random variables $\innerp{\pG}{e_i}$, $i\in I$, in the above
representation \eqref{eq:repG} are independent, we have, in general, no
information on the (in)dependence of $\Re\innerp{\pG}{e_i}$ and
$\Im\innerp{\pG}{e_j}$, $i,j\in I$, as the following observation shows:
Analogously to \eqref{eq:cov:inner}, we obtain for the
pseudo-covariance
\begin{align}\label{eq:pseudocov:inner}
  \Expect{\innerp{\G}{e_i}\cdot\innerp{\G}{e_j}} &= \Expect{\int_{\R^d}
    \int_{\R^d}
    \G(s)\conj{e_i(s)}\cdot \G(t)\conj{e_j(t)}\,\rho(ds)\,\rho(dt)}\notag\\
  &= \int_{\R^d}\int_{\R^d}
  \Expect{\G(s)\G(t)}\conj{e_j(t)}
  \conj{e_i(s)}\,\rho(ds)\,\rho(dt)\notag\\
  &= \int_{\R^d}\int_{\R^d}
  \covG(s,t)\conj{e_j(-t)}\cdot\conj{e_i(s)}\,\rho(ds)\,\rho(dt)\notag\\
  &= \innerp{T_\covG\edag{e_j}}{e_i} = \alpha_j  \innerp{\edag{e_j}}{e_i},
\end{align}
where $\edag{e_j}(t) = \conj{e_j(-t)}$. This implies
\begin{equation*}
  \cov{\Re\innerp{\pG}{e_i}}{\Im\innerp{\pG}{e_i}} =
  \frac{\alpha_i}{2}\Im\innerp{\edag{e_i}}{e_i},\quad i\in I.
\end{equation*}
Nevertheless representation \eqref{eq:repG} allows us to
show our main result of this section, linking Gaussian random fields to Gaussian
quadratic forms. The following result can be found -- under stronger
assumptions -- hidden in the proof of Theorem~2 in \cite{FanMichPeneSalo2017}
or could also be deduced from results for Gaussian measures in
infinite-dimensional spaces (see, e.g., \cite[Comments on Section
1.2]{Kuo1975}). Our proof is direct and relies on the already
established Karhunen-Lo\`eve decomposition of the Gaussian random field given
in Lemma~\ref{lem:repG}.

\begin{theorem}[Relation to positive Gaussian quadratic forms;
  cf.~{\cite[p.~206]{FanMichPeneSalo2017}}]
Let $\pG$ be a hermitian, centered, measurable complex Gaussian
random field such that its covariance kernel $\covG$ satisfies
\eqref{eq:integrability}. Let $I\subseteq\N$ and $(\alpha_i)_{i\in I}$ be as
in Mercer's representation \eqref{eq:mercer} of $\covG$. Then 
\begin{equation*}
    \norm{\pG}_\rho^2 = \sum_{i\in I} \alpha_i Z_i^2 \quad \Prob\text{-a.s.},
\end{equation*}
where $(Z_i)_{i\in I}$ is a sequence of independent, standard normally
distributed random variables.\label{thm:repNormGauss}
\end{theorem}
\begin{proof}
Since $T_\covG$ is compact, all eigenspaces associated to
non-vanishing eigenvalues are finite-dimensional. Let
$(\beta_i)_{i\in I'}$ be the mutually different positive eigenvalues of
$T_\covG$, i.e., $\{\beta_i:\ i\in I'\}=\{\alpha_i:\ i\in I\}$,  with
$\beta_i\neq\beta_j$ for $i\neq j$ and $I'\subseteq
I$. Set $J_i:=\{k\in I:\ \alpha_k=\beta_i\}$, $i\in I'$. Thus the number of
elements in $J_i$ is the multiplicity of the eigenvalue $\beta_i$ and the
$J_i$ are finite, mutually disjoint sets and by Lemma~\ref{lem:repG} we
directly obtain the representation
\begin{equation}\label{eq:decomp:norm}
    \norm{\pG}_\rho^2 = \sum_{i\in I} \abs{\innerp{\pG}{e_i}}^2 = \sum_{i\in I'}
  \sum_{k\in J_i} \abs{\innerp{\pG}{e_k}}^2.
\end{equation}
Consider the $\C^{\abs{J_i}}$-valued random vector $U_i =
\frac{1}{\sqrt{\beta_i}}\bigl(\innerp{\pG}{e_k}\bigr)_{k\in J_i}$,
$i\in I'$, collecting all $e_k$ belonging to the eigenspace associated with
$\beta_i$. Since $T_K$ is self-adjoint, all eigenspaces are mutually 
orthogonal. Therefore our considerations on the covariance
\eqref{eq:cov:inner} and on the pseudo-covariance \eqref{eq:pseudocov:inner}
yield that $U_i$, $i\in I'$, are independent complex normally distributed
random variables and their (real) covariance matrix $\covUV$ has
the symmetric block structure
\begin{equation*}
 \covUV =  \frac{1}{2}\begin{pmatrix} \eye_{\abs{J_i}} +
    \bigl(\Re\innerp[]{\edag{e_k}}{e_j}\bigr)_{j,k\in J_i} &
    \bigl(\Im\innerp[]{\edag{e_k}}{e_j}\bigr)_{j,k\in J_i} \\
     \bigl(\Im\innerp[]{\edag{e_k}}{e_j}\bigr)_{j,k\in J_i} & \eye_{\abs{J_i}} -
    \bigl(\Re\innerp[]{\edag{e_k}}{e_j}\bigr)_{j,k\in J_i}.
  \end{pmatrix} =: \frac{1}{2}\bigl(\eye_{2\abs{J_i}} + \matB\bigr),
\end{equation*}
where $\eye_{M}$ denotes the identity matrix in $\R^{M\times M}$ and
\begin{equation*}
  \matB = \begin{pmatrix} \matB_{\text{R}} & \matB_{\text{I}} \\
    \matB_{\text{I}} & -\matB_{\text{R}} \end{pmatrix},\
  \matB_{\text{R}} = \bigl(\Re\innerp[]{\edag{e_k}}{e_j}\bigr)_{j,k\in
    J_i},\ \matB_{\text{I}} =
  \bigl(\Im\innerp[]{\edag{e_k}}{e_j}\bigr)_{j,k\in J_i}.
\end{equation*}
We show that $\matB^2=\eye_{2\abs{J_i}}$, which yields $\covUV^2 = \frac{1}{4}
\bigl(\eye_{2\abs{J_i}} + 2\matB + \matB^2\bigr) = \covUV$.
Using symmetry of $\matB$ we find
\begin{equation*}
  \matB^2 =
  \begin{pmatrix}
    \matB_{\text{R}}^2 + \matB_{\text{I}}^2 &
    \matB_{\text{R}}\matB_{\text{I}}-\matB_{\text{I}}\matB_{\text{R}} \\
    \matB_{\text{I}}\matB_{\text{R}}-\matB_{\text{R}}\matB_{\text{I}}  &
    \matB_{\text{R}}^2 + \matB_{\text{I}}^2
  \end{pmatrix}
  =
  \begin{pmatrix}
    \matB_{\text{R}}^2 + \matB_{\text{I}}^2 & \mat{0} \\
    \mat{0} &  \matB_{\text{R}}^2 + \matB_{\text{I}}^2
  \end{pmatrix}.
\end{equation*}
Using the elementary equation $\Re z\Re w + \Im z\Im w = \Re z\conj{w}$ we can
compute for the block entries on the diagonal
\begin{align*}
  (\matB_{\text{R}}^2 + \matB_{\text{I}}^2)_{j,k} &= \sum_{m\in J_i} \bigl(
  \Re\innerp[]{\edag{e_m}}{e_j}\Re\innerp[]{\edag{e_k}}{e_m}  +
  \Im\innerp[]{\edag{e_m}}{e_j}\Im\innerp[]{\edag{e_k}}{e_m}\bigr) \\ 
  &= \Re \sum_{m\in J_i}
  \innerp[]{\edag{e_m}}{e_j}\innerp[]{e_m}{\edag{e_k}} 
  = \Re \innerp{\sum_{m\in
      J_i}\innerp[]{e_m}{\edag{e_k}}\edag{e_m}}{e_j}.
\end{align*}
Symmetry of $\rho$ yields $\innerp{e_m}{\edag{e_k}}=\innerp{e_k}{\edag{e_m}}$
and since $(\edag{e_m})_{m\in J_i}$ is an orthonormal basis of the associated
eigenspace we find
\begin{equation*}
  \sum_{m\in J_i} \innerp{e_k}{\edag{e_m}}\edag{e_m} = e_k,
\end{equation*}
which implies
\begin{equation*}
  (\matB_{\text{R}}^2 + \matB_{\text{I}}^2)_{j,k} = \Re \innerp{e_k}{e_j} =
  \delta_{j,k},\qquad j,k\in J_i.
\end{equation*}
As idempotent and positive semi-definite matrix, $\covUV$ can
have only the eigenvalues $0$ and $1$ and from $\tr(\covUV)=\abs{J_i}$ we
find that $\covUV$ has rank $\abs{J_i}$. An application of \cite[Theorem
2]{Duchetal2016} yields for each $i\in I'$, 
\begin{equation}\label{eq:decomp:uncor}
  \sum_{k\in J_i} \abs{\innerp{\pG}{e_k}}^2 = \beta_i\norm{U_i}^2 = \beta_i
  \sum_{\ell=1}^{\abs{J_i}} W_\ell^2,
\end{equation}
where $W_1,\dotsc,W_{\abs{J_i}}$ are independent standard normally 
distributed random variables. Using an enumeration
$\{k_1,\dotsc,k_{\abs{J_i}}\}$ of $J_i$ and setting $Z_{k_\ell}:=W_\ell$,
$1\leq \ell\leq \abs{J_i}$, we find by plugging \eqref{eq:decomp:uncor} into
\eqref{eq:decomp:norm} and rearranging the summation 
\begin{align*}
 \norm{\pG}_\rho^2&= \sum_{i\in I'} \sum_{k\in J_i} \abs{\innerp{\pG}{e_k}}^2
 = \sum_{i\in I'} \beta_i \sum_{\ell=1}^{\abs{J_i}} W_\ell^2 \\
&= \sum_{i\in I'} \beta_i\sum_{k\in J_i} Z_k^2 =
  \sum_{i\in I'}  \sum_{k\in J_i} \alpha_k Z_k^2= \sum_{i\in I} \alpha_i Z_i^2.
\end{align*}
Since the $W_\ell$ -- and thus $Z_i$ -- arise as
linear transformation of the independent Gaussian vectors $U_i$, they are
independent.
\end{proof}

\begin{rem}[Gaussian quadratic forms are squared norms]
All positive Gaussian quadratic forms (with summable
coefficients) arise as in Theorem~\ref{thm:repNormGauss}, i.e., as squared
norm of a Gaussian random field: In fact, let $Q=\sum_{i\in\N} \alpha_i Z_i^2$
be a Gaussian quadratic form with $\alpha_i\geq 0$, $i\in\N$. Let $\rho$
be a nonnegative, symmetric, $\sigma$-finite Borel measure on $\R^d$ and
$(e_i)_{i\in\N}$ be an orthonormal system in $L^2(\rho)$ such that
$e_i(-t)=\conj{e_i(t)}$, $t\in\R^d$. Then
\begin{equation*}
\pG(t) = \sum_{i\in\N} \sqrt{\alpha_i}e_i(t)Z_i,\quad t\in\R^d,
\end{equation*}
defines a centered, measurable complex Gaussian random field and we find
\begin{equation*}
  \norm{\pG}_\rho^2 = \int_{\R^d} \abs{\pG(t)}^2 \rho(dt) = \sum_{i,j\in \N}
  \sqrt{\alpha_i}\sqrt{\alpha_j} Z_iZ_j \innerp{e_i}{e_j}_\rho = \sum_{i\in
    \N} \alpha_i Z_i^2 = Q.
\end{equation*}
Note that in this construction, neither $\rho$ nor $(e_i)_{i\in\N}$ are
uniquely determined.
\end{rem}

The last result of this section gives a representation of
the moments of the above Gaussian quadratic form (as studied in
Lemma~\ref{lem:Q4}) in terms of the covariance kernel $\covG$ of the random
field $\pG$. To this end, we denote for the sequence $(\alpha_i)_{i\in I}$
from Mercer's representation \eqref{eq:mercer} and $k\in\N$
\begin{equation}\label{eq:defmu}
  \sumalpha{k} = \sumalpha{k}\bigl((\alpha_i)_{i\in I}\bigl) := \sum_{i\in I}
  \alpha_i^k.
\end{equation}
Using standard operator theory it is obvious that $\alpha_i^k$, $i\in I$, are
the (positive) eigenvalues of the operator $T_\covG^k$ and that $T_\covG^k$ is
an integral operator with kernel given by iterated integration of $\covG$. In
particular, we find the following representation:
\begin{lemma}\label{lem:repKernelp}
Let $\covG$ be the covariance kernel of a measurable second-order
random field satisfying \eqref{eq:integrability}. For $k\in\N$ and $u\in
L^2(\rho)$ and $s\in\R^d$
\begin{equation*}
  (T_\covG^ku)(s) = \int_{(\R^d)^k} \covG(s,\itint{k})\prod_{j=1}^{k-1}
  \covG(\itint{k+1-j},\itint{k-j})u(\itint{1})\,\rho^{\otimes
    k}(d\itint{1},\dotsc,d\itint{k}).
\end{equation*}
\end{lemma}
\begin{proof}
We proceed by induction on $k$. The statement for $k=1$ reduces to the
definition of $T_\covG$ and there is nothing to show. Assume the assertion is
valid for some $k\in\N$. Then we find for $u\in L^2(\rho)$ and $s\in\R^d$
\begin{align*}
  (T_\covG^{k+1}u)&(s) = T_\covG\bigl(T_\covG^k u\bigr)(s) = \int_{\R^d}
  \covG(s,\itint{k+1})\bigl(T_\covG^k u\bigr)(\itint{k+1})\,\rho(d\itint{k+1}) \\
  &\overset{(*)}{=}   \int_{\R^d}  \covG(s,\itint{k+1}) \int_{(\R^d)^k}
  \covG(\itint{k+1},\itint{k})\prod_{j=1}^{k-1}
  \covG(\itint{k+1-j},\itint{k-j})\cdot  u(\itint{1})\\ 
  &\hspace{6.2cm} \,\rho^{\otimes
    k}(d\itint{1},\dotsc,d\itint{k})\,\rho(d \itint{k+1}) \\
  &=   \int_{(\R^d)^{k+1}}
  \covG(s,\itint{k+1})\prod_{j=1}^k
  \covG(\itint{k+2-j},\itint{k+1-j})u(\itint{1}) \\
  &\hspace{6.7cm}\,\rho^{\otimes(k+1)}(d\itint{1},\dotsc,d\itint{k+1}),
\end{align*}
where the induction assumption was used in step $(*)$. Hence, the assertion
for $k+1$ is proven.
\end{proof}
This representation allows us to connect the moments of the quadratic
form $\norm{\G}_\rho^2$ with the covariance kernel $K$.
\begin{proposition}[Moments as kernel
  integrals]\label{prop:pnormintegral}
Let $\covG$ be the covariance kernel of a hermitian, centered,
measurable second-order random field satisfying \eqref{eq:integrability}. Let
$I\subseteq\N$ and $(\alpha_i)_{i\in I}$ be as in Mercer's representation
\eqref{eq:mercer} of $\covG$. Then for $k\in\N$
\begin{equation*}
  \sumalpha{k} =\sum_{i\in I} \alpha_i^k 
  = \int_{(\R^d)^k} \prod_{j=1}^{k-1}
  \covG(\itint{k-j+1},\itint{k-j})\cdot
  \covG(\itint{1},\itint{k})\,\rho^{\otimes k}(d\itint{1},\dotsc,d\itint{k}).
\end{equation*}
Moreover, $\sumalpha{1}<\infty$ implies $\sumalpha{k}<\infty$ for
all $k\in \N$.
\end{proposition}
\begin{proof}
For $k=1$, orthogonality and a direct application of Mercer's representation
\eqref{eq:mercer} yield
\begin{equation*}
\sumalpha{1} = \sum_{i\in I} \alpha_i\norm{e_i}^2_\rho = \int_{\R^d}
  \sum_{i\in I} \alpha_i e_i(\itint{1})\conj{e_i(\itint{1})}\rho(d\itint{1}) =
  \int_{\R^d} \covG(\itint{1},\itint{1})\rho(d\itint{1}).
\end{equation*}
Now assume $k>1$. We obtain
\begin{align*}
  \sumalpha{k} &= \sum_{i\in I} \alpha_i^k = \sum_{i\in I}
  \alpha_i\innerp{T_\covG^{k-1}e_i}{e_i}  = \int_{\R^d} \sum_{i\in I} \alpha_i
  \bigl(T_\covG^{k-1}e_i\bigr)(\itint{k})
 \conj{e_i(\itint{k})}\,\rho(d\itint{k}) \\
  &\overset{(*)}{=} \int_{\R^d} \sum_{i\in I} \alpha_i \int_{(\R^d)^{k-1}}
  \covG(\itint{k},\itint{k-1})\prod_{j=1}^{k-2}
 \covG(\itint{k-j},\itint{k-1-j})\cdot e_i(\itint{1})  \\
&\hspace{4.5cm} \cdot
\conj{e_i(\itint{k})}\,\rho^{\otimes(k-1)}(d\itint{1},\dotsc,d\itint{k-1})
\,\rho(d\itint{k})\\
&= \int_{(\R^d)^k}   \covG(\itint{k},\itint{k-1})\prod_{j=1}^{k-2}
  \covG(\itint{k-j},\itint{k-1-j})\cdot
\sum_{i\in I}
\alpha_ie_i(\itint{1})\conj{e_i(\itint{k})}\\
&\hspace{8cm}\rho^{\otimes
  k}(d\itint{1},\dotsc,d\itint{k}) \\
&= \int_{(\R^d)^k}   \covG(\itint{k},\itint{k-1})\prod_{j=1}^{k-2}
\covG(\itint{k-j},\itint{k-1-j})\cdot\covG(\itint{1},\itint{k})\\
&\hspace{8cm}\rho^{\otimes
    k}(d\itint{1},\dotsc,d\itint{k}) \\
 &= \int_{(\R^d)^k}  \prod_{j=1}^{k-1}
  \covG(\itint{k-j+1},\itint{k-j})\cdot\covG(\itint{1},\itint{k})\,\rho^{\otimes
    k}(d\itint{1},\dotsc,d\itint{k}),
\end{align*}
where we have used Lemma~\ref{lem:repKernelp} in step $(*)$ and Mercer's
representation \eqref{eq:mercer} in the penultimate line. The last statement
follows as in Lemma \ref{lem:Q4}.
\end{proof}

\begin{rem}\label{rem:mu}
Since for the applications in \eqref{eq:qf-estimates} and in the upcoming
Section~\ref{sec:distance-multivariance} the representations of
$\sumalpha{1},\dotsc,\sumalpha{4}$ are of special interest, we rephrase
the integral representation of Proposition~\ref{prop:pnormintegral} in a more
accessible form:
\begin{align*}
  \sumalpha{1} &= \int \covG(t,t)\,\rho(dt),\\
  \sumalpha{2} &= \iint \covG(s,t)\covG(t,s)\,\rho(dt)\,\rho(ds),\\
  \sumalpha{3} &= \iiint
  \covG(s,t)\covG(t,r)\covG(r,s)\,\rho(dt)\,\rho(ds)\,\rho(dr),\\ 
  \sumalpha{4} &= \iiiint  \covG(s,t)\covG(t,r)\covG(r,u)\covG(u,s)
  \,\rho(du)\,\rho(dt)\,\rho(ds)\,\rho(dr).
\end{align*}
\end{rem}


%% file: distance-multivariance.tex
\label{sec:multivariance}
The basic setting for the detection of (in)dependence using distance
multivariance is as follows, cf. \cite{Boet2019,BoetKellSchi2018a}.

Let $X_i$, $1\leq i\leq n$, be random variables with values in $\R^{d_i}$ and
set $\vect{X}:=(X_1,\dots,X_n)$. Independent copies of $\vect{X}$ are denoted
by $\vect{X}^{(l)}$ for $1\leq l\leq N$. Samples of $X$ are denoted by small
$x$, e.g., $\vect{x}^{(l)}=(x_1^{(l)},\dots x_n^{(l)})$ is a sample of
$\vect{X}^{(l)}$.

Let $\rho_i$ be symmetric measures with full support on $\R^{d_i}$ such that
$\int 1\land \abs{t_i}^2 \,\rho_i(dt_i) < \infty$ and $\rho_S:= \otimes_{i\in
  S} \rho_i$ for $S\subset\{1,\dots,n\}$. Moreover, $\rho :=
\rho_{\{1,\dots,n\}}$ and for $t = (t_1,\dots,t_n) \in \R^{\sum_{i =1}^n d_i}$
set $t_S:= (t_i)_{i\in S}$. The $L^2$-norm with respect to $\rho_S$ is denoted
by $\norm{.}_{\rho_S}$, e.g., $\|g\|_{\rho_S} = \sqrt{\int |g(t_S)|^2\,
  \rho(dt_S)}.$ The characteristic function of $X_i$ is denoted by
$f_i(t_i):=f_{X_i}(t_i):=\E(\ee^{\ii X_i \cdot t_i})$.

Then the \textbf{distance multivariance} $M_{\rho_S}$ of $X_1,\dots,X_n$ is
defined by
\begin{align}
M_{\rho_S}(X_1,\dots,X_n)&:=\|Z_S\|_{\rho_S}\notag\\
\shortintertext{with}\label{eq:Z}
 Z_S(t_S):=Z_S(X_1,\dots,X_n;t_S) &:= \E\left(\prod_{i\in S} \left(\ee^{\ii
                                        X_i\cdot
                                        t_i}-f_{X_i}(t_i)\right)\right)
\end{align}
and \textbf{sample distance multivariance} $\hN M_{\rho_S}$ of
$\vect{x}^{(1)},\dots,\vect{x}^{(N)}$ is defined by
\begin{align}
 \hN M_{\rho_S}(\vect{x}^{(1)},\dots,\vect{x}^{(N)})&:=\|\hN Z_S\|_{\rho_S}\notag\\
\shortintertext{with} 
 \label{eq:ZN} \hN Z_S(t_S):=\hN Z_S(\vect{x}^{(1)},\dots,\vect{x}^{(N)};t_S) &:=
 \frac{1}{N} \sum_{l=1}^N \prod_{i\in S} \left(\ee^{\ii x_i^{(l)}\cdot t_i} -
   \frac{1}{N} \sum_{k=1}^N \ee^{\ii x_i^{(k)}\cdot t_i} \right).
 \end{align}
Note that the latter definition is different but equivalent to the definition
given in \cite[Eq. (4.1)]{BoetKellSchi2018a},
cf. \cite[Eq. (S.7)]{BoetKellSchi2018a-supp}. Using \eqref{eq:ZN} as definition
makes it obvious that it is the natural choice of an empirical approximation
to \eqref{eq:Z}. For a proof of the strong consistency of this estimator see
\cite[Thm. 4.5]{BoetKellSchi2018a}. Also note that the notation is slightly
different to \cite{BoetKellSchi2018a} and \cite{Boet2019}, since here we
always write $(X_1,\dots,X_n)$ although $S$ might be a proper subset of
$\{1,\dots,n\}$.
This helps to keep the notation for total multivariance and $m$-multivariance
and their estimators unified. Based on the above define
\textbf{total (distance) multivariance:}
\begin{align}
\label{def:total}
  \overline{M}_\rho(X_1,\dots,X_n) &:=  \sum_{\substack{2\leq |S|\leq n\\ S
      \subset\{1,\dots,n\}}} M_{\rho_S}(X_1,\dots,X_n),
  \shortintertext{and \textbf{$\bm{m}$-(distance) multivariance}:}
\label{def:m}
  {M}_{m,\rho}(X_1,\dots,X_n) &:=  \sum_{\substack{ |S| = m\\ S
  \subset\{1,\dots,n\}}} M_{\rho_S}(X_1,\dots,X_n)
\end{align}
for $m \in \{2,\dots,n\}.$ The next result is fundamental to the theory of
distance multivariance.

\begin{theorem}[Characterization of independence; \protect{\cite[Theorem
    3.4]{BoetKellSchi2018a} and \cite[Proposition 5.1]{Boet2019}}] For random
  variables $X_1,\dots,X_n$ the following characterizations of their
  (in)dependence hold
\begin{eqnarray*}
 \overline{M}_\rho(X_1,\dots,X_n)=  0 &\Leftrightarrow & X_i \text{
                                                          independent},\\
M_\rho(X_1,\dots,X_n) =  0 \text{ and } X_i \text{ $(n-1)$-independent}
                                       &\Leftrightarrow & X_i \text{
                                                          independent},\\
{M}_{m,\rho}(X_1,\dots,X_n)=  0 \text{ and } X_i \text{ $(m-1)$-independent}
                                       &\Leftrightarrow & X_i \text{
                                                          $m$-independent},
\end{eqnarray*}
where the random variables are called $\bm{k}$\textbf{-independent} if every
subfamily of $\{X_1,\dots,X_n\}$ with $k$ elements is independent.
\end{theorem}

Thus distance multivariance can be used to characterize (in)dependence. Now,
the beauty (practical utility) of this approach stems from the fact that for a
sample $\vect{x}^{(1)},\dots,\vect{x}^{(N)}$ of $\vect{X}$ the sample
distance multivariance has the following computationally feasible representation
\begin{gather} 
  \notag\hN M_{\rho_S}^2(\vect{x}^{(1)},\dots,\vect{x}^{(N)}) = \frac{1}{N^2}
  \sum_{j,k=1}^N \prod_{i\in S}(A_i)_{j,k}\\
  \label{def:sm-cdms}\text{ with }A_i := - C B_i C,\  C := I-{\textstyle
    \frac{1}{N}} \bm{1},\ B_i:
  = \left(\psi_i\left(x_i^{(j)}-x_i^{(k)}\right)\right)_{j,k=1,\dots,N},\\
  \notag\text{ where } \psi_i(x_i) := \int_{\R^{d_i}} 1- \cos(x_i\cdot t_i)\,
  \rho_i(dt_i),
\end{gather}
i.e., it is the sum of all entries of the Hadamard product of the $A_i$, which
are the doubly centered distance matrices for the distances induced by the
continuous negative definite functions corresponding to the measures
$\rho_i$. For further details see \cite{Boet2019, BoetKellSchi2018}. Moreover
also for sample total- and sample $m$-multivariance computationally feasible
representations are available. They are much faster than just summing up the
corresponding sample multivariance, as \eqref{def:total} and \eqref{def:m}
would suggest.

\begin{remark}[Moment conditions] \label{rem:momcond}
Note that for statistical tests based on distance multivariance (see
\cite[Theorems 2.5, 5.2 and 8.3, Remark 2.6]{Boet2019}) one of the following
integrability conditions is required for all $1\leq i\leq n$
\begin{gather}
  \label{eq:con-convergence2}  \Expect{\psi_i(X_i)^2}<\infty
  \shortintertext{or}
  \label{eq:con-convergence}
  \Expect{\psi_i(X_i)}<\infty \text{ and }
  \Expect{(\log(1+|X_i|^2))^{1+\varepsilon}} < \infty \text{ for some
  }\varepsilon>0.
\end{gather}
In this paper we only consider the condition \eqref{eq:con-convergence} -- but
it seems reasonable that the main results remain valid also if only
\eqref{eq:con-convergence2} holds.

A further important condition is the \emph{joint $\psi$-moment condition}, 
\begin{equation} \label{eq:joint-moment}
  \Expect{\prod_{i \in S} \psi_i(X_i)} < \infty \text{ for all }S \subset
  \{1,\dots,n\},
\end{equation}
which is required for the finiteness of expectation representations of
multivariance, \cite[Equation (5) and Section 8.2]{Boet2019}.
\end{remark} 

For testing independence the distribution of $\hN
M_{\rho_S}(\vect{X}^{(1)},\dots,\vect{X}^{(N)})$ (or of a transformation of
it) under the hypothesis of independence (i.e., $X_1,\dots,X_n$ are
independent) has to be known or estimated, compare with \cite[Sec.\
4]{Boet2019}. One approach is to use the conservative estimate \eqref{eq:qfSB}
or our extension \eqref{eq:tail}. We summarize the methods in Section
\ref{sec:tests}.

In some settings it is useful to normalize the
estimators such that the limit has unit expectation, this is called
\textbf{normalized sample distance multivariance}. It is denoted by $\Mskript$
instead of $M$ and it is obtained by replacing $B_i$ in \eqref{def:sm-cdms} by 
\begin{equation}\label{eq:Bnormalized}
\Bskript_i:= \frac{1}{\frac{1}{N^2}\sum_{j,k=1}^N \psi_i(x_i^{(j)}-x_i^{(k)})} B_i.
\end{equation}
Furthermore, for the corresponding \textbf{normalized total multivariance} and
\textbf{normalized $\bm{m}$-multivariance} the estimator is additionally also
scaled by $(2^n-n-1)^{-1}$ and $\binom{n}{m}^{-1}$, respectively. For more
details see \cite[Sec.\ 2 and 5]{Boet2019}.

Finally, note that distance multivariance and normalized distance multivariance
are always translation invariant. Moreover, normalized distance multivariance
is scale invariant if the measures $\rho_i$ are such that the functions
$\psi_i$ given in \eqref{def:sm-cdms} are of the form
$\psi_i(x_i)=|x_i|^{\beta_i}$ with $\beta_i\in (0,2]$, see \cite[Prop.\
2.4]{Boet2019}.

\subsection{Limit of the raw estimators -- $\sqrt{N} \cdot \hN Z_S$}

The following result simplifies the original representation obtained in
\cite[Eq. (S.15)]{BoetKellSchi2018a-supp} of the limit of the raw (i.e.,
without taking the norm) estimator $\sqrt{N}\cdot \hN Z_S.$
Based on it the moments of the limit distribution can be calculated directly,
see Proposition \ref{pro:moments-ZNlimit}.

\begin{theorem} \label{thm:ZNBMlimit}
Let $X_1,\dots,X_n$ be independent such that
$\Expect{(\log(1+|X_i|^2))^{1+\varepsilon}}$ is finite for all $1\leq
i\leq n$ and some $\varepsilon>0$, 
then $\hN Z_S(t_S) = 0$ for $|S| = 1$, and for $|S|\geq 2$
\begin{equation} \label{eq:ZNBMlimit}
\sqrt{N}\cdot \hN Z_S(t_S) \xrightarrow[\quad N \to \infty \ ]{d}
\G_S(t_S): = \int \prod_{i\in S} \left(\ee^{\ii x_i\cdot
    t_i}-f_{X_i}(t_i)\right)\,d\W(\vect{F}_\vect{X}(x)),
\end{equation}
where $\W$ is a Brownian sheet indexed by $[0,1]^{n}$ and
$\vect{F}_\vect{X}$ is the vector of all distribution functions,
i.e., 
$\vect{F}_\vect{X}(x) = (F_{X_1}(x_1),\dots,F_{X_n}(x_n))$.
\end{theorem}
\begin{proof}
  The proof is similar to the beginning of the proof of \cite[Thm.\
  4.5]{BoetKellSchi2018a}.
  Let $X_1,\dots,X_n$ be independent and $\hN f_S$ denotes the empirical
  characteristic function of $x_S^{(1)},\dots,x_S^{(N)}$, where
  $\bm{x}^{(1)},\dots,\bm{x}^{(N)}$ is a sample of $\bm{X}$.

  First recall a basic identity for products of differences
  \begin{equation} \label{eq:prodid}
    \prod_{i\in S} (a_i-b_i) = \sum_{R\subset S} (-1)^{|S|-|R|}\left(\prod_{i\in
        R}a_i \cdot \prod_{j \in S \setminus R} b_j\right).
  \end{equation}
  Hence the independence of the $X_j$, $1\leq j\leq n$, implies
  \begin{equation*}
    \prod_{j\in S} (f_j(t_j)-\hN f_j(t_j)) = \sum_{R\subset S} (-1)^{|S|-|R|}
    \left(f_R(t_R) \cdot \prod_{j \in S \setminus R}\hN f_j(t_j)\right)
  \end{equation*}
  and \eqref{eq:ZN} becomes
  \begin{equation*}
    \begin{split}
      \hN Z_S(t_S)  =& \sum_{R\subset S} (-1)^{|S|-|R|} \left(\hN f_R(t_R) \cdot
        \prod_{j \in S \setminus R}\hN f_j(t_j)\right)\\
      =& \prod_{j\in S} (f_j(t_j)-\hN f_j(t_j))  \\
      & + \sum_{R\subset S} (-1)^{|S|-|R|} \left((\hN f_R(t_R) - f_R(t_R))\cdot
        \prod_{j \in S \setminus R}\hN f_j(t_j)\right).
    \end{split}
  \end{equation*}
  Recall that in general for a distribution function $G$ on $\R^n$ and a
  function $g$ on $\R^{|S|}$
  \begin{equation*}
    \int_{\R^{|S|}} g(x_S) \, dG_S(x_S) = \int_{\R^n} g(x_S) \, dG(x). 
  \end{equation*}
  Thus $\hN f_S(t_S) - f_S(t_S) = \int \ee^{\ii x_S\cdot t_S}\, d(\hN
  F_\vect{X}(x) - F_\vect{X}(x))$ and by the central limit theorem
  \begin{equation*}
    \sqrt{N}\left(\hN F_\vect{X}(x) -
      F_\vect{X}(x)\right)\xrightarrow[\quad N \to \infty \quad]{d} Z \quad
    \text{ where } \quad Z\sim N(0, F_\vect{X}(x)(1-F_\vect{X}(x)).
  \end{equation*}
  To avoid confusion, note that $F_\vect{X}$ is the common distribution
  function, and this is different to the vector of the marginal distribution
  functions $\vect{F}_\vect{X}$. Furthermore, the independence implies
  $F_\vect{X} = \prod_{i=1}^n F_{X_i}$. 
  
  Extending the convergence to the sample paths (cf. \cite[Equation
  (S.15)]{BoetKellSchi2018a-supp}) yields
  \begin{equation} \label{eq:ZNlimit}
    \sqrt{N} \cdot \hN Z_S(t_S) \xrightarrow[\quad N \to \infty \quad]{d}
    \sum_{R\subset S} (-1)^{|S|-|R|}\int \ee^{\ii x_R\cdot t_R}\,d\B(x) \prod_{j
      \in S \setminus R} f_{X_j}(t_j),
  \end{equation}
  where $\B$ is a Gaussian random field indexed by $\R^{d_1 + \dotsb + d_n}$ with
$\E(\B(t))= 0$ and
\begin{equation} \label{eq:bridgecov}
\E(\B(s) \B(t)) = \prod_{i=1}^n F_{X_i}(s_i\land t_i) - \prod_{i=1}^n
F_{X_i}(s_i) F_{X_i}(t_i).
\end{equation} 
Using \eqref{eq:prodid} again in \eqref{eq:ZNlimit} yields
\begin{equation*}
\sqrt{N}\cdot \hN Z_S(t_S) \xrightarrow[\quad N \to \infty \quad]{d}
\int \prod_{i\in S} \left(\ee^{\ii x_i\cdot t_i}-f_{X_i}(t_i)\right)\,d\B(x).
\end{equation*}
Note that for a Brownian bridge $\widetilde \B$ from 0 to 0 with
multi-dimensional index set $[0,1]^{n}$ one has $\E(\widetilde{\B}(t)) = 0$
and $\E(\widetilde{\B}(s) \widetilde{\B}(t)) = \prod_{i=1}^n (s_i \land t_i) -
\prod_{i=1}^n s_it_i$ and thus
\begin{equation*}
\B(\cdot) \stackrel{d}{=} \widetilde{\B}(\vect{F}_\vect{X}(\cdot)),
\end{equation*} 
since both sides are centered Gaussian random fields with identical covariance
structure, hereto note that $F_{X_i}(s_i) \land F_{X_i}(t_i) = F_{X_i}(s_i
\land t_i)$.

For clarification, note that one might be tempted to consider the alternative
$\widetilde \A(F_\vect{X})$ where $\widetilde \A$ is the classical Brownian
bridge, i.e., it has the one-dimensional index set $[0,1]$ and
$\E(\widetilde{\A}(s) \widetilde{\A}(s')) = (s \land s') - s s'$. But note
that in general $(\prod F_{X_i}(t_i))\land (\prod F_{X_i}(t_i')) \neq \prod (
F_{X_i}(t_i)) \land F_{X_i}(t_i'))$, and thus this would yield a different
covariance than required by \eqref{eq:bridgecov}.

To continue the proof, recall that for a Brownian sheet $\W$ indexed by
$[0,1]^{n}$ also
\begin{equation*}
\left(\widetilde{\B}(\tilde t)\right)_{\tilde t \in [0,1]^{n}} \stackrel{d}{=}
\left(\W(\tilde t) - \W(\vect{1}_n) \cdot \prod_{j=1}^{n}\tilde t_j\right)_{\tilde
  t \in [0,1]^{n}}
\end{equation*}
holds, where $\vect{1}_n = (1,\dots,1) \in \R^n$. Thus  
\begin{multline*}
\int \prod_{i\in S} \left(\ee^{\ii x_i\cdot t_i}-f_{X_i}(t_i)\right)\,d\B(x) = 
\int \prod_{i\in S} \left(\ee^{\ii x_i\cdot
    t_i}-f_{X_i}(t_i)\right)\,d\W(\vect{F}_\vect{X}(x)) \\  {} -
\W(\vect{1}_n) \int \prod_{i\in S} \left(\ee^{\ii x_i\cdot
    t_i}-f_{X_i}(t_i)\right)\,d\vect{F}_\vect{X}(x)
\end{multline*}
and 
\begin{equation*}
\int \prod_{i\in S} \left(\ee^{\ii x_i\cdot
    t_i}-f_{X_i}(t_i)\right)\,d\vect{F}_\vect{X}(x) = \prod_{i\in S} \int
\left(\ee^{\ii x_i\cdot t_i}-f_{X_i}(t_i)\right)\,dF_{X_i}(x_i) = 0
\end{equation*}
implies the result, \eqref{eq:ZNBMlimit}.
\end{proof}
As in \eqref{eq:ZNBMlimit} we set for $S \subset \{1,\dots,n\}$ with $|S|>1$ 
\begin{equation}\label{eq:ZNBMlimit2}
\G_S(t_S):= \int \prod_{i\in S} \left(\ee^{\ii x_i\cdot
    t_i}-f_{X_i}(t_i)\right)\,d\W(\vect{F}_\vect{X}(x)).
\end{equation}
The case $|S|=1$ is excluded, since in this case the limit would be trivial
and thus it differs from the function defined here.

Now we can calculate moments and properties of $\G_S$ explicitly.
\begin{proposition} \label{pro:moments-ZNlimit}
Let $S,R \subset \{1,\dots,n\}$ with $|S|,|R|>1$. Then $\G_S$ and $\G_R$ are
independent for $S \neq R$ and 
\begin{align}
  \Expect{\G_S(t_S)} &\, =\, 0,\\
  \label{eq:Gcov}\Expect{\G_S(t_S)\conj{\G_S(t_S')}} &\, =\,  \prod_{i\in S}
  \left(f_i(t_i-t_i') - f_i(t_i)f_i(-t_i')\right), \\
  \label{eq:pseudocov} \Expect{\G_S(t_S)\G_S(t_S')} &\, =\,  \prod_{i\in S}
  \left(f_i(t_i+t_i') -
    f_i(t_i)f_i(t_i')\right).
\end{align}
\end{proposition}

\begin{remark}
Proposition~\ref{pro:moments-ZNlimit} shows that the complex
Gaussian random field $\G_S$ has the kernels given in \eqref{eq:Gcov} and
\eqref{eq:pseudocov} (cf.\ Section \ref{sec:gaussianprocess}).
\end{remark}

\begin{proof}
Recall the It\^o isometry for a Brownian sheet $\W$
(e.g. \cite[Equation (1)]{SottViit2015}) for two
functions $f,g \in L^2([0,1]^n,dx))$:
\begin{equation*}
\E\left( \int_{[0,1]^n} f(x) \, d\W(x) \cdot \int_{[0,1]^n} g(x) \,
  d\W(x)\right) = \int_{[0,1]^n} f(x)g(x) \, dx.
\end{equation*}
Thus for two functions  $f,g\in L^2(\R^{\sum d_i},F_{\vect{X}})$ 
and independent $X_1,\dots, X_n$:
\begin{equation*}
\E\left( \int f(x) \, d\W(\vect{F}_\vect{X}(x)) \cdot \int g(x) \,
  d\W(\vect{F}_\vect{X}(x))\right) = \int f(x)g(x) \, dF_\vect{X}(x),
\end{equation*}
where $dF_\vect{X} = d\vect{F}_\vect{X}$ due to the independence.  This
implies
\begin{equation*}
\begin{split}
\E(\G_S(t_S)&\conj{\G_R(t_R')})  = \E\Biggl[  \int \prod_{i\in S}
\left(\ee^{\ii x_i\cdot t_i}-f_{X_i}(t_i)\right)\,d\W(\vect{F}_\vect{X}(x))
\,\cdot \\ &\hspace{4cm}  \conj{\int \prod_{i\in R} \left(\ee^{\ii x_i\cdot
      t_i'}-f_{X_i}(t_i')\right)\,d\W(\vect{F}_\vect{X}(x))}\Biggr]\\
&= \int \prod_{i\in S} \left(\ee^{\ii x_i\cdot t_i}-f_{X_i}(t_i)\right) \cdot
\prod_{i\in R} \left(\ee^{-\ii x_i\cdot
    t_i'}-f_{X_i}(-t_i')\right)\,dF_S(x_S)\\
&= \begin{cases}
0& \text{for } S\neq R,\\
\prod_{i\in S} \left(f_i(t_i-t_i') - f_i(t_i)f_i(-t_i')\right)& \text{for } S = R.
\end{cases}
\end{split}
\end{equation*}
Analogously, \eqref{eq:pseudocov} is proved. Thus for $S\neq R$ the
random variables $\G_S$ and $\G_R$ are uncorrelated and they are jointly
Gaussian. Therefore they are independent.
\end{proof}

\subsection{Limit of (scaled) sample distance multivariance -- $N \cdot \hN M^2_{\rho_S}$}

Next we can derive a new representation of the $L^2(\rho_S)$-norm of the
random field $\G_S$.

\begin{proposition} \label{prop:replimit}
  Let $\E(\psi_i(X_i))<\infty$ for $1\leq i\leq n$ and $\G_S$ be
  the random field defined in \eqref{eq:ZNBMlimit2}. Then
  \begin{equation*}
    \norm{\pG_S}_{\rho_S}^2 = \iint \prod_{i\in S} \Psi_i(x_i,y_i)
    \,d\W(\vect{F}_\vect{X}(x))\,d\W(\vect{F}_\vect{X}(y))  \qquad
    \Prob\text{-a.s.}
  \end{equation*}
  with
  \begin{equation}
    \begin{split}
      \label{eq:def-Psi}
      \Psi_i&(x_i,y_i) := \\
      &-\psi_i(x_i-y_i) + \E(\psi_i(X_i-y_i)) +
      \E(\psi_i(x_i-X_i')) - \E(\psi_i(X_i-X_i')),
    \end{split}
  \end{equation}
  where $(X_1', \dots, X_n')$ is an independent copy of $(X_1, \dots, X_n)$
  and the $X_i$ are independent. Here $\W$ and $\vect{F}_\vect{X}$ are as in
  Theorem \ref{thm:ZNBMlimit}.
\end{proposition}
The proof of Proposition \ref{prop:replimit} is a technical application of
Fubini's theorem and can be found in Section \ref{sec:proof:replimit} in the
Appendix.

\begin{remark} \label{rem:L2limit}
For the application of distance multivariance in independence testing it is
fundamental to note that
\begin{equation*}
N \cdot \hN M^2_{\rho_S}(\vect{X}^{(1)},\dots,\vect{X}^{(N)})
\xrightarrow[\quad N \to \infty \quad]{d} \norm{\pG_S}_{\rho_S}^2
\end{equation*}
for independent $X_1,\dots,X_n$ satisfying \eqref{eq:con-convergence}. This
statement is a non-trivial consequence (see \cite[Theorem
4.5]{BoetKellSchi2018a}) of the convergence noted in Theorem
\ref{thm:ZNBMlimit}.
\end{remark}

Furthermore, the moments can be calculated.
\begin{corollary} \label{cor:momG}
Let $\E(\psi_i(X_i))<\infty$ for $1\leq i\leq n$, then the moments
of the $L^2(\rho)$-norm of $\G_S$ defined in \eqref{eq:ZNBMlimit2} are finite
and given by
\begin{eqnarray*}
\E(\norm{\pG_S}^2_{\rho_S}) & = & \mu_S^{(1)},\\
\V(\norm{\pG_S}^2_{\rho_S}) = \E((\norm{\pG_S}^2_{\rho_S}- \mu_S^{(1)})^2) & =
                                & 2 \mu_S^{(2)},\\
\E((\norm{\pG_S}^2_{\rho_S}- \mu_S^{(1)})^3) & = & 8 \mu_S^{(3)},\\
 \skw(\norm{\pG_S}^2_{\rho_S}) &=& \frac{8 \mu_S^{(3)} }{(2
                                   \mu_S^{(2)})^\frac{3}{2}} = \frac{4
                                   \mu_S^{(3)} }{\sqrt{2} (
                                   \mu_S^{(2)})^\frac{3}{2}}, \\
\E((\norm{\pG_S}^2_{\rho_S}- \mu_S^{(1)})^4) & = & 48 \mu_S^{(4)} + 12
                                                   (\mu_S^{(2)})^2,\\
 \exkurt(\norm{\pG_S}^2_{\rho_S}) &=& \frac{48 \mu_S^{(4)}+ 12
                                      (\mu_S^{(2)})^2}{ (2\mu_S^{(2)})^2}-3=
                                      12\frac{ \mu_S^{(4)}}{ (\mu_S^{(2)})^2},
\end{eqnarray*}
with 
\begin{equation*}
\mu_S^{(k)} = \prod_{i\in S} \mu_i^{(k)}
\end{equation*}
and 
\begin{equation} \label{eq:muk}
\mu_i^{(k)} = \int\dotsi\int \prod_{j=1}^{k-1}
  \covG_i(\itint{k-j+1}_i,\itint{k-j}_i)\cdot
  \covG_i(\itint{1}_i,\itint{k}_i)\rho_i(d\itint{1}_i)\dotsb\rho_i(d\itint{k}_i)
\end{equation}
and 
\begin{equation*}
 \covG_i(s_i,t_i):=f_i(s_i - t_i) - f_i(s_i) f_i( - t_i) \text{ for $s_i,t_i
   \in \R^{d_i}$ and $1\leq
   i\leq n$}.
\end{equation*}
\end{corollary}
\begin{proof}
All representations follow by direct calculation (which might become quite
technical) or using the results of Section \ref{sec:gaussianprocess}
(Proposition~\ref{prop:pnormintegral}, Remark~\ref{rem:operator:product}).

For the finiteness note that (using the inequality \cite[Equation
  (66)]{Boet2019})
\begin{equation*}
\mu_S^{(1)} = \int \prod_{i \in S} (1-|f_i(t_i)|^2)\,d\rho_S = \prod_{i \in S}
\E(\psi_i(X_i-X_i')) \leq 2^{|S|} \E(\psi_i(X_i)) < \infty,
\end{equation*}
which implies the finiteness of all moments by
Proposition~\ref{prop:pnormintegral}.
\end{proof}

To use the above we need to analyze the $\mu_i^{(k)}$, which will be done in
the next section. Before we want to point out some related facts.

\begin{remark} \label{rem:limitmom}
\begin{enumerate}
\item Let $\E(\psi_i(X_i))<\infty$ for $1\leq i\leq n$, then
\begin{align}
\label{eq:limitE}\E(\norm{\pG_S}^2_{\rho_S}) &= \phantom{2} \prod_{i\in S}
\E(\psi_i(X_i-X_i')),\\
\E(\norm{\pG_S}^4_{\rho_S}) & =  \phantom{2}\prod_{i\in S}
\left[\E(\psi_i(X_i-X_i'))\right]^2 +  2 \prod_{i\in S} M_{\rho_i \otimes
  \rho_i}^2(X_i,X_i),\\
 \label{eq:repk2}\V(\norm{\pG_S}^2_{\rho_S}) &=  2 \prod_{i\in S} M_{\rho_i
   \otimes \rho_i}^2(X_i,X_i),
\end{align}
where
\begin{align*}
\MoveEqLeft[2] M^2_{\rho_i \otimes \rho_i}(X_i,X_i)\\
&= \int \int \left|\E( (\ee^{\ii X_i\cdot s_i}-f_{X_i}(s_i))(\ee^{\ii X_i\cdot
    t_i}-f_{X_i}(t_i))) \right|^2 \,\rho_i(ds_i)\,\rho_i(dt_i)\\
& = \int \int \left|f_i(s_i+t_i) - f_i(s_i)f_i(t_i)\right|^2
\,\rho_i(ds_i)\,\rho_i(dt_i).
\end{align*}

\item \label{rem:limitmomnormalized} If $\norm{\pG_S}^2_{\rho_S}$ is
  normalized by its expectation, assuming $\E(\psi_i(X_i))<\infty$ for 
  $1\leq i\leq n$, then the skewness and excess-kurtosis are unaltered
  and
  \begin{equation*}
    \E\left(\frac{\norm{\pG_S}^2_{\rho_S}}{\mu_S^{(1)}}\right) = 1\quad \text{ and
    } \quad
    \V\left(\frac{\norm{\pG_S}^2_{\rho_S}}{\mu_S^{(1)}}\right)  = 2
    \frac{\mu_S^{(2)}}{(\mu_S^{(1)})^2}.
  \end{equation*}
\end{enumerate}
\end{remark}

\subsection{Moments of the limit distribution -- $\mu_i^{(k)}$}

In this section we consider the $\mu_i^{(k)}$ defined in \eqref{eq:muk} for
one fixed $i$ and drop the subscript $i$ for readability up to Corollary
\ref{cor:estlimitmom}. Thus we have a random variable $X$ and independent
copies of it denoted by $X',X'',X''',X''''$, a symmetric measure $\rho$
satisfying $\int 1\land |t|^2 \,\rho(dt) < \infty$ and $\psi(x) := \int 1-
\cos(x\cdot t) \,\rho(dt)$.
 
\begin{lemma} \label{lem:mureps}
Let $\E(\psi_i(X_i)^2)<\infty$ for $1\leq i\leq n$, then the
$\mu^{(k)}$ given in \eqref{eq:muk} have the following representations
\begin{eqnarray*}
\mu^{(1)} &=& \E(\psi(X-X')),\\
\mu^{(2)} &= &\E\left(\psi(X-X')\psi(X'-X)\right) -
               2\E\left(\psi(X-X')\psi(X'-X'')\right) \\
  && {}+  \left[\E\left(\psi(X-X')\right)\right]^2,\\
  \mu^{(3)} &=& -\E\left(\psi(X-X')\psi(X'-X'')\psi(X''-X)\right) \\
  &&{}+ 3 \E\left(\psi(X-X')\psi(X'-X'')\psi(X''-X''')\right)\\
&&{}-3 \E\left(\psi(X-X')\psi(X'-X'')\right)\E\left(\psi(X-X')\right) +
   \left[\E\left(\psi(X-X')\right)\right]^3,\\
\mu^{(4)} &= &\E\left(\psi(X-X')\psi(X'-X'')\psi(X''-X''')\psi(X'''-X)\right)\\
&&{}-4 \E\left(\psi(X-X')\psi(X'-X'')\psi(X''-X''')\psi(X'''-X'''')\right)\\
&&{}+4 \E\left(\psi(X-X')\psi(X'-X'')\psi(X''-X''')\right)
   \E\left(\psi(X-X')\right)\\
&&{}+2 \left[\E\left(\psi(X-X')\psi(X'-X'')\right) \right]^2\\
&&{}-4 \E\left(\psi(X-X')\psi(X'-X'')\right)
   \left[\E\left(\psi(X-X')\right)\right]^2+
   \left[\E\left(\psi(X-X')\right)\right]^4.
\end{eqnarray*}
\end{lemma}
\begin{proof}
The representation of $\mu^{(1)}$ follows by \eqref{eq:limitE}.

Note that by \eqref{eq:repk2} and \cite[Equation (30)]{BoetKellSchi2018} 
\begin{multline*}
  \mu^{(2)} = M^2_{\rho \otimes \rho}(X,X) =  \E\left(\psi(X-X')\psi(X'-X)\right)
  \\ - 2\E\left(\psi(X-X')\psi(X'-X'')\right) +  
\left[\E\left(\psi(X-X')\right)\right]^2.
\end{multline*}
By Proposition~\ref{prop:pnormintegral} (see also
Remark~\ref{rem:mu})
\begin{align*}
\mu^{(3)}&= \int \covG(s,t)\covG(t,r) \covG(r,s)\, \rho(ds)\rho(dt)\rho(dr) \\
&= \int \left[f(s-t)-f(s)\conj{f(t)}\right]\cdot
\left[f(t-r)-f(t)\conj{f(r)}\right]\cdot \\
&\hspace{3cm} \left[f(r-s)-f(r)\conj{f(s)}\right]\,
\rho(ds)\,\rho(dt)\,\rho(dr).
\end{align*}
Expanding the product and using the symmetry (i.e., that $s,t,r$ can be
interchanged) yields
\begin{multline*}
\mu^{(3)}= \int  f(s-t)f(t-r)f(r-s) - 3 f(s-t)f(t-r)f(r)\conj{f(s)}\\
+ 3 f(s-t) f(t)\conj{f(s)}|f(r)|^2  - |f(s)f(t)f(r)|^2\,
\rho(ds)\rho(dt)\rho(dr).
\end{multline*}
If $X,X',X'',X'''$ are independent and identically distributed then the terms
of the integrand can be rewritten as
\begin{align*}
f(s-t)f(t-r)f(r-s) &= \E\left(\ee^{\ii s\cdot(X-X')}\ee^{\ii
    t\cdot(X'-X'')}\ee^{\ii r\cdot(X''-X)}\right) \\
f(s-t)f(t-r)f(r)\conj{f(s)}&= \E\left(\ee^{\ii s\cdot(X-X')}\ee^{\ii
    t\cdot(X'-X'')}\ee^{\ii r\cdot(X''-X''')}\right)\\
f(s-t) f(t)\conj{f(s)}|f(r)|^2 &= \E\left(\ee^{\ii s\cdot(X-X')}\ee^{\ii
    t\cdot(X'-X'')}\right)\E\left(\ee^{\ii r\cdot(X-X')}\right)\\
|f(s)f(t)f(r)|^2 &=\E\left(\ee^{\ii s\cdot(X-X')}\right)\E\left(\ee^{\ii
    t\cdot(X-X')}\right)\E\left(\ee^{\ii r\cdot(X-X')}\right).
\end{align*}
Interchanging the order of integration, expanding the products and collecting
the terms yields
\begin{equation}
\begin{split}
\label{eq:repk3}
  \mu^{(3)} = &-\E\left(\psi(X-X')\psi(X'-X'')\psi(X''-X)\right) \\
  &+ 3 \E\left(\psi(X-X')\psi(X'-X'')\psi(X''-X''')\right)\\
  &- 3 \E\left(\psi(X-X')\psi(X'-X'')\right)\E\left(\psi(X-X')\right)\\
  & +
  \left[\E\left(\psi(X-X')\right)\right]^3.
\end{split}
\end{equation}
By an analogous (but longer) calculation one gets the representation of
$\mu^{(4)}$.
\end{proof}

The estimation of these $\mu_i^{(k)}$ is straightforward by the law of large
numbers for V-statistics. But in this specific setting we can reduce the
required moments.

\begin{theorem} \label{thm:estmu}
Let $\E(\psi_i(X_i))<\infty$ for $1\leq i\leq n$, $x =
(x^{(1)},\dots,x^{(N)})$ be a sample of $X$ and set
$B:=(\psi(x^{(j)}-x^{(k)}))_{j,k=1,\dots,N}$ then
\begin{align}
m:= \frac{1}{N^2}|B| \quad \xrightarrow{N\to \infty } \quad&\mu^{(1)},\\
\label{eq:estmu2}\frac{1}{N^2} |B\circ B| - \frac{2}{N^3} |B^2| + m^2 \quad
\xrightarrow{N\to \infty } \quad &\mu^{(2)}, \\
-\frac{1}{N^3} |B^2\circ B| + \frac{3}{N^4}|B^3| - \frac{3}{N^3}|B^2|m + m^3
\quad \xrightarrow{N\to \infty } \quad&\mu^{(3)},\\
\frac{1}{N^4} |B^3\circ B| - \frac{4}{N^5}|B^4| + \frac{4}{N^4}|B^3|m + 2
\left[ \frac{1}{N^3} |B^2|\right]^2 \hspace{2cm} &\\
-\frac{4}{N^3}|B^2|m^2 + m^4\quad \xrightarrow{N\to \infty }
\quad&\mu^{(4)},\notag
\end{align}
where $\circ$ denotes the Hadamard product and for a matrix $A$
the sum of the absolute values of its entries is denoted by $|A|$. (But note
that here all matrices have always non-negative entries anyway.) The
convergence is meant in the sense of the strong law of large numbers (SLLN),
i.e., almost sure convergence if in the estimators the samples are replaced by
the corresponding random variables.
\end{theorem}
\begin{rem}
The estimator for $\mu^{(2)}$ in \eqref{eq:estmu2} is nothing but the
estimator of $M(X,X)$, i.e.,
\begin{equation*}
\hN M(X,X) =
\frac{1}{N^2}\left|\left[(I-\frac{1}{N}\bm{1})B(I-\frac{1}{N}\bm{1})\right]\circ
  \left[(I-\frac{1}{N}\bm{1})B(I-\frac{1}{N}\bm{1})\right]\right|.
\end{equation*}
\end{rem}
\begin{proof}[Proof of Theorem \ref{thm:estmu}]
We start with the stronger assumption $\E(\psi_i(X_i)^2)<\infty$ for
$1\leq i\leq n$: To estimate $\mu^{(k)}$ note that by the strong
law of large numbers for V-statistics -- see, e.g., \cite[Theorem
3.3.1]{KoroBoro1994} --
\begin{equation*}
	\frac{1}{N^{p+1}} \sum^N_{k_1,\dots,k_{p+1} = 1} \prod_{j=1}^p
        \psi(X^{(k_j)}-X^{(k_{j+1})}) \xrightarrow[N\to\infty]{a.s.} \E\left(
          \prod_{j=1}^p \psi(X^{(j)}-X^{(j+1)})\right)
\end{equation*}
and for $p>1$
\begin{multline*}
	\frac{1}{N^{p}} \sum^N_{k_1,\dots,k_{p} = 1}
        \psi(X^{(k_p)}-X^{(k_1)})\prod_{j=1}^{p-1}
        \psi(X^{(k_j)}-X^{(k_{j+1})}) \\
        \xrightarrow[N\to\infty]{a.s.}\ \E\left(\psi(X^{(p)}-X^{(1)})
          \prod_{j=1}^{p-1} \psi(X^{(j)}-X^{(j+1)})\right).
\end{multline*}
Thus, given a sample $x = (x^{(1)},\dots,x^{(N)})$, setting
 $B:=(\psi(x^{(i)}-x^{(k)}))_{i,k=1,\dots,N}$ yields (just recall that
 $(C\cdot D)_{k,i} = \sum_j c_{k,j}d_{j,i}$)
\begin{align*}
 \sum^N_{k_1,\dots,k_{p+1} = 1} \prod_{j=1}^p \psi(x^{(k_j)}-x^{(k_{j+1})}) &= |B^p|,\\
 \sum^N_{k_1,\dots,k_{p} = 1} \psi(x^{(k_p)}-x^{(k_1)})\prod_{j=1}^{p-1}
 \psi(x^{(k_j)}-x^{(k_{j+1})}) &= |B^{p-1} \circ B|.
\end{align*}
Now using these approximations for each of the summands in the representations
given in Lemma \ref{lem:mureps} yield the estimators.

Finally, the moment assumption is relaxed by the approximation argument
presented in Section~\ref{app:relaxmom} in the Appendix. 
\end{proof}

Putting Corollary \ref{cor:momG} and Theorem \ref{thm:estmu} together yields
the key result.

\begin{corollary}\label{cor:estlimitmom}
  Let $\E(\psi_i(X_i))<\infty$ for $1\leq i\leq n$,
  $\vect{x}^{(1)},\dots,\vect{x}^{(N)}$ be samples of $(X_1,\dots,X_n)$ (with
  possibly dependent components!), $S\subset \{1,\dots,n\}$ with $|S|\geq 2$
  and let $\norm{\pG}^2_{\rho_S}$ be the distributional limit of the test
  statistic $N\cdot \hN M^2_{\rho_S}$ under the hypothesis of
  independence. Then
\begin{eqnarray*}
  \hN\mu_S^{(1)}&\xrightarrow{N \to \infty}& \E(\norm{\pG}^2_{\rho_S}),\\
  2\cdot \hN\mu_S^{(2)}&\xrightarrow{N \to \infty}& \V(\norm{\pG}^2_{\rho_S}),\\
  8 \cdot \hN\mu_S^{(3)}&\xrightarrow{N \to \infty}& \E\left[
                                                     (\norm{\pG}^2_{\rho_S} -
                                                     \E(\norm{\pG}^2_{\rho_S}))^3
                                                     \right],\\ 
  48\cdot\hN\mu_S^{(4)}+ 12 \cdot\left( \hN\mu_S^{(2)}\right)^{
  2}&\xrightarrow{N \to \infty}& \E\left[(\norm{\pG}^2_{\rho_S} -
                                 \E(\norm{\pG}^2_{\rho_S}))^4\right], 
\end{eqnarray*}
in the sense of almost sure convergence (replacing samples in the estimators
by the corresponding random variables) and using $\hN\mu_S^{(k)} :=
\prod_{i\in S} \hN\mu_i^{(k)}$ with
\begin{align} \label{eq:muestimators}
\hN\mu_i^{(1)} &:= \frac{1}{N^2}|B_i|, \\
\hN\mu_i^{(2)} &:= \frac{1}{N^2} |B_i\circ B_i| - \frac{2}{N^3} |B_i^2| +
\frac{1}{N^4}|B_i|^2, \\
\hN\mu_i^{(3)} &:= -\frac{1}{N^3} |B_i^2\circ B_i| + \frac{3}{N^4}|B_i^3| -
\frac{3}{N^3}|B_i^2|\frac{1}{N^2}|B_i| + \frac{1}{N^6}|B_i|^3, \\
\hN\mu_i^{(4)} &:= \frac{1}{N^4} |B_i^3\circ B_i| - \frac{4}{N^5}|B_i^4| +
\frac{4}{N^4}|B_i^3|\frac{1}{N^2}|B_i| + 2 \left[ \frac{1}{N^3}
  |B_i^2|\right]^2 \\
&\hspace{5cm} -\frac{4}{N^3}|B_i^2|\frac{1}{N^4}|B_i|^2 +
\frac{1}{N^8}|B_i|^4,\notag
\end{align}
with $B_i := (\psi_i(x_i^{(l)}-x_i^{(k)}))_{l,k=1,\dots,N}$ and $\psi_i$ is
given in \eqref{def:sm-cdms}. As before $|A|:= \sum_{l,k=1}^N |a_{l,k}|$.
\end{corollary}

\begin{remark}\label{rem:unbiased}
Based on the given estimators the corresponding unbiased estimators can also
be calculated, we require these for $\mu_i^{(1)}, \mu_i^{(2)}$ and $\mu_i^{(3)}$
(cf. Section \ref{sec:tests}). The convergence of parts of the estimators
introduced below will require $\psi_i$-moments of an order 2 and higher, but
as described in Section~\ref{app:relaxmom} in the Appendix finally
for all $\mu_i^{(k)}$ only $\E(\psi_i(X_i))<\infty$ for $1\leq i\leq n$
is required.

For $\mu_i^{(1)}$ the unbiased estimator is
\begin{equation*}
\frac{1}{N (N-1)}|B_i| \xrightarrow{N \to \infty} \mu_i^{(1)}.
\end{equation*}
Moreover, the following estimators are unbiased by direct calculations 
\begin{align}
\label{eq:approxbunbiased}  \hN b_i := \frac{1}{N (N-1)} |B_i\circ B_i|  \quad
&\xrightarrow{N\to \infty } \quad b_i:=\E(\psi_i(X_i-X_i')^2),\\
\label{eq:approxcunbiased}  \hN c_i :=\frac{1}{N(N-1)(N-2)} \big(|B_i^2| &-
  |B_i\circ B_i|\big)  \\ & \xrightarrow{N\to \infty } \quad
c_i:=\E(\psi_i(X_i-X_i')\psi_i(X_i'-X_i'')),\notag\\
\label{eq:approxdunbiased}  \hN d_i :=\frac{1}{N(N-1)(N-2)(N-3)}&\left(|B_i|^2+
  2|B_i\circ B_i| - 4 |B_i^2| \right) \\ &\xrightarrow{N\to \infty } \quad
d_i:=[\E(\psi_i(X_i-X_i'))]^2.\notag
  \end{align}
 Thus using the representation obtained in Lemma \ref{lem:mureps} the unbiased
 estimator for $\mu_i^{(2)}$ is
 \begin{equation} \label{eq:approxmu2unbiased}
 \hN b_i - 2\cdot \hN c_i + \hN d_i \quad \xrightarrow{N\to \infty } \quad
 \mu_i^{(2)}.
 \end{equation}
Straightforward, but more tedious computations (see
Section~\ref{app:unbiased} in the Appendix), give unbiased estimators for the
representation \eqref{eq:repk3}. Setting
\begin{align*}
  e_i &:=
  \E\left(\psi_i(X_i-X_i')\psi_i(X_i'-X_i'')\psi_i(X_i''-X_i)\right),\\
  f_i &:=
  \E\left(\psi_i(X_i-X_i')\psi_i(X_i'-X_i'')\psi_i(X_i''-X_i''')\right)
\end{align*}
one obtains the unbiased estimators (with $m_i := \E(\psi_i(X_i-X_i'))$) 
\begin{align}
\label{eq:approxeunbiased} \hN e_i:= \frac{1}{N (N-1) (N-2)} |B_i^2\circ B_i|
\quad
\xrightarrow{N\to \infty } \quad &e_i,\\
\hN f_i:= \frac{1}{N (N-1) (N-2) (N-3)}
\bigl(|B_i^3|-|B_i^2\circ B_i|\notag\hspace{1.5cm} &\\
\label{eq:approxfunbiased}{} -2|(B_i\circ B_i)\cdot B_i|+|B_i\circ B_i\circ
B_i|\bigr) \quad\xrightarrow{N\to \infty } \quad& f_i,\\
\hN y_i:=   \frac{(N-5)!}{N!}
 \bigl(|B_i^2|\cdot|B_i|-|B_i\circ B_i|\cdot|B_i|\hspace{3cm} \notag \\ 
 {}-2|\colsum(B_i)\circ\colsum(B_i)\circ\colsum(B_i)| -4|B_i\circ
 B_i\circ B_i|\hspace{1.5cm}   \label{eq:approxmcunbiased}\\
 {}-4|B_i^3|+2|B_i^2\circ B_i|+10|(B_i\circ B_i)\cdot B_i|\bigr) \quad
 \xrightarrow{N\to \infty } \quad& m_i\cdot c_i,\notag\\
\hN u_i:=  \frac{(N-6)!}{N!}
 \bigl(|B_i|^3+16|B_i\circ
 B_i\circ B_i|-48|(B_i\circ B_i)\cdot B_i|\hspace{1cm}\notag\\ {}- 8|B_i^2\circ
 B_i|+6 |B_i|\cdot|B_i\circ B_i|
 +24|B_i^3|\hspace{1cm}\label{eq:approxm3unbiased}\\
 {}+16|\colsum(B_i)\circ\colsum(B_i)\circ\colsum(B_i)|
 -12|B_i^2|\cdot|B_i|\bigr) \quad
 \xrightarrow{N\to \infty } \quad& m_i^3,\notag
\end{align}
where $\colsum(A)$ denotes the vector of the column
sums $\sum_{k=1}^N A_{j,k}$, $1\leq j\leq N$, of a matrix $A$.

Thus the unbiased estimator for $\mu_i^{(3)}$ is
\begin{equation*}
-1 \cdot \hN e_i + 3\cdot \hN f_i - 3 \cdot \hN y_i + \hN u_i \quad
\xrightarrow{N\to \infty } \quad \mu_i^{(3)}.
\end{equation*}
In all instances, convergence is meant in the sense of almost sure convergence
(replacing samples in the estimators by the corresponding random variables).
\end{remark}

\subsection{Moments of the finite sample distribution} \label{sec:finitesample}

In practice one never has an infinite sample, thus in fact not the
limit distribution but the finite sample distribution is
relevant. For sample distance multivariance we show in the following that one
can calculate (and estimate) the moments of the finite sample distribution. In
general these moments can differ considerably from the limit
moments, cf.\ Example \ref{ex:marginal_mom_sample}. Thus the use of these
moments in the quadratic form estimates provides better (e.g., closer to the
nominal size) tests than using the limit moments. But note that especially for
very small sample sizes (the mythical $N < 30$) the use of the central limit
theorem is doubtful, i.e., the distribution of sample distance multivariance
might not yet be close to that of a Gaussian quadratic form (see Example
\ref{ex:not_gQF}).

In order to analyze the finite sample distribution we denote by $\hN
\Psi_i$ the empirical approximation to $\Psi_i$ defined in
\eqref{eq:def-Psi}, i.e.,
\begin{align}
\notag \hN \Psi_i(x_i^{(j)},x_i^{(k)}):= &\ \hN
\Psi_i(x_i^{(j)},x_i^{(k)};\,\vect{x}^{(1)},\dots, \vect{x}^{(N)}) :=\\
\label{eq:def-PsiN}  &- \psi_i(x_i^{(j)}- x_i^{(k)}) + \frac{1}{N} \sum_{l=1}^N
  \psi_i(x_i^{(j)}-x_i^{(l)}) \\ 
\notag  &+ \frac{1}{N} \sum_{m=1}^N \psi_i(x_i^{(m)}-x_i^{(k)})- \frac{1}{N^2}
\sum_{m,l=1}^N \psi_i(x_i^{(m)}-x_i^{(l)}).
\end{align}
Then by \cite[Eqs.\ (4.3), (4.4)]{BoetKellSchi2018a}, or directly by noting
that $\hN \Psi_i(x_i^{(j)},x_i^{(k)}) = (A_i)_{jk}$ with $A_i$ given in
\eqref{def:sm-cdms}, sample distance multivariance has the representation
\begin{equation*}
\hN M^2 (\vect{x}^{(1)},\dots, \vect{x}^{(N)}) = \frac{1}{N^2} \sum_{j,k=1}^N
\prod_{i=1}^n \hN\Psi_i(x_i^{(j)},x_i^{(k)}).
\end{equation*}

Based on this one can calculate (under the hypothesis of independence) the
expectation of the test statistic $N \cdot \hN M^2$ and (with a lot of effort)
also the variance. The calculation of skewness and higher moments seems
for us technically out of reach, but it also turns out that the limit skewness
works well in the statistical applications in Section
\ref{sec:examples}. Nevertheless, note that for the related (but different)
finite sample permutation statistic the first three moments have been
calculated explicitly in \cite[Chapter 3, Theorems 3.5, 3.6 and
3.9]{Guet2017}.
\begin{theorem}[Moments of sample distance
  multivariance] \label{thm:finitemoms}\hfill
  Let $X_1,\dots,X_n$ be independent and $\E(\psi_i(X_i)^2)<\infty$ for
  $1\leq i\leq n$. Then
\begin{align}
  \E( N \cdot \hN M_\rho^2 &(\vect{X}^{(1)},\dots, \vect{X}^{(N)}))\notag  \\
  &=\frac{(N-1)^n + (-1)^n (N-1)}{N^n} \prod_{i=1}^n
  \E(\psi_i(X_i-X_i')),\label{eq:finitemomEalt}\\
&= \prod_{i=1}^n \left[ \left(1-\frac{1}{N}\right) \mu_i^{(1)}\right] +
(N-1)\prod_{i=1}^n \left[ \left(-\frac{1}{N}\right)
  \mu_i^{(1)}\right]\notag\\[2ex]
\E( [N\cdot \hN M_\rho^2&(\vect{X}^{(1)},\dots,\vect{X}^{(N)} )]^2) \notag\\
&=\frac{1}{N^2}\sum_{k=1}^7C(N,k)
\prod_{i=1}^n \left[\frac{b(N,k) b_i + c(N,k) c_i +
    d(N,k)d_i}{N^4}\right],\label{eq:finite2mom}
\end{align}
where $b_i:=\E(\psi_i(X_i-X_i')^2)$,
$c_i:=\E(\psi_i(X_i-X_i')\psi_i(X_i'-X_i''))$, $d_i:=[\E(\psi_i(X_i-X_i'))]^2$
and the coefficients $C(N,k),b(N,k),c(N,k),d(N,k)$ are given in
Table \ref{tab:coef} on page \pageref{tab:coef}.
\end{theorem}  
\begin{remark} \label{rem:finitemoms}
\begin{enumerate}
\item \label{rem:finitmoms:Nu_vs_lb} Note that if $\mu_i^{(1)}$ is
  approximated by its unbiased estimator $\frac{1}{N(N-1)}|B_i|$ given in
  Remark \ref{rem:unbiased} then $(1-\frac{1}{N})$ times this estimator
  becomes the biased estimator of $\mu_i^{(1)}$ given in Theorem
  \ref{thm:estmu}. Thus by \eqref{eq:finitemomEalt} the unbiased estimate of
  the finite sample mean is close to the biased estimate of the limit mean. In
  fact for odd $n$ the biased limit mean estimate is larger, for even $n$ it
  is smaller.
\item Analyzing the limit behavior of the coefficients in
  \eqref{eq:finite2mom} and of the whole moment recovers the limit result as
  derived in Corollary \ref{cor:momG}. This is indicated by the extra columns
  in Table \ref{tab:coef} on page \pageref{tab:coef}.
\item \label{rem:finitemoms:est} Note that the values $b_i,c_i,d_i$ in
  \eqref{eq:finite2mom} can be directly estimated. Given a sample $x =
  (x^{(1)},\dots,x^{(N)})$ of $X$ and
  $B_i:=(\psi_i(x_i^{(j)}-x_i^{(k)}))_{j,k=1,\dots,N}$ then
\begin{align}
\label{eq:approxa} \frac{1}{N^2} |B_i\circ B_i| \quad \xrightarrow{N\to \infty
} \quad& b_i,\\
\label{eq:approxb} \frac{1}{N^3} |B_i^2| \quad \xrightarrow{N\to \infty }
\quad& c_i,\\
\label{eq:approxc} \frac{1}{N^2}|B_i| \quad \xrightarrow{N\to \infty } \quad&
\sqrt{d_i} = \E(\psi_i(X_i-X_i')).
  \end{align}
  The corresponding unbiased estimators for $b_i,c_i$ and $d_i$ are also
  available, see Remark~\ref{rem:unbiased}. See also \eqref{eq:estmu2} and
  \eqref{eq:approxmu2unbiased}, which show that these are also used for the
  estimation of the variance (or second moment) of the limit.

\item \label{rem:finitemoms:var} For $k=7$ (cf.~Table
  \ref{tab:coef} on page \pageref{tab:coef}) the corresponding summand in
  \eqref{eq:finite2mom} is roughly $N \prod_{i=1}^n (4 c_i -
  3d_i)$. Especially for large $n$ this summand (times $\frac{1}{N^2}$)
  vanishes very slowly as the sample size increases. In fact, the difference
  between the finite sample moment and the moment of the limit is mainly due
  to this summand.
\end{enumerate}
\end{remark}
\begin{table}[ht]
\begin{tabular}{rl|l||l}
	          & coefficients                            & $N\to\infty$                            & overall $N\to \infty$                                 \\
	          &                                         & $\frac{C}{N^2}$,  $\frac{a/b/c/d}{N^4}$ &                                                       \\ \hline
	\multicolumn{3}{l}{1. case: $|\{j,k,l,m\}|=4$}                                                &                                                       \\[3pt]
	$C(N,1)=$ & $N(N-1)(N-2)(N-3)$                      & $\infty $                               &                                                       \\
	$a(N,1)=$ & $- N^2 $                                & $0$                                     & $0$                                                   \\
	$b(N,1)=$ & $6  N + 2  N^2 $                        & $0$                                     & $0$                                                   \\
	$c(N,1)=$ & $- 24  N - 4  N^2 $                     & $0$                                     & $0$                                                   \\
	$d(N,1)=$ & $18  N + 3  N^2 $                       & $0$                                     & $0$                                                   \\ \hline
	\multicolumn{3}{l}{2. case: $|\{j,k\}\cap \{l,m\}|=2$}                                        &                                                       \\[3pt]
	$C(N,2)=$ & $2N(N-1)$                               & $2$                                     &                                                       \\
	$a(N,2)=$ & $- N^2 $                                & $0$                                     & $0$                                                   \\
	$b(N,2)=$ & $6  N - 2  N^2 - 2  N^3 + N^4 $         & $1$                                     & *             \\
	$c(N,2)=$ & $- 24  N + 12  N^2 + 4  N^3 - 2  N^4 $  & $-2$                                    & * \\
	$d(N,2)=$ & $18  N - 9  N^2 - 2  N^3 + N^4 $        & $1$                                     &  *             \\ 
\multicolumn{4}{l}{* are together equal to}\\
\multicolumn{4}{l}{\phantom{* }$2\prod_{i=1}^n\Bigl(\E(\psi_i(X_i-X_i')^2)-2\E(\psi_i(X_i-X_i')\psi(X_i'-X_i''))+E(\psi_i(X_i-X_i'))]^2\Bigr)$}\\
	\hline
	\multicolumn{3}{l}{3. case:  $j\neq k$ and $l\neq m$ and $|\{j,k\}\cap \{l,m\}|=1$}           &                                                       \\[3pt]
	$C(N,3)=$ & $4N(N-1)(N-2)$                          & $\infty $                               &                                                       \\
	$a(N,3)=$ & $- N^2 $                                & $0$                                     & $0$                                                   \\
	$b(N,3)=$ & $6  N - N^3 $                           & $0$                                     & $0$                                                   \\
	$c(N,3)=$ & $- 24  N + 4  N^2 + 2  N^3 $            & $0$                                     & $0$                                                   \\
	$d(N,3)=$ & $18  N - 3  N^2 - N^3 $                 & $0$                                     & $0$                                                   \\ \hline
	\multicolumn{3}{l}{4. case: $j=k$ and $k\neq l$ and $l=m$}                                    &                                                       \\[3pt]
	$C(N,4)=$ & $N(N-1)$                                 & $1$                                     &                                                       \\
	$a(N,4)=$ & $- N^2 + 2  N^3 - N^4 $                 & $-1$                                    & $0$                                                   \\
	$b(N,4)=$ & $6  N - 2  N^2 $                        & $0$                                     & $0$                                                   \\
	$c(N,4)=$ & $- 24  N + 12  N^2 $                    & $0$                                     & $0$                                                   \\
	$d(N,4)=$ & $18  N - 9  N^2 - 2  N^3 + N^4 $        & $1$                                     & $\prod_{i=1}^n[\E(\psi_i(X_i-X_i'))]^2$               \\ \hline
	\multicolumn{3}{l}{5. case:  $|\{j,k,l,m\}|=2$ and three coincide}                            &                                                       \\[3pt]
	$C(N,5)=$ & $4N(N-1)$                               & $4$                                     &                                                       \\
	$a(N,5)=$ & $- N^2 + N^3 $                          & $0$                                     & $0$                                                   \\
	$b(N,5)=$ & $6  N - 4  N^2 $                        & $0$                                     & $0$                                                   \\
	$c(N,5)=$ & $- 24  N + 20  N^2 - 4  N^3 $           & $0$                                     & $0$                                                   \\
	$d(N,5)=$ & $18  N - 15  N^2 + 3  N^3 $             & $0$                                     & $0$                                                   \\ \hline
	\multicolumn{3}{l}{6. case: $|\{j,k,l,m\}|=3$ and ($j=k$ or $l=m$)}                           &                                                       \\[3pt]
	$C(N,6)=$ & $2N(N-1)(N-2)$                          & $\infty $                               &                                                       \\
	$a(N,6)=$ & $- N^2 + N^3 $                          & $0$                                     & $0$                                                   \\
	$b(N,6)=$ & $6  N $                                 & $0$                                     & $0$                                                   \\
	$c(N,6)=$ & $- 24  N + 4  N^2 $                     & $0$                                     & $0$                                                   \\
	$d(N,6)=$ & $18  N - 3  N^2 - N^3 $                 & $0$                                     & $0$                                                   \\ \hline
	\multicolumn{3}{l}{7. case: $j=k=l=m$ }                                                       &                                                       \\[3pt]
	$C(N,7)=$ & $N$                                     & $0$                                     &                                                       \\
	$a(N,7)=$ & $- N^2 + 2  N^3 - N^4 $                 & $-1$                                    & $0$                                                   \\
	$b(N,7)=$ & $6  N - 10  N^2 + 4  N^3 $              & $0$                                     & $0$                                                   \\
	$c(N,7)=$ & $- 24  N + 44  N^2 - 24  N^3 + 4  N^4 $ & $4$                                     & $0$                                                   \\
	$d(N,7)=$ & $18  N - 33  N^2 + 18  N^3 - 3  N^4 $   & $-3$                                    & $0$
\end{tabular} 
\caption{The coefficients for the representation of the variance of the finite
  sample estimator (Theorem \ref{thm:finitemoms}) and the overall limit
  behavior. The coefficient $a(N,k)$ corresponds to the $0$ terms, thus it
  does not show in the actual representation formula.}\label{tab:coef}
\end{table}

\begin{proof}[Proof of Theorem \ref{thm:finitemoms}]
The value of the expectation can be deduced from
\cite[(S.10)]{BoetKellSchi2018a-supp} and \cite[(4.9)]{BoetKellSchi2018a}. We
give here a shorter direct proof. The independence of the $X_i$ and the
linearity of the expectation implies
\begin{equation} \label{eq:ENM}
\E( N \cdot \hN M^2 (\vect{X}^{(1)},\dots, \vect{X}^{(N)})) = \frac{1}{N}
\sum_{j,k=1}^N \prod_{i=1}^n \E( \hN\Psi_i(X_i^{(j)},X_i^{(k)})),
\end{equation}
and, using $\psi_i(0) = 0$ and \eqref{eq:def-PsiN}, gives
\begin{align} 
 \E( \hN\Psi_i(X_i^{(j)},X_i^{(k)})) & =
 (\kronecker{j}{k}-1)\E(\psi_i(X_i-X_i')) +2 \frac{N-1}{N}
 \E(\psi_i(X_i-X_i')) \notag\\
  &\hspace{4.2cm} {} - \frac{N^2-N}{N^2} \E(\psi_i(X_i-X_i'))\notag\\
 &=  (\kronecker{j}{k}-\frac{1}{N})\E(\psi_i(X_i-X_i')), \label{eq:ENPsi}
\end{align}
where $\kronecker{j}{k}$ denotes the Kronecker symbol, i.e.,
$\kronecker{j}{k}= 1$ for $j= k$ and $\kronecker{j}{k} = 0$ for $j\neq k$.
 The observation that in \eqref{eq:ENM} are $N$ summands with $j=k$ and
 $N^2-N$ summands with $j\neq k$ yields the formula for the expectation.

For the second moment we proceed analogously:
\begin{multline*}
\E( [N\cdot \hN M^2(\vect{X}^{(1)},\dots,\vect{X}^{(N)} )]^2)\\ =
\frac{1}{N^2} \sum_{j,k,l,m=1}^N \prod_{i=1}^n \E\left(
\hN\Psi_i(X_i^{(j)},X_i^{(k)}) \hN\Psi_i(X_i^{(l)},X_i^{(m)})\right).
\end{multline*}
Now note that expanding the product in $\E(\hN\Psi_i(X_i^{(j)},X_i^{(k)})
\hN\Psi_i(X_i^{(l)},X_i^{(m)}))$ and using the linearity of the expectation
yields summands of the form
\begin{equation*}
\E\left(\psi_i(X_i^{(j')}-X_i^{(k')})\psi_i(X_i^{(l')}-X_i^{(m')})\right)
= \begin{cases}
  0 & \text{ if $j'=k'$ or $l'=m'$},\\
b_i & \text{ if $|\{j',k'\}\cap \{l',m'\}|=2$},\\
c_i & \text{ if  $j'\neq k'$ and $l'\neq m'$}\\
& \text{ and $|\{j',k'\}\cap \{l',m'\}|=1$},\\
d_i & \text{ if $|\{j',k',l',m'\}|=4$},
\end{cases}
\end{equation*}
with $b_i:=\E(\psi_i(X_i-X_i')^2)$,
$c_i:=\E(\psi_i(X_i-X_i')\psi_i(X_i'-X_i''))$,
$d_i:=[\E(\psi_i(X_i-X_i'))]^2$. Moreover note that the value of
$\E(\hN\Psi_i(X_i^{(j)},X_i^{(k)}) \hN\Psi_i(X_i^{(l)},X_i^{(m)}))$ also depends
on the actual combination of $j,k,l,m$ (which is determined by the outer
sum). A careful analysis shows that one has to distinguish 7 cases of the
outer sum, the frequency of these cases is given by the coefficients $C(N,.)$
and for each case the relevant frequency of the four cases $(0,b_i,c_i,d_i)$
is given by the coefficients $a_i(N,.)$, $b_i(N,.)$, $c_i(N,.)$ and
$d_i(N,.)$, respectively. All coefficients are listed in Table \ref{tab:coef}
on page \pageref{tab:coef}.
\end{proof}

For the distribution-free test of independence proposed in
\cite{BoetKellSchi2018a} (see also \cite{SzekRizzBaki2007} for the classical
case of distance covariance) it was necessary to normalize the test statistic
by the expectation of the limit. It turns out (cf.\
Remark~\ref{rem:methods}.\ref{rem:methods:normalize}) that also in general the
normalized estimators are favorable. In practice the normalization factor
(cf.\ \eqref{eq:Bnormalized}) is estimated using
\begin{equation} \label{eq:defNh}
\hN h(\vect{x}^{(1)},\dots, \vect{x}^{(N)}) := \prod_{i=1}^n \hN
h_i(\vect{x}^{(1)},\dots, \vect{x}^{(N)}):= \prod_{i=1}^n \frac{1}{N^2}
\sum_{j,k=1}^N \psi_i(x_i^{(j)}-x_i^{(k)}).
\end{equation}

Thus for normalized distance multivariance one has to analyze the distribution
of the test statistic $N\cdot \frac{\hN M^{2}}{\hN h}$.

\begin{theorem}[Moments of normalized sample distance
  multivariance] \label{thm:finitemomsnormalized}\hfill
Let $X_1,\dots,X_n$ be independent, non-constant and
$\E(\psi_i(X_i)^2)<\infty$ for $1\leq i\leq n$. Then
\begin{align*}
  \E\Biggl( N \cdot &\frac{\hN M^2 (\vect{X}^{(1)},\dots, \vect{X}^{(N)})}{\hN
    h(\vect{X}^{(1)},\dots, \vect{X}^{(N)})}\Biggr)  = 1 + (N-1)
  \left(\frac{-1}{N-1}\right)^{n}
  = 1 -
  \left(\frac{-1}{N-1}\right)^{n-1},\\
  \E\Biggl( \biggl[N\cdot &\frac{\hN M^2(\vect{X}^{(1)},\dots,\vect{X}^{(N)}
    )}{\hN h(\vect{X}^{(1)},\dots, \vect{X}^{(N)})}\biggr]^2\Biggr)  \\
  &= \frac{1}{N^2} \sum_{k=1}^7 C(N,k) \prod_{i=1}^n \left[\frac{b(N,k)
      b_{i,N}     +     c(N,k)     c_{i,N}     +  d(N,k)d_{i,N}}{N^4}\right],
\end{align*}
where the coefficients $C(N,k),b(N,k),c(N,k),d(N,k)$ are as in the
previous theorem (i.e., they are given in Table \ref{tab:coef} on page
\pageref{tab:coef}). But the other terms have to be modified as follows:
\begin{align*}
b_{i,N}&:=\E\Bigl(\frac{\psi_i(X_i-X_i')^2}{\hN h_i^2(\vect{X}^{(1)},\dots,
\vect{X}^{(N)})}\Bigr),\qquad
c_{i,N}:=\E\Bigl(\frac{\psi_i(X_i-X_i')\psi_i(X_i'-X_i'')}{\hN
  h_i^2(\vect{X}^{(1)},\dots, \vect{X}^{(N)})}\Bigr), \\
d_{i,N}&:=\E\Bigl(\frac{\psi_i(X_i-X_i')\psi_i(X_i''-X_i''')}{\hN
h_i^2(\vect{X}^{(1)},\dots, \vect{X}^{(N)})}\Bigr).
\end{align*}
\end{theorem}  
\begin{remark} \label{rem:finitenormmoms}
  \begin{enumerate}
  \item If at least one random variable is constant, then the moments in
    Theorem~\ref{thm:finitemomsnormalized} can be set to 0 (using the
    convention $0/0 = 0)$.
  \item Note that $b_{i,N},\ c_{i,N}$ and $d_{i,N}$ in Theorem
    \ref{thm:finitemomsnormalized} are more difficult to estimate than those in
    Theorem \ref{thm:finitemoms}. One approach is to estimate each expectation
    by the quotient of the expectations, i.e., $b_i,c_i,d_i$ of Theorem
    \ref{thm:finitemoms} divided by
    \begin{equation*}
      \E(\hN h_i^2( \vect{X}^{(1)},\dots,\vect{X}^{(N)}))=\frac{1}{N^4} (C(N,2)
      b_i + C(N,3) c_i + C(N,1) d_i),
    \end{equation*}
    where $b_i,c_i$ and $d_i$ can be approximated as in
    \eqref{eq:approxa}-\eqref{eq:approxc}. An alternative idea (which did not
    prove useful in our simulations) is to use the estimators given in
    \eqref{eq:approxa}-\eqref{eq:approxc} but replace therein the $B_i$ by the
    normalized distance matrices $\Bskript_i:=\frac{N^2}{|B_i|} B_i$.
  \item Note that the normalizing factor $\hN h$ given in \eqref{eq:defNh} is
    the biased estimator of the limit expectation (cf.\ Corollary
    \ref{cor:estlimitmom}). One could also consider different normalization
    factors, e.g., the unbiased estimator given in Remark \ref{rem:unbiased} or
    the finite sample estimator given in Theorem \ref{thm:finitemoms}. These
    estimators are constant multiples of $\hN h$ (where the constant depends on
    $N$). Thus in these cases the moments of the normalized sample multivariance
    are also just constant multiples of those derived in Theorem
    \ref{thm:finitemomsnormalized}.
  \end{enumerate}
\end{remark}
\begin{proof}[Proof of Theorem \ref{thm:finitemomsnormalized}]
The proof is analogous to the proof of Theorem \ref{thm:finitemoms} just note
that the terms therein are now divided by $\hN h(\vect{X}^{(1)},\dots,
\vect{X}^{(N)})$, which yields that their expectations have a simple
value. Hereto the fact that all summands of $\hN h_i(\vect{X}^{(1)},\dots,
\vect{X}^{(N)})$ are i.i.d., except those which vanish, implies
\begin{equation}\label{eq:Epsino}
\E\left( \frac{\psi_i(X_i^{(j)}-X_i^{(k)})}{\hN h_i(\vect{X}^{(1)},\dots,
    \vect{X}^{(N)})}
\right) = \begin{cases}  0 &\text{for } j = k,\\ \frac{N}{N-1} &\text{for }
  j\neq k.\end{cases}
\end{equation}
For the second moment note that 
\begin{equation*}
1 = \E \left(\frac{\hN h_i^2( \vect{X}^{(1)},\dots,\vect{X}^{(N)})}{\hN
    h_i^2( \vect{X}^{(1)},\dots,\vect{X}^{(N)})} \right)= C(N,2) b_{i,N} +
C(N,3) c_{i,N} + C(N,1) d_{i,N}.
\end{equation*}
But further simplifications seem not possible, therefore the expected values
in each case now depend on $N$.
\end{proof}

\subsection{Sample total multivariance and sample
  $m$-multivariance} \label{sec:total}
 
Let $C_m:= \{S \subset \{1,\dots,n\}\, :\, |S| = m\}$ and $C_\text{total}:=
\bigcup_{m=2}^n C_m$ then
\begin{align}\label{eq:sumtotal}
\hN \overline{M}( \vect{x}^{(1)},\dots,\vect{x}^{(N)}) &:= \sqrt{\sum_{S\in
    C_\text{total}} \hN M_S^2( \vect{x}^{(1)},\dots,\vect{x}^{(N)})}\\
\label{eq:summ}\text{ and } \quad
\hN M_m( \vect{x}^{(1)},\dots,\vect{x}^{(N)}) &:= \sqrt{ \sum_{S\in C_m} \hN
  M_S^2( \vect{x}^{(1)},\dots,\vect{x}^{(N)})}
\end{align}
are the estimators of total- and $m$-multivariance, respectively (for more
details see \cite{Boet2019}). Since they are structurally
identical we drop for the moment the subscript of $C$ and set
$l:=\max\{|S|\,:\, S\in C\}$.

Under the hypothesis of $l$-independence of the $X_i$ -- given
\eqref{eq:con-convergence} -- we have, using the convergence of each summand
(Remark \ref{rem:L2limit}),
\begin{equation} \label{eq:multiClimit}
	\sum_{S \in C} N\cdot \hN M_S^{2}(
        \vect{X}^{(1)},\dots,\vect{X}^{(N)}) \xrightarrow[N\to\infty]{d}
        \sum_{S \in C}  \norm{\pG_S}^2_{\rho_S}.
\end{equation}
Now the independence of $\G_S$ and $\G_R$ for $S\neq R$ (cf.\ Proposition
\ref{pro:moments-ZNlimit}) and Corollary \ref{cor:momG} imply that the moments
of the limit are determined by quantities of the form
\begin{equation}\label{eq:def-gc}
g_C(\vect{\mu}^{(k)}):=\sum_{S\in C} \prod_{i \in S} \mu_i^{(k)}  
\end{equation}
which can be estimated by the corresponding empirical versions
$g_C(\hN \vect{\mu}^{(k)})$.
Note that the function $g_C$ can be given for each case in a computationally
more efficient form (compare with \cite[Eqs.\ (46) and (47)]{Boet2019}):
\begin{eqnarray*}
g_{C_\text{total}}(\vect{\mu}) &:=& \sum_{\substack{S\subset\{1,\dots,n\}\\
  |S| \geq 2}} \prod_{i\in S} \mu_i =  \prod_{i=1}^n(\mu_i + 1) - \sum_{i
  =1}^n \mu_i - 1,\\
g_{C_2}(\vect{\mu})&:=& \sum_{\substack{S\subset\{1,\dots,n\}\\ |S| = 2}}
  \prod_{i\in S} \mu_i  =  \frac{1}{2} \left[\left(\sum_{i=1}^n \mu_i\right
  )^2 - \sum_{i =1}^n \mu_i^2\right], \\
 g_{C_3}(\vect{\mu})&:=&\sum_{\substack{S\subset\{1,\dots,n\}\\ |S| = 3}}
  \prod_{i\in S} \mu_i   \\
  &=& \frac{1}{6} \Biggl[\left(\sum_{i=1}^n \mu_i\right
  )^3 - 3 \left(\sum_{i =1}^n \mu_i\right)\left(\sum_{i =1}^n \mu_i^2 \right)
   + 2  \sum_{i =1}^n \mu_i^3\Biggr].
\end{eqnarray*}
Thus we have proved the following extension of Corollary \ref{cor:estlimitmom}.

\begin{corollary} \label{cor:estlimitmmom}
Let $\E(\psi_i(X_i))<\infty$ for $1\leq i\leq n$ and denote by
$\pH_C:= \sum_{S \in C} \norm{\pG_S}^2_{\rho_S}$ the distributional limit in
\eqref{eq:multiClimit} under the assumption of $l$-independence of the random
variables, with $l:=\max\{|S|\,:\, S\in C\}$. Then
\begin{eqnarray*}
g_C(\hN \vect{\mu}^{(1)}) &\xrightarrow{N\to \infty}& \E(\pH_C),\\
2g_C(\hN \vect{\mu}^{(2)}) &\xrightarrow{N\to \infty}& \V(\pH_C),\\
8g_C(\hN \vect{\mu}^{(3)}) &\xrightarrow{N\to \infty}& \E((\pH_C-\E(\pH_C))^3),\\
48g_C(\hN \vect{\mu}^{(4)}) + 12 g_C((\hN \vect{\mu}^{(2)})^{\circ 2})
                          &\xrightarrow{N\to \infty}& \E((\pH_C-\E(\pH_C))^4),
\end{eqnarray*} 
where $\hN \vect{\mu}^{(k)} = (\hN \mu_1^{(k)},\dots,\hN \mu_n^{(k)})$ and
$\hN \mu_i^{(k)}$ are the biased (Corollary~\ref{cor:estlimitmom}) or unbiased
(Remark~\ref{rem:unbiased}) estimators for $\mu_i^{(k)}$. Here
$\vect{\mu}^{\circ 2}$ denotes the Hadamard power of the vector, i.e., each
component is squared. The convergence is almost sure convergence (replacing
samples by the corresponding random variables).
\end{corollary}

\begin{remark}\label{rem:clttest}
\begin{enumerate}
\item Proposition \ref{prop:replimit} yields also the representation
  \begin{equation}\label{seems-important-gets-a-number}
    \sum_{S\in C} \norm{\pG_S}^2_{\rho_S} =\iint 
    g_C\bigl(\vect{\Psi}(x,y)\bigr)\,d\W(\vect{F}_\vect{X}(x))
    d\W(\vect{F}_\vect{X}(y))\ \ \ \Prob\text{-a.s.,}
  \end{equation} 
where $\vect{\Psi} = (\Psi_1,\dots,\Psi_n)$ with $\Psi_i$ defined in
\eqref{eq:def-Psi}.
\item Sample total- and $m$-multivariance can be standardized by
  \begin{equation} \label{eq:mstand}
    \frac{\sum_{S\in C} N \cdot \hN M^2_{\rho_S} - g_C(\hN
      \vect{\mu}^{(1)})}{\sqrt{2g_C(\hN \vect{\mu}^{(2)})}}.
  \end{equation}
  Note that by Proposition \ref{pro:moments-ZNlimit} the sums in
  \eqref{eq:sumtotal} and \eqref{eq:summ} are (in the limit $N\to \infty$) 
  sums of independent random variables. Thus in the case of many summands
  (e.g., $n$ large) one could also try to use the central limit theorem to
  determine its standardized distribution.
  At least in the case of $m$-multivariance with independent identically
  distributed marginals the limit is standard normally distributed (as $n\to
  \infty$ and $N\to \infty$), since in this case all summands of the
  $m$-multivariance are independent and identically distributed. Naturally,
  this yields a test for $m$-independence based on the standard normal
  distribution, see \eqref{eq:clt}. In fact this is related to the test of
  pairwise independence presented in \cite{YaoZhanShao2017}. But they use the
  square roots (with appropriated sign) of the unbiased distance covariance.
\end{enumerate}
\end{remark}

For the finite sample case we need analogous to \eqref{eq:def-gc} the
following function defined for vectors $\vect{u},\vect{v},\vect{w}\in\R^n$
\begin{equation} \label{eq:def-GC}
G_C(\vect{u},\vect{v},\vect{w}) := \sum_{S \in C} \sum_{S' \in C}
\Biggl[\prod_{i \in S\backslash S'} u_i \Biggr] \cdot \Biggl[\prod_{i \in
  S'\backslash S} v_i \Biggr] \cdot \Biggl[\prod_{i \in S\cap S'} w_i
\Biggr].
\end{equation}
For the sets of interest $G_C$ has the following (numerically tractable)
representations
\begin{align*}
G_{C_\text{total}}&(\vect{u},\vect{v},\vect{w}) :=  \sum_{\substack{k_1,k_2 =
    1\\ k_i \text{distinct}}}^n u_{k_1}v_{k_2} + \sum_{i=1}^n (u_i + v_i
+w_i)+ \prod_{i=1}^n (u_i + v_i + w_i +1) \\
& 
- \prod_{i=1}^n(1+v_i) \left(1 + \sum_{i=1}^n \frac{u_i+w_i}{1+v_i}\right) -
\prod_{i=1}^n(1+u_i) \left(1 + \sum_{i=1}^n \frac{v_i+w_i}{1+u_i}\right) + 1,\\
G_{C_2}(\vect{u},&\vect{v},\vect{w}) :=\frac{1}{2} \sum_{\substack{k_1,k_2 =
    1\\ k_i \text{distinct}}}^n w_{k_1}w_{k_2} + \sum_{\substack{k_1,k_2,k_3 =
    1\\ k_i \text{distinct}}}^n u_{k_1}v_{k_2}w_{k_3}\\
&  + \frac{1}{4}
\sum_{\substack{k_1,\dots,k_4 = 1\\ k_i \text{distinct}}}^n
u_{k_1}u_{k_2}v_{k_3}v_{k_4}, \\
G_{C_3}(\vect{u},&\vect{v},\vect{w}) := \frac{1}{6}
\sum_{\substack{k_1,k_2,k_3 = 1\\ k_i \text{distinct}}}^n
w_{k_1}w_{k_2}w_{k_3} + \frac{1}{2}\sum_{\substack{k_1,\dots,k_4 = 1\\ k_i
    \text{distinct}}}^n u_{k_1}v_{k_2}w_{k_3}w_{k_4}\\
 &+ \frac{1}{4} \sum_{\substack{k_1,\dots,k_5 = 1\\ k_i \text{distinct}}}^n
 u_{k_1}u_{k_2}v_{k_3}v_{k_4}w_{k_5}  + \frac{1}{36}
  \sum_{\substack{k_1,\dots,k_6 = 1\\ k_i \text{distinct}}}^n
  u_{k_1}u_{k_2}u_{k_3}v_{k_4}v_{k_5}v_{k_6}.
\end{align*}
Using $G_C$ we get the following extension of Theorem \ref{thm:finitemoms}.
\begin{theorem}[Moments of sample total and
  m-multivariance]\label{cor:mfinitemom}\hfill{}\\  
Let $\E(\psi_i(X_i)^2)<\infty$ for $1\leq i\leq n$ and
$X_1,\dots,X_n$ be $l$-independent with $l:= \max\{|S|\,:\, S\in C\}$. Then
\begin{align}
\E&\left(\sum_{S\in C} N \cdot \hN M^2_{\rho_S}(\vect{X}^{(1)},\dots,
  \vect{X}^{(N)})\right) \label{eq:mfinitemomE}\\
\notag& \quad  \quad =
g_C\left(\left(1-\frac{1}{N}\right)\vect{\mu}^{(1)}\right) + (N-1)
g_C\left(\left(-\frac{1}{N}\right) \vect{\mu}^{(1)}\right),\\
\label{eq:mfinitemomEsq}\Biggl[\E& \left(\sum_{S\in C} N \cdot \hN
  M^2_{\rho_S}(\vect{X}^{(1)},\dots, \vect{X}^{(N)})\right) \Biggr]^2 \\
\notag & \quad \quad  = \frac{1}{N^2} \sum_{k=1}^7 C(N,k)
G_C\left(e_1(N,k)\vect{\mu}^{(1)},e_2(N,k)\vect{\mu}^{(1)},
  e_1(N,k)e_2(N,k)\vect{d}\right),\\
\label{eq:mfinitemomE2}\E&\left[ \left(\sum_{S\in C} N \cdot \hN
    M^2_{\rho_S}(\vect{X}^{(1)},\dots, \vect{X}^{(N)})\right)^2 \right] \\
\notag & \quad \quad  = \frac{1}{N^2} \sum_{k=1}^7 C(N,k)
G_C\biggl(e_1(N,k)\vect{\mu}^{(1)},e_2(N,k)\vect{\mu}^{(1)},\\
\notag &\hspace{6.5cm} \frac{b(N,k) \vect{b} + c(N,k) \vect{c}
  + d(N,k) \vect{d}}{N^4}\biggr).
\end{align}
The coefficients $C(N,k),b(N,k),c(N,k),d(N,k)$ are as in
Theorem \ref{thm:finitemoms} (i.e., they are given in Table \ref{tab:coef} on
page \pageref{tab:coef}). Moreover,  $\vect{b} := (b_1,\dots,b_n)$, the vectors
$\vect{c},\vect{d}$ are defined analogously and
\begin{align*}
e_1(N,k) &:= \begin{cases} -\frac{1}{N}& \text{for } k = 1,2,3,\\
  1-\frac{1}{N}& \text{for } k = 4,5,6,7, \end{cases}\\
e_2(N,k) &:= \begin{cases} -\frac{1}{N}& \text{for } k = 1,2,3,5,6,\\
  1-\frac{1}{N}& \text{for } k = 4,7. \end{cases}
\end{align*}
\end{theorem}
\begin{remark}
The three expectations in the above theorem are given in such a way that
plugging in unbiased estimators of the parameters
$\vect{\mu}^{(1)},\vect{b},\vect{c},\vect{d}$ the formulas provide an unbiased
estimate of the left hand side. This would not be the case if we just use the
square of \eqref{eq:mfinitemomE} for \eqref{eq:mfinitemomEsq}.
\end{remark}
\begin{proof}[Proof of Theorem \ref{cor:mfinitemom}]
Equation \eqref{eq:mfinitemomE} is a direct consequence of
\eqref{eq:finitemomEalt} and \eqref{eq:def-gc}.

For \eqref{eq:mfinitemomE2} note that
\begin{equation*}
\E\left[ \left(\sum_{S\in C} N \cdot \hN M^2_{\rho_S}\right)^2 \right]
=\sum_{S\in C}\sum_{{S'\in C}} \E(N \cdot \hN M^2_{\rho_S}\cdot N \cdot \hN
M^2_{\rho_{S'}})
\end{equation*}
and
\begin{equation*}
\begin{split}
\E( &\hN M^2_{\rho_S} \hN M^2_{\rho_{S'}}) = \sum_{j,k = 1}^N \sum_{l,m=1}^N
\E\left[ \prod_{i \in S} \prod_{i' \in S'} \hN\Psi_i(X_i^{(j)},X_i^{(k)})\cdot
  \hN\Psi_{i'}(X_{i'}^{(l)},X_{i'}^{(m)})\right]\\
&= \sum_{j,k = 1}^N \sum_{l,m=1}^N \left[ \prod_{i \in S\backslash S'} \E
  \left(\hN\Psi_i(X_i^{(j)},X_i^{(k)})\right)\right] \cdot  \left[\prod_{i \in
    S'\backslash S}  \E \left(\hN\Psi_i(X_i^{(l)},X_i^{(m)})\right)\right]\\
&\phantom{ = \sum_{j,k = 1}^N \sum_{l,m=1}^N } \cdot \left[\prod_{i \in S\cap
    S'} \E \left(\hN\Psi_i(X_i^{(j)},X_i^{(k)})\cdot
    \hN\Psi_i(X_i^{(l)},X_i^{(m)})\right)\right].
\end{split}
\end{equation*}
The last factor is computed as in Theorem \ref{thm:finitemoms}. The first and
second factor can be simplified using \eqref{eq:ENPsi} and considering the
values of $j,k,l,m$ according to the cases of Table \ref{tab:coef} to get the
factors $e_1(N,.)$ and $e_2(N,.)$. Here note that $j=k$ and $l=m$ for the
cases $1.,2.,3.$, and  $j\neq k$ and $l\neq m$ for the cases $4.,7.$. In the
cases $5.,6.$ half of the summands satisfy $j=k$ and $l\neq m$, for the other
half is $j\neq k$ and $l=m$ thus it seems that one should consider these
subcases, but due to symmetry of the sums $\sum_{S\in C}\sum_{{S'\in C}}$ it
simplifies to \eqref{eq:mfinitemomE2}.

Equation \eqref{eq:mfinitemomEsq} is proved analogously, just replacing 
in the above calculations the
factors \\$  \E \left(\hN\Psi_i(X_i^{(j)},X_i^{(k)})\cdot
    \hN\Psi_i(X_i^{(l)},X_i^{(m)})\right)$ by  the new factors  $\E
\left(\hN\Psi_i(X_i^{(j)},X_i^{(k)})\right)\E\left(
  \hN\Psi_i(X_i^{(l)},X_i^{(m)})\right)$.
\end{proof}

To formulate the analogous result for normalized multivariance define
\begin{align}\label{def:Nhs}
\hN h_S(\vect{x}^{(1)},\dots, \vect{x}^{(N)}) &:= \prod_{i\in S} \hN
h_i(\vect{x}^{(1)},\dots, \vect{x}^{(N)})  \notag\\
& := \prod_{i \in S} \frac{1}{N^2} \sum_{j,k=1}^N \psi_i(x_i^{(j)}-x_i^{(k)}).
\end{align}
\begin{theorem}[Moments of normalized sample total and
  m-multivariance]\label{cor:mfinitemomnormal}\hfill\\
  Let $X_1,\dots,X_n$ be $l$-independent and non-constant with $l:= \max\{|S|\,:\, S\in C\}$ and $\E(\psi_i(X_i)^2)<\infty$ for $1\leq i\leq n$. Then
\begin{align}
    \label{eq:mfinitemomEnormal}\E\biggl(\sum_{S\in C} & N \frac{\hN
      M^2_{\rho_S}}{\hN h_S}\biggr) = g_C(\vonen)+(N-1)
    g_C\left(-\frac{1}{N-1}\vonen\right),\\
\label{eq:mfinitemomE2normal} \E\Biggl[ \biggl(\sum_{S\in C} &N \frac{\hN
      M^2_{\rho_S}}{\hN h_S}\biggr)^2 \Biggr] = \frac{1}{N^2}
    \sum_{k=1}^7 C(N,k) G_C\biggl(f_1\vonen,f_2\vonen,
 \frac{b\ \vect{b}_N + c \vect{c}_N + d \vect{d}_N}{N^4}\biggr),
\end{align}
where the random variables are omitted in the notation on the left hand
side and for $f_1,f_2,b,c,d$ the arguments $(N,k)$ are omitted. The
coefficients $C(N,k), b, c, d$ are as in Theorem
\ref{thm:finitemomsnormalized} (i.e., they are given in Table \ref{tab:coef}
on page~\pageref{tab:coef}). Moreover, $\vonen := (1,\dots,1)\in \R^n,$
$\vect{b}_N := (b_{1,N},\dots,b_{n,N})$, the vectors $\vect{c}_N,\vect{d}_N$
are defined analogously and
\begin{align*}
  f_1(N,k) &:= \begin{cases} -\frac{1}{N-1}& \text{for } k = 1,2,3,\\
    1& \text{for } k = 4,5,6,7,\end{cases}\\
  f_2(N,k) &:= \begin{cases} -\frac{1}{N-1}& \text{for } k = 1,2,3,5,6,\\
    1& \text{for } k = 4,7. \end{cases}
\end{align*}
\end{theorem}
\begin{proof}
  Note that by \eqref{eq:Epsino}
  \begin{equation*}
    \E\left( \frac{\hN\Psi_i(X_i^{(j)},X_i^{(k)})}{\hN
        h_i(\vect{X}^{(1)},\dots, \vect{X}^{(N)})}
    \right) = \left(\kronecker{j}{k} - \frac{1}{N}\right) \frac{N}{N-1}
    = \begin{cases}  1 &\text{for } j = k,\\ -\frac{1}{N-1} &\text{for } j\neq
      k.\end{cases}
  \end{equation*}
  Now, the proof is analogous to the proof of Theorem \ref{cor:mfinitemom}. 
\end{proof}

\begin{remark}
  \begin{enumerate}
  \item If some of the random variables $X_{1},\ldots,X_{n}$ are constant,
    then \eqref{eq:mfinitemomEnormal} and \eqref{eq:mfinitemomE2normal} remain
    valid if the corresponding components of $\vonen$ are replaced by 0.
  \item To avoid confusion, note that depending on the number of elements in
    $C$, here denoted by $|C|$, normalized sample multivariance is by definition
    \begin{equation*}
      \frac{1}{|C|} \cdot \sum_{S \in C} N \frac{\hN
        M_{\rho_S}^2(\vect{x}^{(1)},\dots, \vect{x}^{(N)})}{\hN
        h_S(\vect{x}^{(1)},\dots, \vect{x}^{(N)})}.
    \end{equation*}
    Thus the values in Theorem \ref{cor:mfinitemomnormal} have to be scaled by
    $\frac{1}{|C|}$ and $\frac{1}{|C|^2}$, respectively.
  \end{enumerate}
\end{remark}

\subsection{Testing independence using distance multivariance} \label{sec:tests}

In order to use distance multivariance for independence tests of the random
vectors $X_1,\dots,X_n$ the moment conditions given in Remark
\ref{rem:momcond} have to hold. Depending on the type of multivariance the
test statistic $T$ and the corresponding additional assumptions are given in
the following table.
\medskip

{\renewcommand{\arraystretch}{1.5} 
\noindent\begin{tabular}{l|l|c|l}
  & test statistic & assumption on $X_i$ & $H_0$ ($X_i$ are )\\
\hline
multivariance & $N \cdot \hN M^2$ & $(n-1)$-independent  & independent
  \\\hline
total multivariance & $N \cdot \hN \overline{M}^2$ &  --- & independent \\
\hline
$m$-multivariance & $N \cdot \hN M_m^2$ & $(m-1)$-independent  & $m$-independent
\end{tabular}}

\medskip
Moreover one can also consider the corresponding normalized test statistics,
i.e., replacing distance multivariance $M$ by normalized distance
multivariance $\Mskript$.

It is known by \cite[Thm.\ 4.5, 4.10, Cor.\ 4.16, 4.18]{BoetKellSchi2018a} and
\cite[Thm. 2.5, 5.2 and 8.3]{Boet2019} that under these assumptions the test
statistic $T$ diverges to $\infty$ for $N \to \infty$ if and only if $H_0$ is
violated. Thus it is standard to define the corresponding tests as follows,
and due to the divergence property these tests are  consistent against all
alternatives for $N\to \infty$ (under the stated assumptions).
\begin{test*}[for a test statistic $T$ which diverges to $\infty$ if and only
  if $H_0$ is violated]
  Let $\vect{x}^{(i)}$, $1\leq i\leq N$, be samples of $\vect{X} =
  (X_1,\dots, X_n)$. Then a test of $H_0$ with significance level
  $\lambda$ is given by rejecting $H_0$ if the p-value of the sample,
  i.e., $ \Prob(T(\vect{X}_{H0}^{(1)}, \dots, \vect{X}_{H0}^{(N)})
  \geq T(\vect{x}^{(1)},\dots,\vect{x}^{(N)})),$ is less than
  $\lambda$. Here for each $i$ the vector $\vect{X}_{H0}^{(i)}$ is
  distributed as $\vect{X}$ under $H_0$, i.e., with components
  satisfying $H_0$.
\end{test*} 

To actually perform such tests one has to compute or estimate the p-value for
the given sample, each method constitutes a different test with its own
empirical power and empirical size. There are various methods, which we
collect here in some detail for the convenience of the reader and also in
order to have a reference for the examples and comparisons in the next
section.

\begin{description}
\item[I. Quadratic form estimates:] By Theorem \ref{thm:ZNBMlimit},
  Proposition \ref{pro:moments-ZNlimit}, Remark \ref{rem:L2limit} and Equation
  \eqref{eq:multiClimit} the distributional limit of the test statistics given
  in the above table can be written as the $L^2(\rho)$-norm of a centered,
  complex, hermitian Gaussian random field $\pG$ and hence as a Gaussian
  quadratic form $Q = \sum_{i \in \N} \alpha_i Z_i^2$,
  cf.~Theorem~\ref{thm:repNormGauss}. This also holds for
  $m$- and total multivariance, since their summands are independent
  (cf. Proposition~\ref{pro:moments-ZNlimit}).
  
  \begin{description}
  \item[I.a Moment methods:] The p-value is estimated based on some of
    the moments of the quadratic form. In the following $Y_r$ denotes a
    chi-squared distributed random variable with (possibly fractional)
    $r$ degrees of freedom, for details see
    Remark~\ref{rem:tail}.\ref{rem:tail:chi}.
    \begin{itemize}
    \item The \textbf{classical estimate} of Sz\'ekely and Bakirov
      \cite{SzekBaki2003} uses only the mean:
      \begin{equation} \label{eq:c1}
        \Prob(Q \geq x) \leq \Prob\left(Y_1 \geq \frac{x}{\E(Q)}\right).
      \end{equation}
      This method is only valid for $x/\E(Q)\geq  1.5365$ or equivalently
      for p-values less than 0.215 (which is sufficient for any commonly
      used significance level). In this setting it is the simplest and most
      unrestrictive approach, since the mean always exists under the basic
      assumptions (cf. Remark \ref{rem:momcond}). In the case of univariate
      Bernoulli marginals it is sharp for multivariance (Remark
      \ref{rem:bernoulli-sharp}) but in general it is (very) conservative
      (e.g., Example \ref{ex:mv_bern}).
    \item The \textbf{variance based estimate} derived in Theorem \ref{thm:tail}
      which uses the mean and variance:
      \begin{equation} \label{eq:cv}
	\Prob(Q \geq x) \leq \Prob\left(\alpha Y_\frac{1}{\alpha} \geq
          \frac{x}{\E(Q)}\right)\text{ with }\alpha = \sqrt{\frac{\Var{Q}}{2
            \E(Q)^2}}.
      \end{equation}
      It is also only valid for p-values less than 0.215, see Remark
      \ref{rem:tail}.\ref{rem:tail:x0}.
    \item \textbf{Pearson's estimate}, see, e.g., \cite{Imhof1961},
      uses mean, variance and skewness:
      \begin{equation} \label{eq:pearson}
        \begin{split}
	\Prob(Q \geq x) &\approx  \Prob\left(\beta Y_\frac{1}{\beta^2} -
          \frac{1}{\beta} \geq \sqrt{2}\cdot
          \frac{x-\E(Q)}{\sqrt{\Var{Q}}}\right) \\ &\text{ with }\beta =
        \frac{\skw(Q)}{\sqrt{8}}.
      \end{split}
    \end{equation}
      Among the above moment methods this is the most powerful, see, e.g.,
      Example \ref{ex:mv_bern}. This can be reformulated as estimating the
      quadratic form by the Pearson Type III distribution with the same mean,
      variance and skewness as $Q$ (cf.\ Section \ref{sec:pearson} in the
      Appendix).
    \item \textbf{LTZ's estimate}  (Liu et
      al.~\cite{LiuTangZhan2009}) chooses under all (non-central)
      chi-squared distributions the one having the same first three
      moments as the quadratic form and minimizing the absolute error in
      the fourth moment. In our case, i.e., for quadratic forms of centered
      random variables their method reduces to Pearson's 
      three-moment approach.
    \end{itemize}
  \item[I.b Eigenvalue methods:]
    The p-value is computed from the distribution function
    of a finite-dimensional quadratic form given by an approximation of
    (some of) the coefficients $\alpha_i$ of the Gaussian quadratic form
    $Q = \sum_{i\in\N} \alpha_i Z_i^2$. The coefficients $\alpha_i$
    arise as eigenvalues of the integral operator $T_\covG$ associated
    with the covariance kernel         
    \begin{equation*}
      \covG(s,t) = \prod_{i=1}^n \covG_i(s_i,t_i) = \prod_{i=1}^n \left
        (f_i(s_i-t_i)-f_i(s_i)\conj{f_i(t_i)}\right),
    \end{equation*}
    (see Theorem~\ref{thm:repNormGauss}) which can be computed
    from the characteristic functions $f_i$ of the marginals (if known) or
    estimated by the empirical characteristic functions $\hN f_i$ based
    on the samples. Due to the product structure of $\covG$ one can solve
    the eigenvalue problem for the kernels $\covG_i$, separately
    (cf.\ Remark~\ref{rem:operator:product}).
    A standard technique for integral operators is the \textbf{Nystr\"om
      method}, cf.\ \cite{Atkinson2005}: The integral is discretized
    using a suitable numerical quadrature scheme, i.e.,
    \begin{equation*}
      T_{\covG_i}u(s) = \int_{\R^{d_i}} \covG_i(s,t)u(t)\,\rho_i(dt)
      \approx  \sum_{j=1}^m \covG_i(s,t_j)u(t_j)w_j
    \end{equation*}
    for some order $m\in\N$, mutually different nodes
    $t_1,\dotsc,t_m\in\R^{d_i}$ and weights $w_1,\dotsc,w_m\ge 0$ and the
    eigenvalues $\alpha_j^{(i)}$ of $T_{\covG_i}$ are approximated by computing the
    eigenvalues of the (symmetrized) matrix
    \begin{equation*}
      \bigl(\sqrt{w_j}\covG_i(t_j,t_k)\sqrt{w_k}\bigr)_{j,k=1,\dotsc,m}.
    \end{equation*}         
    The p-value is then computed, based on the estimated coefficients
    $\alpha_i$, either by
    \textbf{numerical approximation of series representations} of the
    distribution function (e.g.,
    \cite{Farebrother1984,Tzir1987}) or \textbf{numerical Fourier
      inversion} of the distribution function (e.g.,
    \cite{Davies1980,Imhof1961}).
  \end{description}
  
\item[II. Central limit theorem:] For identically distributed marginals
  the summands of $m$-multivariance are (in the limit under $H_0$)
  independent and identically distributed random variables (see
  Remark~\ref{rem:clttest}), thus by the central limit theorem the
  quadratic form is (in the limit) normally distributed. This yields the
  approximation
  \begin{equation} \label{eq:clt}
    \Prob(Q \geq x) \approx \Prob\left(Z  \geq
      \frac{x-\E(Q)}{\sqrt{\Var{Q}}}\right),
  \end{equation}
  where $Z$ is a standard normally distributed random variable. 
\item[III. Sampling methods:] The p-value is computed 
  by evaluating an
  empirical distribution function which is obtained either using
  \textbf{Monte Carlo simulations} of the distribution under $H_0$ or
  by resampling from the sample with replacement (\textbf{bootstrap})
  or without replacement (\textbf{permutation}), see \cite{Boet2019}
  for details in the context of distance multivariance. In general
  these methods are slow, since for the estimate of the p-value of a
  given sample the test statistic has to be evaluated for many samples.
\end{description}

In the next section we will compare these methods using various examples. But
before note that for most of these methods some parameters (in our case mostly
moments) have to be known or estimated from the sample. 
In fact, the previous sections provide many \textbf{methods to estimate the
  required moments}, and also each of these constitutes a different test
with its own empirical power and empirical size. Here we briefly comment
on the available options: 

\begin{remark}[On choosing the estimators for the moment
  methods]\label{rem:methods} \hfill
\begin{enumerate}
\item  Using the moments of the \textbf{limit} (Corollaries
  \ref{cor:estlimitmom} and \ref{cor:estlimitmmom}) or the \textbf{finite
    sample} versions (Section \ref{sec:finitesample}, Corollaries
  \ref{cor:mfinitemom} and  \ref{cor:mfinitemomnormal}). The latter clearly
  provide a more precise description of the distribution of the test
  statistic. But note that only in the limit the test statistics are really
  Gaussian quadratic forms. Nevertheless, the examples of the next section
  indicate that the use of the finite sample versions is always
  recommended. One should also note that currently the approximation of finite
  sample moments requires $\psi_i$-moments of order 2, whereas the estimation
  of the limit moments requires only  $\psi_i$-moments of order 1.
\item  Using \textbf{unbiased} or \textbf{biased} estimators (Remarks
  \ref{rem:unbiased} and \ref{rem:finitemoms}). In general unbiased estimators
  help to prevent systematic errors, thus their use is (if available)
  recommended. But keep in mind that a functional transformation of an
  unbiased estimator is usually not unbiased anymore, thus the method as a
  whole might be still biased.
\item \label{rem:methods:normalize} Using standard \textbf{distance
    multivariance} or \textbf{normalized distance multivariance} (Remark
  \ref{rem:limitmom}.\ref{rem:limitmomnormalized}, Theorems
  \ref{thm:finitemomsnormalized} and \ref{cor:mfinitemomnormal}). For
  multivariance the tests with and without normalization can differ since
  without normalization different marginals might have (depending on their
  scale) different influence on the value. For $m$- and total multivariance
  the normalization matters also in a second way: With normalization
  each summand of $m$- and total multivariance has (under $H_0$) the same
  expected value, but without normalization this is not necessarily the
  case. Thus for $m$- and total multivariance with not identically distributed
  marginals it is certainly recommended to use the normalized version. 
  Also for the other cases it seems reasonable to use normalized multivariance
  since this appears to be more robust (Example~\ref{ex:robust}) and for
  $\psi_i= |\cdot|^{\beta_i}$ with $\beta_i\in(0,2]$ it is scale invariant
  (cf.\ the end of the introduction to Section \ref{sec:multivariance}).
\end{enumerate}
\end{remark}

\begin{remark}[Bias vs.\ conservative] \label{rem:bias_vs_conservative} For the
  moment methods given above an overestimation (i.e., a value larger than the
  true value) of the mean $m$ and of the variance $v$ results in a larger
  (thus more conservative) p-value. In this sense a positive bias (in contrast
  to a negative bias) is preferred for the parameter estimation.

 The behavior of \eqref{eq:pearson} for a biased skewness is much more
 involved and depends also on the actual value of
 $\sqrt{2}(x-\E(Q))\V(Q)^{-1/2}$.
\end{remark}

\begin{remark}[\eqref{eq:c1} is sharp for multivariance with Bernoulli
  marginals] \label{rem:bernoulli-sharp}\hfill\\
In the case that $X_j$, $j\in S$, are independent and identically Bernoulli
distributed with parameter $\tfrac{1}{2}$ one finds from the characteristic
function  $f_j(t_j) = \frac{1}{2}(\ee^{\ii t_j}+1)$ and
Proposition~\ref{pro:moments-ZNlimit} that for sample multivariance the
limit random field $\pG_S$ defined in \eqref{eq:ZNBMlimit2} has covariance
kernel
\begin{align*}
\covG_S(t_S,t_S') &= \E(\pG_S(t_S)\conj{\pG_S(t_S')}) = \prod_{j\in S}
\left(f_j(t_j-t_j') - f_j(t_j)f_j(-t_j')\right) \\ &= \prod_{j\in S}
\frac{1}{4}(\ee^{\ii t_j}-1)(\ee^{-\ii t_j'}-1) =: g_S(t_S)\cdot\conj{g_S(t_S')}.
\end{align*}
Therefore, the associated operator $T_{\covG_S}$, see
\eqref{eq:operator}, has rank one and $\norm{g_S}_{\rho_S}^2$ is its only
non-vanishing eigenvalue. Thus, 
\begin{equation*}
  \pG_S(t_S) = \norm{g_S}_{\rho_S}g_S(t_S)\cdot Z,
\end{equation*}
for some standard normally distributed random variable $Z$ and the associated
Gaussian quadratic form follows a scaled chi-squared distribution,
\begin{equation*}
  \norm{\pG_S}_{\rho_S}^2 = \norm{g_S}_{\rho_S}^2 Y_1,
\end{equation*}
cf.~Lemma~\ref{lem:repG} and Theorem~\ref{thm:repNormGauss}. In
particular, for the limit distribution in the normalized case, the
bound \eqref{eq:c1} is sharp.
\end{remark}


%% file: comparison.tex
\label{sec:examples} 
In this section we collect examples which illustrate various aspects of our
results. Of major interest is certainly a comparison of the empirical power
and empirical size of the tests introduced in Section \ref{sec:tests}. For
a basic example which distinguishes the methods see Example
\ref{ex:normal_tetrahedron} and for a comprehensive study see Example
\ref{ex:all}. The first step to the estimation of the p-values is the
estimation of the parameters corresponding to the marginal distributions, the
estimation is either based on the samples
(Example~\ref{ex:marginal_mom_sample}) or done a priori in the case of known
marginals (Example~\ref{ex:momknownmarginals}). Thereafter the joint moments
(i.e., the actually required parameters) can be estimated
(Example~\ref{ex:jointmom}) and finally the tests can be performed.

Further examples discuss the distribution of the test statistic for
multivariate Bernoulli marginals (Example \ref{ex:mv_bern}), the classical
estimate \eqref{eq:c1} using various moment estimators (Example \ref{ex:c1}),
the speed of the moment estimation vs.\ the resampling approach (Example
\ref{ex:speed}), the true finite sample distribution of the test statistic
(Example \ref{ex:not_gQF}) and the robustness of the methods (Example
\ref{ex:robust}).

For the sake of clarity, let us summarize and recall the framework. We
consider random variables $X_i$ with values in $\R^{d_i}$ and symmetric
measures $\rho_i$ on $\R^{d_i}$  such that $\int 1\land
\abs{t_i}^2\rho_i(dt_i)<\infty$, $1\leq i\leq n$. Based on $N$ samples
$\vect{x}^{(1)},\dotsc,\vect{x}^{(N)}$ of $\vect{X}=(X_1,\dotsc,X_n)$, we are
interested in the test statistics discussed in Section \ref{sec:tests}, e.g.,
$N\cdot \hN M_{\rho}^2(\vect{X}^{(1)},\ldots,\vect{X}^{(N)})$ with
$\rho=\otimes_{i=1}^n \rho_i$.

In this setting the key parameters are $\mu_i^{(k)}$ as defined in Corollary
\ref{cor:momG} and $b_i,c_i,d_i$ defined in Remark~\ref{rem:unbiased}. To
estimate these and to use them in further computations
we refer to the options
summarized in Section \ref{sec:tests} (in particular Remark \ref{rem:methods};
the key terms which we use here are mostly printed in bold in that section).

There is an infinite choice of possible examples, we try to concentrate on
some key aspects. Thus as sample distributions we mostly consider the
Bernoulli distribution (an extremal distribution for our setting in the sense
of Remark \ref{rem:bernoulli-sharp}), the uniform distribution (in some sense
this is the other \textit{extremal}, see \cite[Example 7.10, case
$n=5$]{Boet2019}) and the normal distribution (which is a standard assumption
for samples).
For the tests, there are two types of examples:
\begin{itemize}
\item $H_0$ examples: The marginals satisfy $H_0$ and the empirical size of
  the test is of major interest. It should be close to or smaller than
  (i.e., conservative tests) the significance level.
\item dependence examples: The marginals violate $H_0$ and the empirical power
  of the test is of major interest. It should be large -- but not larger than
  one could expect based on the true distribution of the test statistic under
  $H_0$ (cf.\ robustness discussed in Example~\ref{ex:robust}).
\end{itemize}

If not mentioned otherwise we use in the examples the following conventions: 
Simulations are based on 10000 samples, the tests are performed with
significance level 0.05, the \textit{benchmark} (true) p-value is computed by
the empirical distribution function of a Monte Carlo sample of the test
statistic under $H_0$. 
The measures
$\rho_i$ are such that the functions $\psi_i(x_i):= \int_{\R^{d_i}} 1-
\cos(x_i\cdot t_i)\,\rho_i(dt_i)$, cf.\ \eqref{def:sm-cdms}, are the Euclidean
distance on $\R^{d_i}$, i.e., $\psi_i(y) = |y|$ for $y\in \R^{d_i}$. For
general examples with other distances (but without the moment based tests) see
\cite{BoetKellSchi2018a} and \cite{Boet2019}.

For the construction of examples with higher order dependence we briefly
recall a classical example: A dice in the shape of a tetrahedron
(e.g., \cite[Example 7.1]{Boet2019}) with sides colored $r$,$g$,$b$ and the
forth side has all three colors on it. For each color define a Bernoulli
random variable $Y_i$ which is 1 if and only if the corresponding color shows
(at the bottom of the tetrahedron) after a throw of this dice. The three
random variables $(Y_1,Y_2,Y_3)$ are dependent but pairwise independent.

\begin{example}[Comparison of the moment methods -- normal tetrahedron]
  \label{ex:normal_tetrahedron}\hfill\\
  Let $(Y_1,Y_2,Y_3)$ be the random variables
corresponding to the tetrahedron mentioned above, $Z_1,Z_2,Z_3$ be independent
standard normal random variables and define $(X_1,X_2,X_3):= (Y_1,Y_2,Y_3) +
r\cdot (Z_1,Z_2,Z_3).$ Figure \ref{fig:normal_tetrahedron-1} shows the
empirical power for the three moment methods (using normalized distance
multivariance) depending on the sample size for the case $r = 0.5$.
Pearson's approximation \eqref{eq:pearson} matches the benchmark, the variance
based estimate \eqref{eq:cv} is slightly less powerful and the classical
method \eqref{eq:c1} is clearly outperformed. Nevertheless, since the test is
in this setting consistent against all alternatives in the limit $N\to \infty$
each method has power 1.
\end{example}

\begin{figure}[H]
    \centering
	\includegraphics[width = 0.7\textwidth]{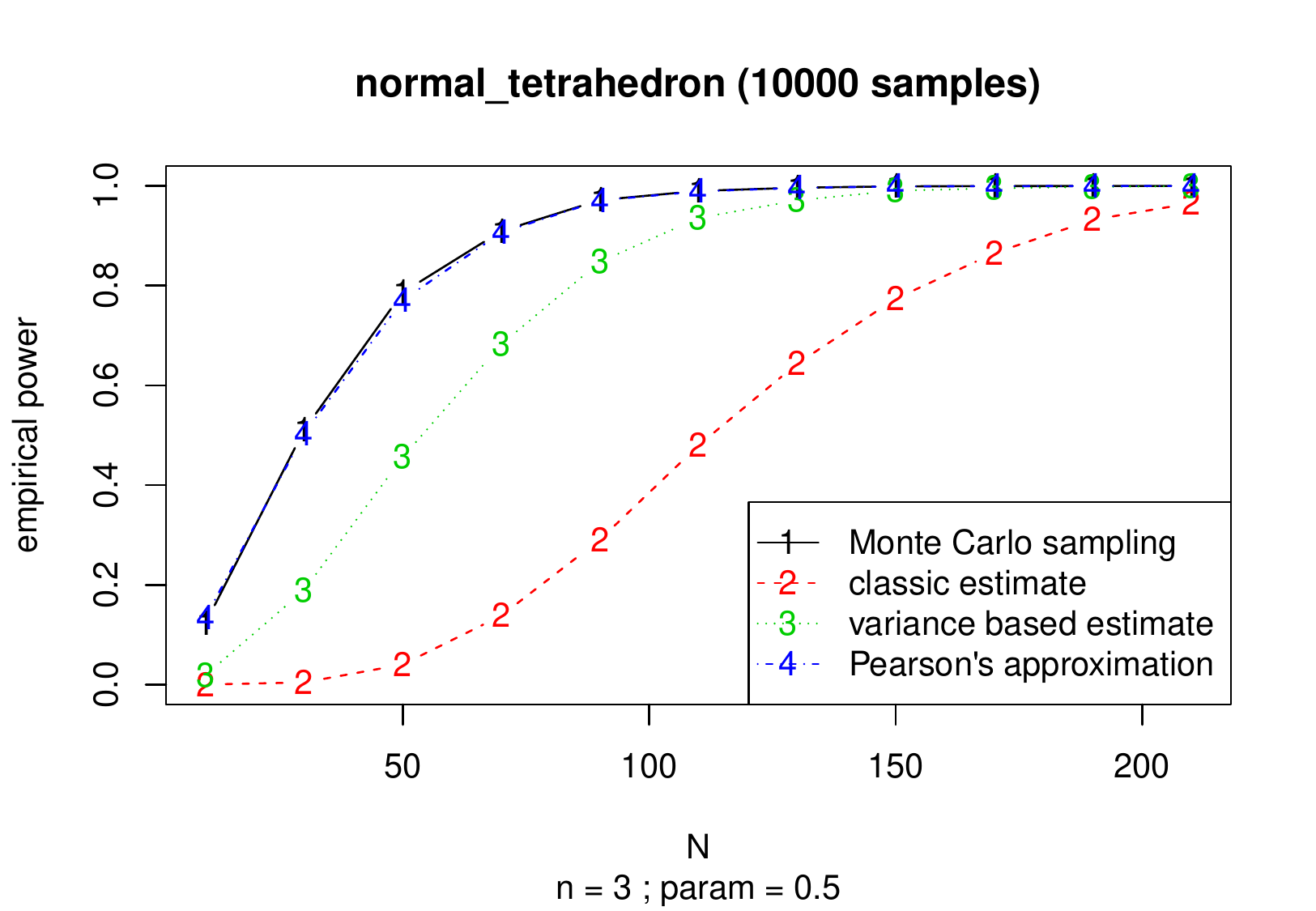}
	\caption{Power of the moment methods for samples of the normal
          tetrahedron (Ex.\ \ref{ex:normal_tetrahedron}).}
	\label{fig:normal_tetrahedron-1}
\end{figure}

\begin{example}[Multivariate Bernoulli marginals] \label{ex:mv_bern} 
Let $Y_1,\dotsc,Y_{10}$ be independent Bernoulli distributed random variables
and define $X_1:=(Y_1,\ldots,Y_5)$ and $X_2:=(Y_6,\ldots,Y_{10})$. Now consider
the sample distance multivariance corresponding to $M(X_1,X_2)$. We
computed 10000 samples of this for $N=100$, and estimated (from these
samples directly, i.e., not with our methods which would only require one
sample) the empirical distribution and its mean, variance and skewness.
The distribution function and the estimates are plotted in Figure
\ref{fig:mv_bern}. This illustrates several important aspects: 1.\ For
multivariate Bernoulli marginals the classical estimate \eqref{eq:c1} is not
sharp, in contrast to the univariate case (cf.\ Remark
\ref{rem:bernoulli-sharp}). 2.\ The classical and our variance based estimate
\eqref{eq:cv} are only tail estimates, and they can be very conservative. 3.\
Pearson's approximation \eqref{eq:pearson} (which uses only one parameter more
than the variance based estimate) works astonishingly well.
\end{example}

\begin{figure}[H]
    \centering
	\includegraphics[width = 0.7\textwidth]{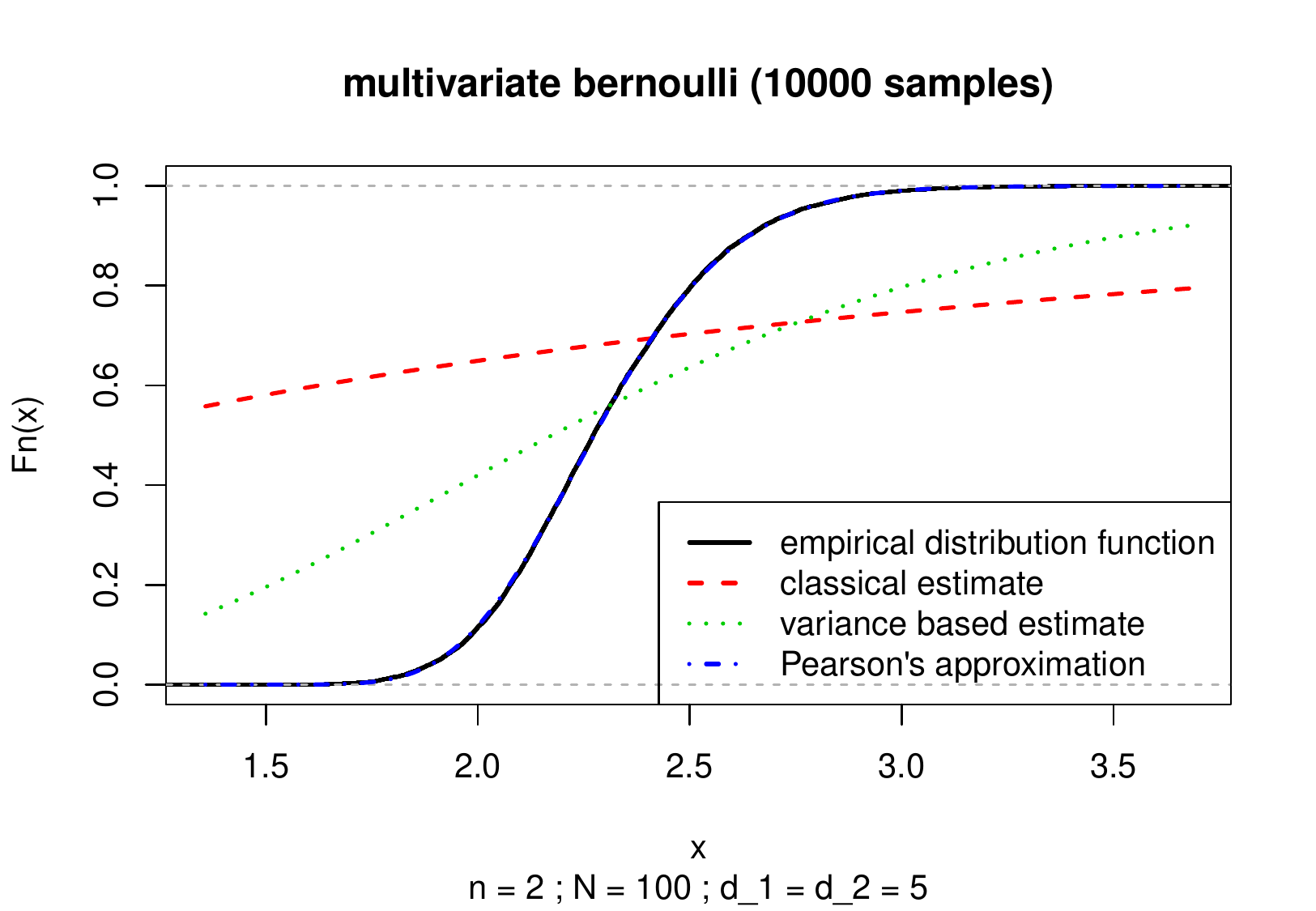}
	\caption{The distribution function predicted by the moment methods for
          multivariate Bernoulli marginals (Ex.\ \ref{ex:mv_bern}).}
	\label{fig:mv_bern}
\end{figure}

\begin{example}[Variants of the classical estimate~\eqref{eq:c1}]
\label{ex:c1} Let $X_1,\dotsc,X_5 $ be independent Bernoulli random
variables. The classical estimate \eqref{eq:c1} is (in the limit) sharp for
this case, see Remark \ref{rem:bernoulli-sharp}. But for the finite sample
case there are at least four ways to estimate the mean which is required for
the tail estimate: The biased or unbiased estimator for the limit
or for the finite sample mean (Corollary\ \ref{cor:estlimitmom}, Remarks\
\ref{rem:unbiased} and \ref{rem:finitemoms}.\ref{rem:finitemoms:est}). In
Figure \ref{fig:c1-1} the empirical size of the corresponding tests is
depicted for (very) small sample sizes. The unbiased finite sample estimator
is closest to the true value. As expected (Remarks
\ref{rem:finitemoms}.\ref{rem:finitmoms:Nu_vs_lb} and
\ref{rem:bias_vs_conservative}) the biased limit estimator is very close to it
but it is slightly less conservative. The biased finite sample estimator
yields a (too) liberal behavior, the unbiased limit estimator yields a (too)
conservative behavior. The latter becomes even more obvious for $n> 2$, see
Figure \ref{fig:c1-2}.
\end{example}

\begin{figure}[H]
    \centering
    \includegraphics[width = 0.7\textwidth]{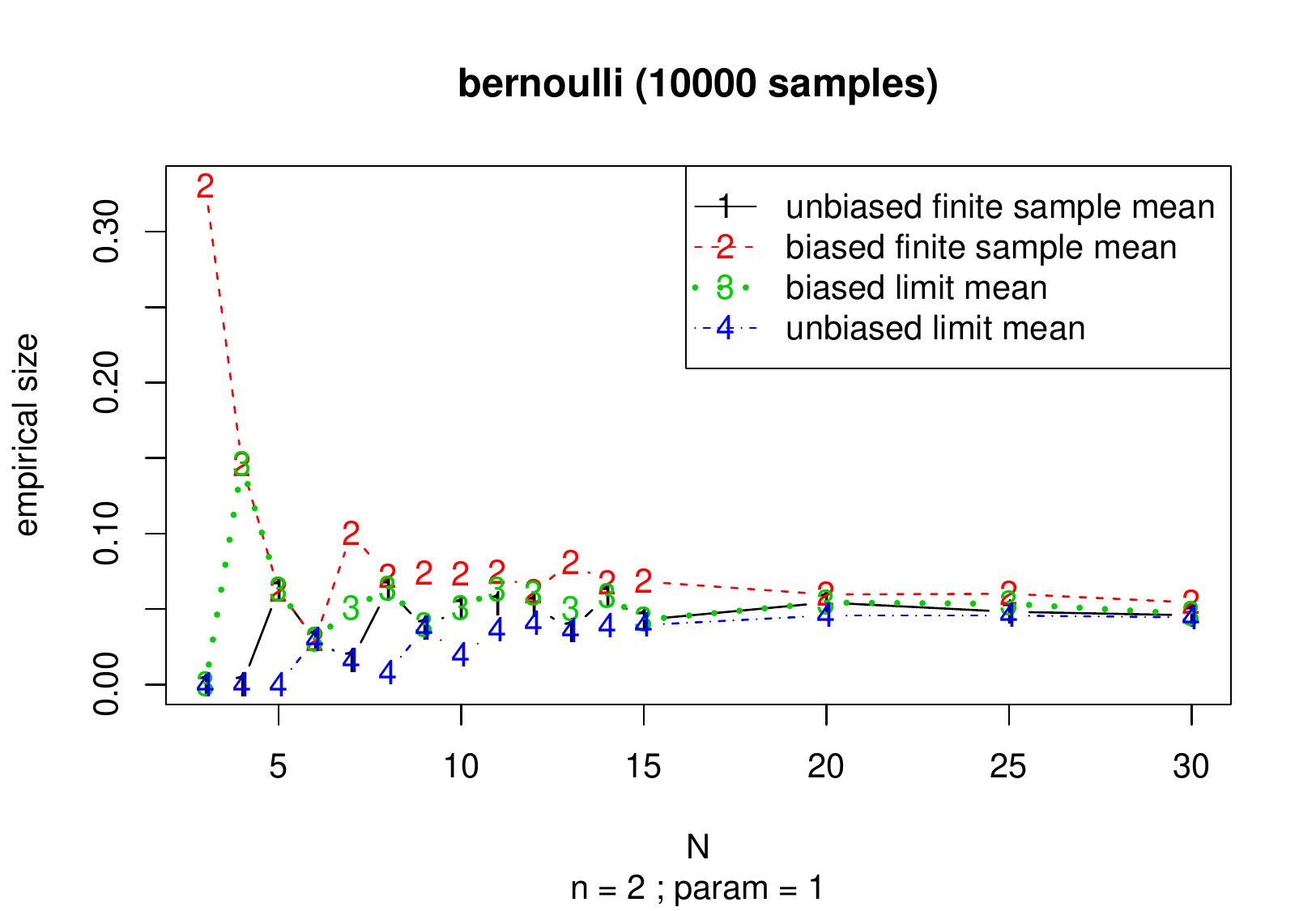}
    \caption{Empirical size using the classical estimate \eqref{eq:c1}
      with 4 different mean estimators for $X_1,X_2$ Bernoulli random
      variates (Ex.\ \ref{ex:c1}).} \label{fig:c1-1}
\end{figure}
  
\begin{figure}[H]
    \centering
    \includegraphics[width = 0.7\textwidth]{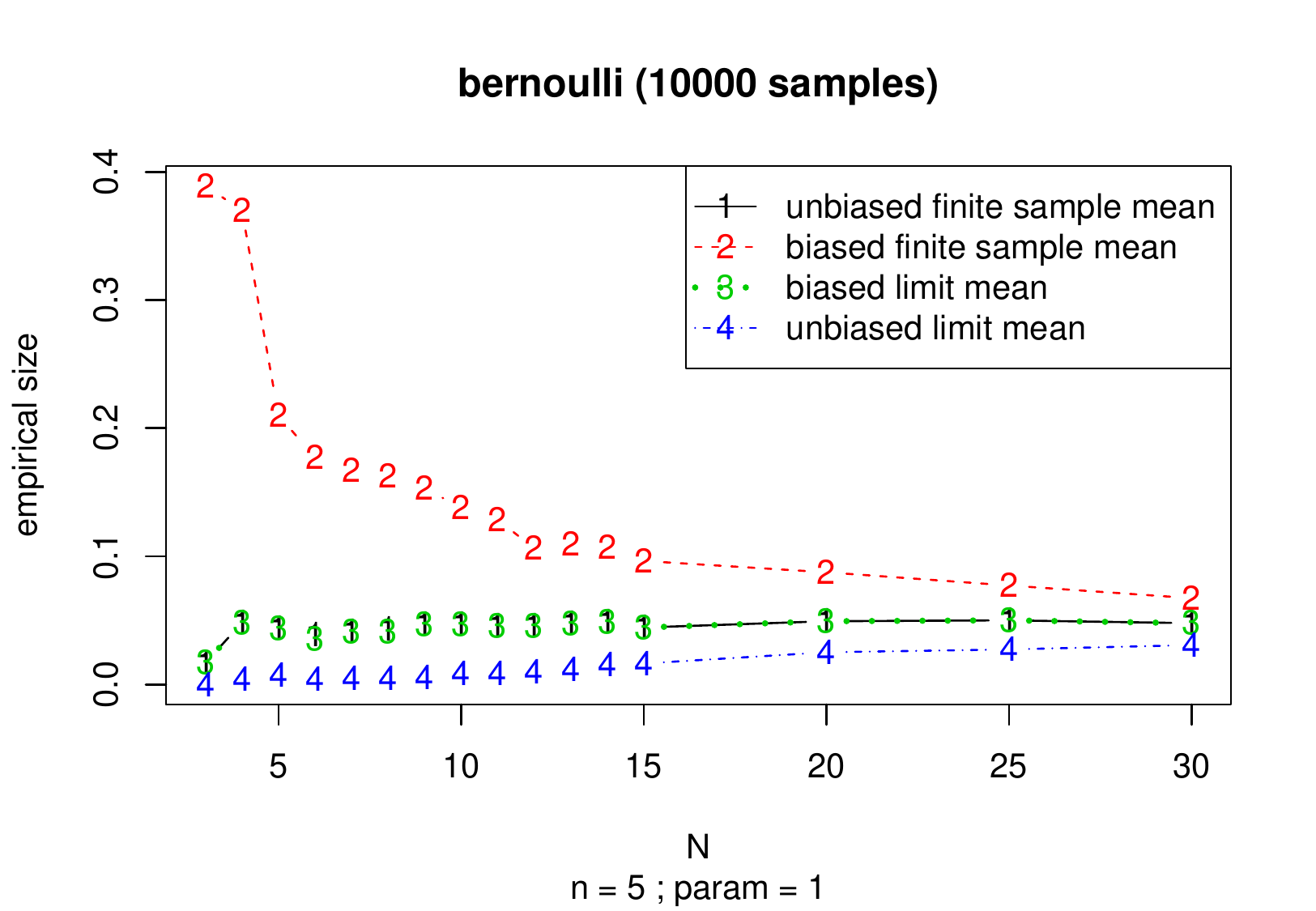}
    \caption{Empirical size using the classical estimate \eqref{eq:c1}
      with 4 different mean estimators for $X_1,\ldots,X_5$ Bernoulli
      random variates (Ex.\ \ref{ex:c1}).}
	\label{fig:c1-2}
\end{figure}

\begin{example}[Moment estimation is faster than
  resampling]\label{ex:speed}
\hspace{-.5em} We compare the time for computation of all moment
estimates for one sample with the time of one evaluation of multivariance for
one resampling. In Figure \ref{fig:speed-1} their ratio is depicted, i.e., the
number of evaluations of multivariance (with resampling the data) which can be
performed in the time it takes to compute all moment estimates. We use
normally distributed marginals and use in the ratio the median of the
computation time of 100 repetitions.

Note that \cite{SzekRizzBaki2007} suggest the use of
\(\lfloor 200+ \frac{5000}{N}\rfloor\) resampling samples, thus -- by the
numbers in Figure \ref{fig:speed-1} -- the moment approach
is clearly faster than the resampling approach even for small samples. A
bottleneck of the moment estimates is (in the current implementation) the
computation of one matrix multiplication (of $N\times N$ matrices) for each
variable, we use an R distribution (MRO 3.5.1) which is improved for such
tasks.
   
\end{example}

\begin{figure}[H]
    \centering
	\includegraphics[width = 0.7\textwidth]{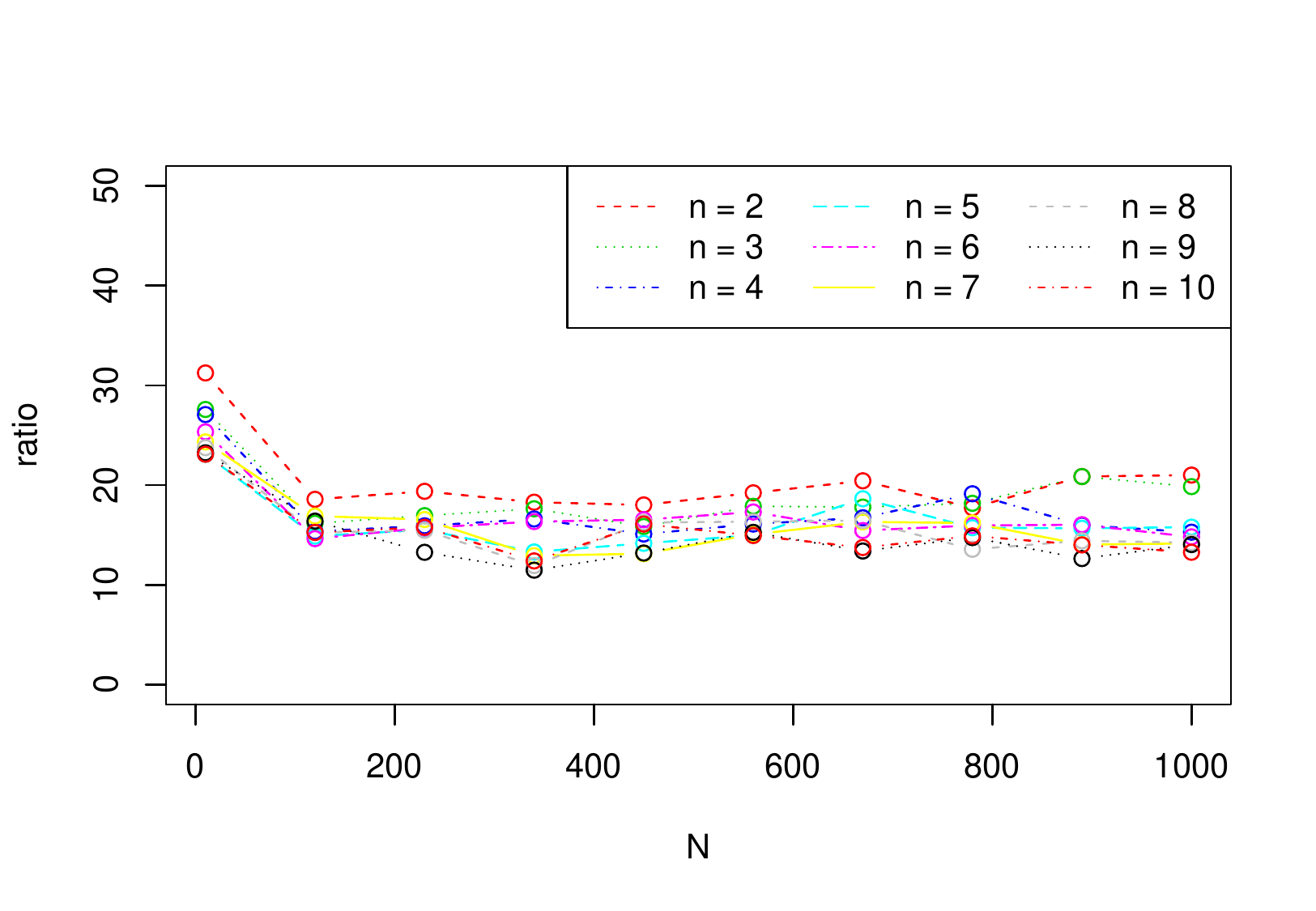}
	\caption{ The number of evaluations of multivariance (with resampling
          the data) which can be performed in the time of the computation of
          all moment estimates (Ex.\ \ref{ex:speed}).}
	\label{fig:speed-1}
\end{figure}

\begin{example}[Estimation of the parameters $\mu_i, b_i,c_i,d_i$ of the
  marginals]\label{ex:marginal_mom_sample}\hfill\\
  There are four settings for the estimation of the moments corresponding to
  the marginal distribution: with or without bias for multivariance with or
  without normalization (Corollary \ref{cor:estlimitmom} and Remarks
  \ref{rem:unbiased}, \ref{rem:finitemoms}.\ref{rem:finitemoms:est} and
  \ref{rem:finitenormmoms}).

Figures \ref{fig:marginal_mom_sample-bernoulli},
\ref{fig:marginal_mom_sample-uniform} and \ref{fig:marginal_mom_sample-normal}
show for Bernoulli, uniform and standard normal variates the estimation of
$\mu^{(1)},\mu^{(2)},\mu^{(3)}$ and $b,c,d$. The biased estimators clearly
show the bias for small sample sizes. In the case of unbiased estimators
without normalization the estimators are really unbiased. But note that in the
case of using unbiased estimators for normalized multivariance there is a
bias, since the transformation of the estimator creates a bias (see also the
comment on biased and unbiased estimators in Remarks \ref{rem:methods} and
\ref{rem:bias_vs_conservative}).
\end{example}

\begin{figure}[H]
    \centering
	\includegraphics[width = 0.7\textwidth]{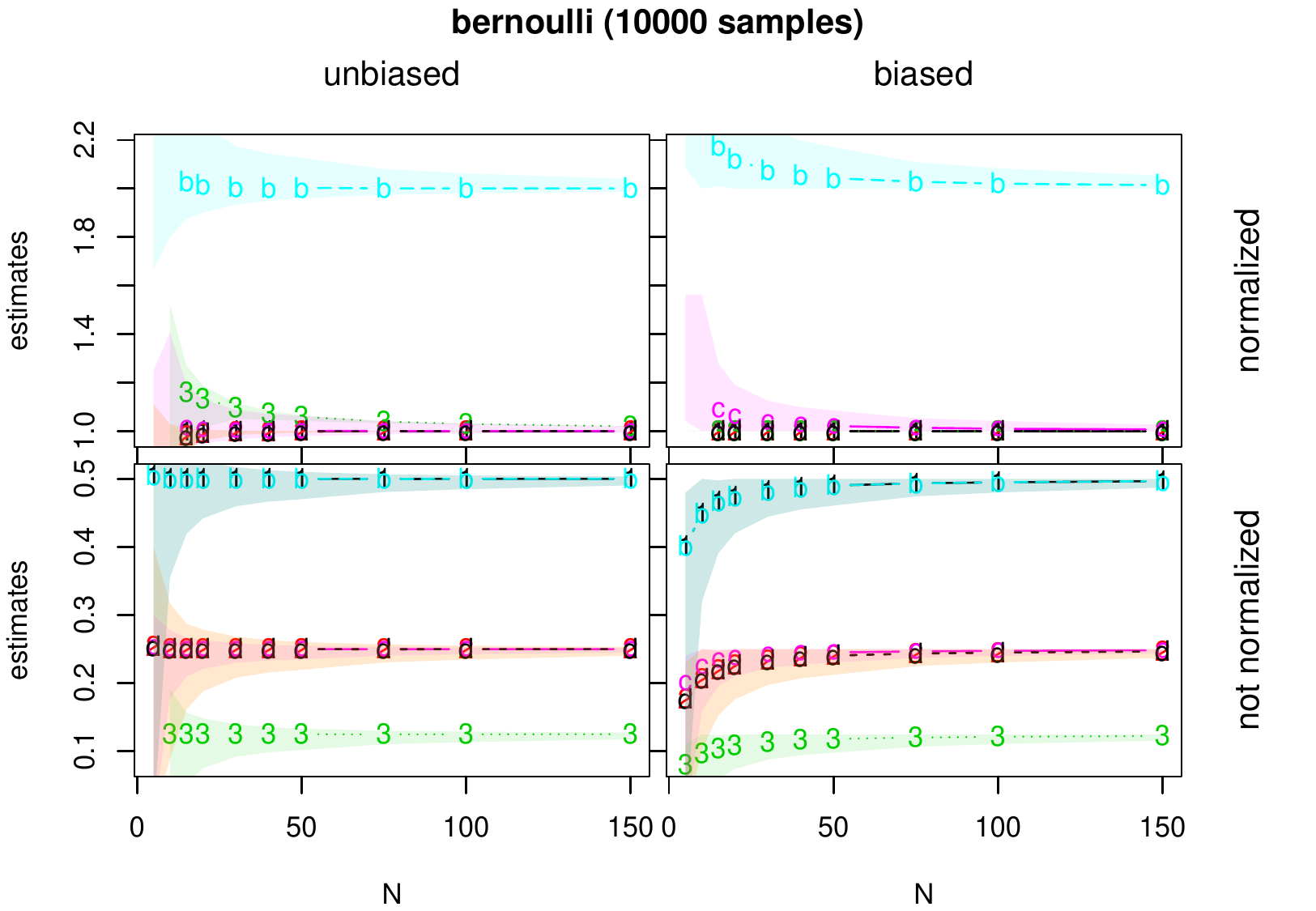}
	\caption{Convergence of the estimates of the parameters of the
          marginal distribution for Bernoulli variates (Ex.\
          \ref{ex:marginal_mom_sample}). The shaded regions describe the
          5\,\% to 95\,\% quantile of the estimates and the labeled line is the
          mean of the estimates. The estimates for $\mu^{(k)}$ are denoted by
          $k=$\texttt{1,2,3} and for $b, c, d$ by  \texttt{b,c,d}.}
	\label{fig:marginal_mom_sample-bernoulli}
\end{figure}

\begin{figure}[H]
    \centering
	\includegraphics[width = 0.7\textwidth]{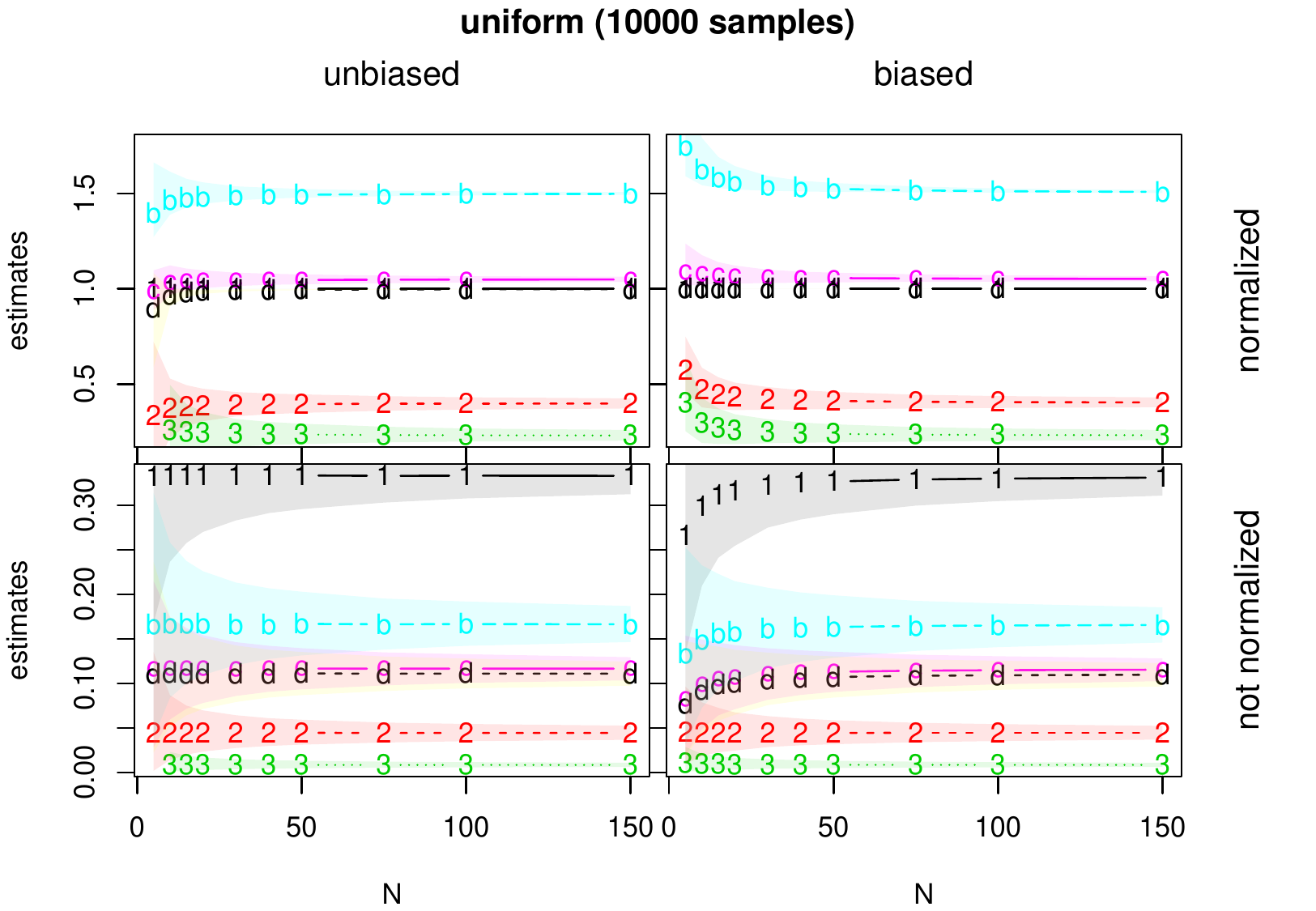}
	\caption{Convergence of the estimates of the parameters of the
          marginal distribution for uniformly distributed variates (Ex.\
          \ref{ex:marginal_mom_sample}).}
	\label{fig:marginal_mom_sample-uniform}
\end{figure}

\begin{figure}[H]
    \centering
	\includegraphics[width = 0.7\textwidth]{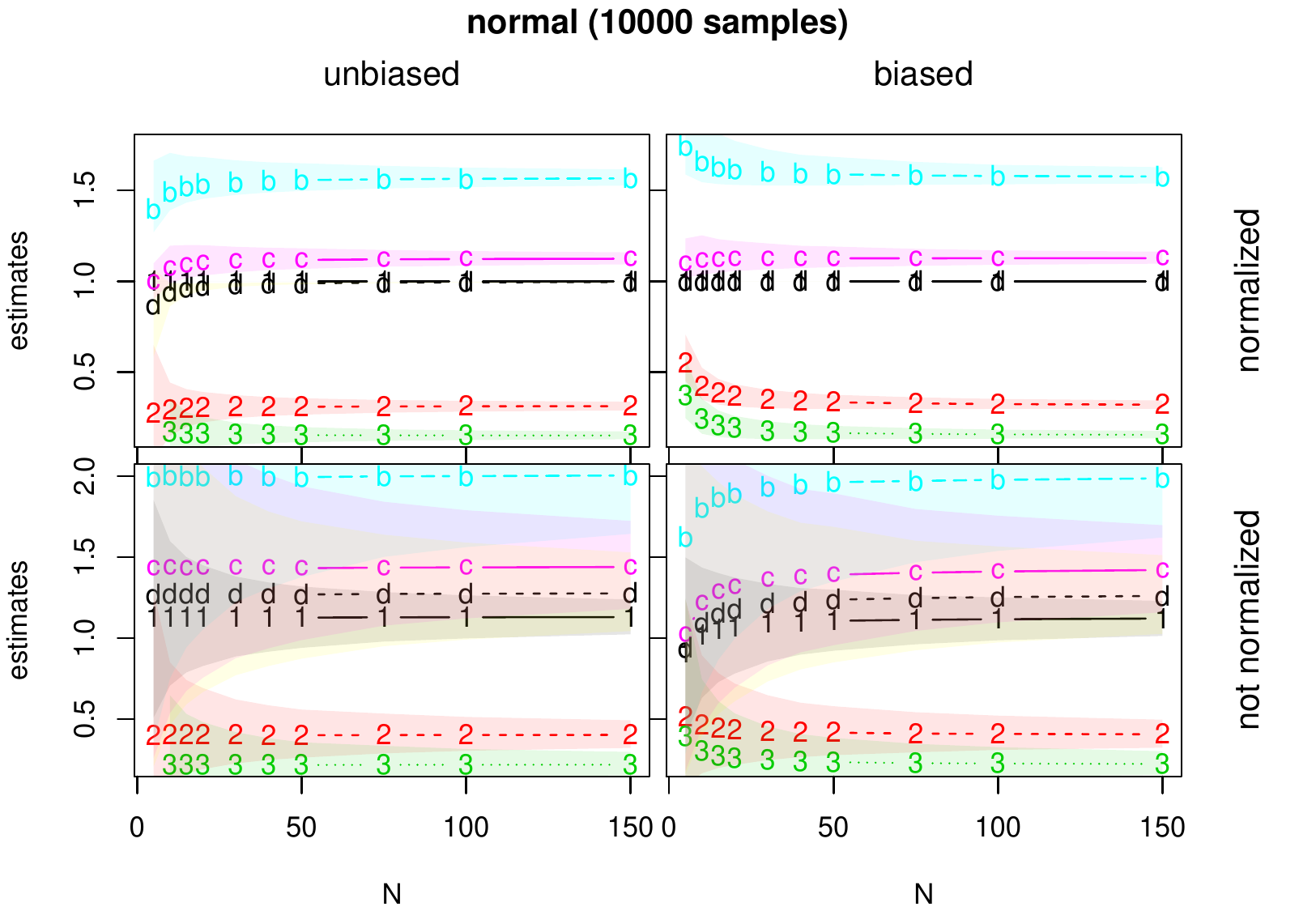}
	\caption{Convergence of the estimates of the parameters of the
          marginal distribution for standard normal variates (Ex.\
          \ref{ex:marginal_mom_sample}).}
	\label{fig:marginal_mom_sample-normal}
\end{figure}

\begin{example}[A priori parameter estimation for known
  marginals] \label{ex:momknownmarginals}\hfill\\
  In the case of known marginal distributions one can also use precomputed
  values for the parameters $\mu_i^{(1)},\mu_i^{(2)},\mu_i^{(3)}$. To 
  evaluate different numerical approaches based on the derived
  representations we give theoretical values (obtained from straightforward
  but tedious computations assisted by \texttt{MAPLE}) for selected
  marginal distributions:
  \begin{center}{\renewcommand{\arraystretch}{1.2}
    \begin{tabular}{r|ccc}
      marginal & $\sumalpha{1}_i$ & $\sumalpha{2}_i$ & $\sumalpha{3}_i$\\\hline
      Bernoulli & $\frac{1}{2}$ & $\frac{1}{4}$ & $\frac{1}{8}$ \\
      uniform  & $\frac{1}{3}$ & $\frac{2}{45}$ & $\frac{8}{945}$\\
      normal & $\frac{2}{\sqrt{\pi}}$ &    $\frac{4\pi+12(1-\sqrt{3})}{3\pi}$
                 & $\frac{8(1-\pi)+12(\sqrt{2}-\sqrt{3})+18\arctan
                   \sqrt{8}}{\pi\sqrt{\pi}}$\\
      exponential & $1$ & $\frac{1}{3}$ & $\frac{1}{6}$
    \end{tabular}}
  \end{center}
  We compare these values to the results of the following three
  numerical schemes (implemented in \texttt{MATLAB}):
  \begin{enumerate}
  \item Use Nystr\"om's method to estimate (some of) the coefficients
    $\alpha_i$, $i\in I$, and use \eqref{eq:defmu};
  \item use numerical quadrature to compute the iterated integrals
    of the kernels $\covG$ in Proposition~\ref{prop:pnormintegral};
  \item use numerical quadrature to compute the expectations in
    Lemma~\ref{lem:mureps}.
  \end{enumerate}
  We also compare the results against values of the estimators
  $\hN\sumalpha{k}$ for the moments of the limit distribution based on
  generated samples using
  \begin{enumerate}\setcounter{enumi}{3}
  \item the biased estimators for the summands given in
    Corollary~\ref{cor:estlimitmom} or
  \item the unbiased estimators for the
    summands given in Remark~\ref{rem:unbiased}.
  \end{enumerate}
  Results for some univariate marginal distributions (Bernoulli with
  $p=\tfrac{1}{2}$, uniform distribution on $[0,1]$, standard normal
  distribution, exponential distribution with $\lambda=1$) are summarized in
  Table~\ref{tab:psums} on page \pageref{tab:psums}, where we used $100$ nodes
  for the quadratures, the estimators are based on sample size $N=1000$,
  and we state the median computation time for $100$ repetitions.\par

  Note that -- using the same quadrature rule for 1. and 2. -- the results
  are the same, but 2. is slightly faster since it does not calculate the
  coefficients explicitely. From the experiments, the quadrature-based methods
  2. and 3. appear to give the closest results in shortest time. Note,
  however, that the quadrature rule for 3. has to/should be adapted to the
  distribution in question, whereas 2. uses the same quadrature rule (in this
  case: Gauss-Hermite quadrature) for all four cases. We therefore recommend,
  if the marginal distributions are known, to use a numerical quadrature rule
  to compute the iterated kernel integrals.
  \begin{table}[H]
    \centering
\begin{tabular}{l|rrrp{2cm}}
    & $\sumalpha{1}_i$ & $\sumalpha{2}_i$ & $\sumalpha{3}_i$ & computation time
                                                               (seconds)\\\hline
    Bernoulli marginal\\
  theoretical value & 0.500000 & 0.250000 & 0.125000 \\
1. Nystr\"om method	 & 0.475400 & 0.226005 & 0.107443 & 0.0074\\ 
2. quadrature (kernels)	 & 0.475400 & 0.226005 & 0.107443 & 0.0050\\ 
3. quadrature (expectation)	 & 0.500000 & 0.250000 & 0.125000 & 0.0003\\ 
4. estimator (biased)	 & 0.499550 & 0.249550 & 0.124663 & 0.1216\\ 
5. estimator (unbiased)	 & 0.500050 & 0.250050 & 0.125037 & 0.1569\\\\

  uniform marginal\\
  theoretical value (approx.)& 0.333333 & 0.044444 & 0.008466 \\
1. Nystr\"om method	 & 0.287288 & 0.044055 & 0.008424 & 0.0074\\ 
2. quadrature (kernels)	 & 0.287288 & 0.044055 & 0.008424 & 0.0051\\ 
3. quadrature (expectation)	 & 0.333306 & 0.044459 & 0.008468 & 0.0018\\ 
4. estimator (biased)	 & 0.326578 & 0.042431 & 0.007889 & 0.1212\\ 
5. estimator (unbiased)	 & 0.326905 & 0.042429 & 0.007883 & 0.1572\\\\

  normal marginal\\
theoretical value (approx.) &  1.128379 & 0.401257 & 0.217387 \\
1. Nystr\"om method	 & 1.082144 & 0.401209 & 0.217387 & 0.0062\\ 
2. quadrature (kernels)	 & 1.082144 & 0.401209 & 0.217387 & 0.0034\\ 
3. quadrature (expectation)	 & 1.123745 & 0.408878 & 0.221314 & 0.0020\\ 
4. estimator (biased)	 & 1.139559 & 0.405160 & 0.219337 & 0.1211\\ 
5. estimator (unbiased)	 & 1.140700 & 0.404386 & 0.218490 & 0.1570\\\\

 exponential marginal\\
  theoretical value (approx.) & 1.000000 & 0.333333 & 0.166667 \\
1. Nystr\"om method	 & 0.953846 & 0.333081 & 0.166583 & 0.0075\\ 
2. quadrature (kernels)	 & 0.953846 & 0.333081 & 0.166583 & 0.0043\\ 
3. quadrature (expectation)	 & 0.995893 & 0.339694 & 0.169711 & 0.0018\\ 
4. estimator (biased)	 & 1.002041 & 0.334769 & 0.168104 & 0.1218\\ 
5. estimator (unbiased)	 & 1.003044 & 0.333300 & 0.166827 & 0.1570  
\end{tabular}
\caption{Comparison of numerical approximation schemes for the sums of
  coefficients for four univariate marginals (Ex.~\ref{ex:momknownmarginals}).}
\label{tab:psums}
\end{table}
\end{example}

\begin{example}[Estimation of the (joint) moments] \label{ex:jointmom} 
  For normal marginals (using the unbiased estimators; Corollary
  \ref{cor:estlimitmom}, Remark \ref{rem:unbiased} and Theorem
  \ref{thm:finitemoms}) Figure \ref{fig:jointmom-2} shows the estimates of the
  moments of the test statistic depending on $N$ for fixed $n= 5$.  These show
  that for small $N$ the limit mean overestimates the finite sample mean, and
  for large $N$ the finite sample variance decreases (very) slowly to the
  limit variance. In Figure \ref{fig:jointmom-1} the estimates depending on
  $n$ for fixed $N = 100$ are shown. In particular for larger $n$ this shows
  that the limit variance underestimates the finite sample variance by far.
  
  The estimates for the skewness improve with the sample size. But for
  increasing $n$ and fixed $N$ a massive underestimation occurs -- note that
  further analysis indicates a very large variability of the Monte
  Carlo estimator for this case, further research might clarify the cause.
  
  Note that the Monte Carlo estimate uses the 10000 sample multivariances to
  estimate the moment, while the other two estimators compute a
  moment estimate for each of the 10000 samples. Thus for the latter we can
  depict (shaded region) the 0.05 to 0.95 quantile of the estimates and the
  mean of the estimates (line). Moreover recall that we did not derive an
  estimator for the skewness of the finite sample distribution of the test
  statistic.
\end{example}

\begin{figure}[H]
    \centering
	\includegraphics[width = 0.7\textwidth]{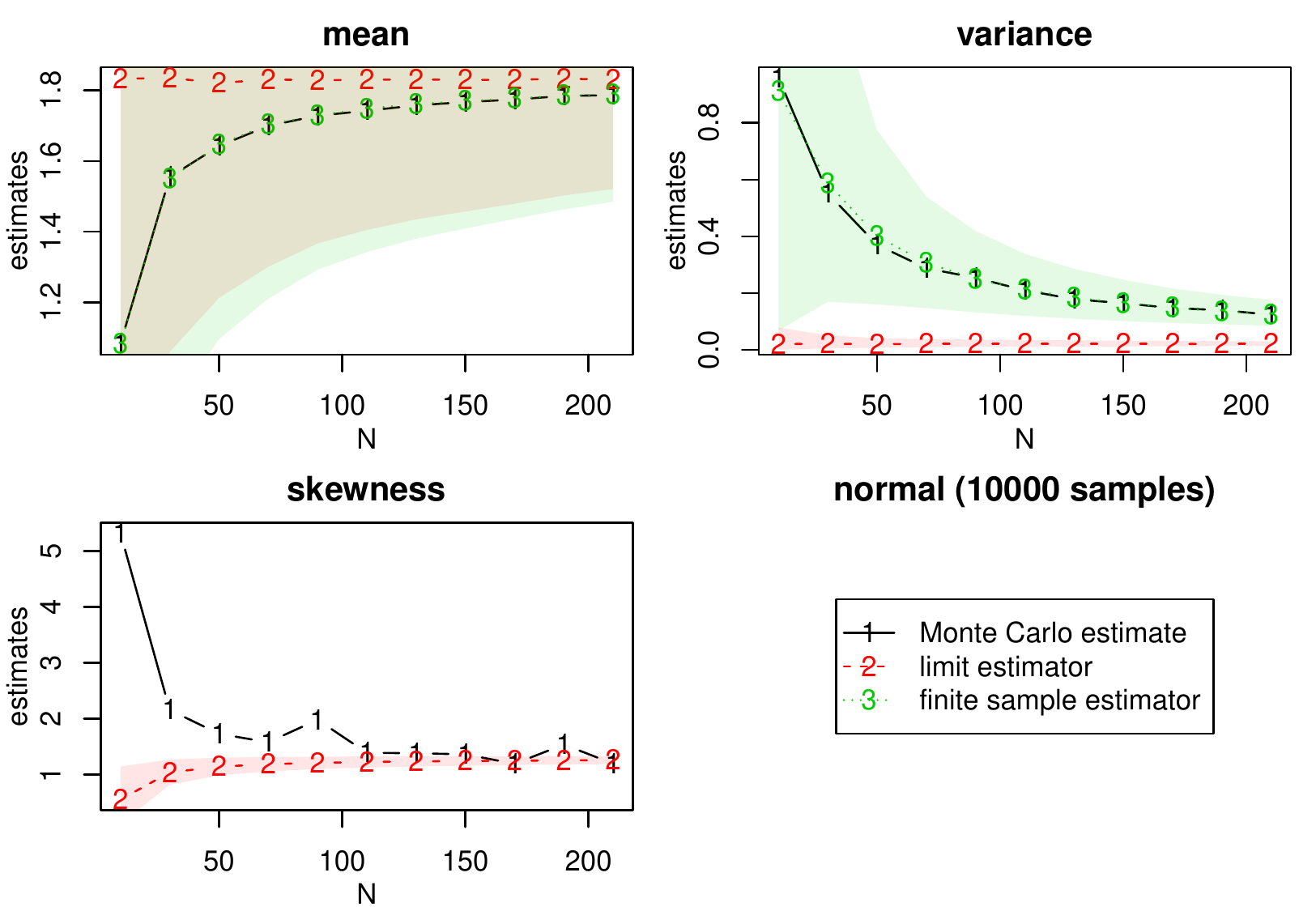}
	\caption{Moment estimation for normal marginals for fixed $n=5$ (Ex.\
          \ref{ex:jointmom}).}
	\label{fig:jointmom-2}
\end{figure}

\begin{figure}[H]
    \centering
	\includegraphics[width = 0.7\textwidth]{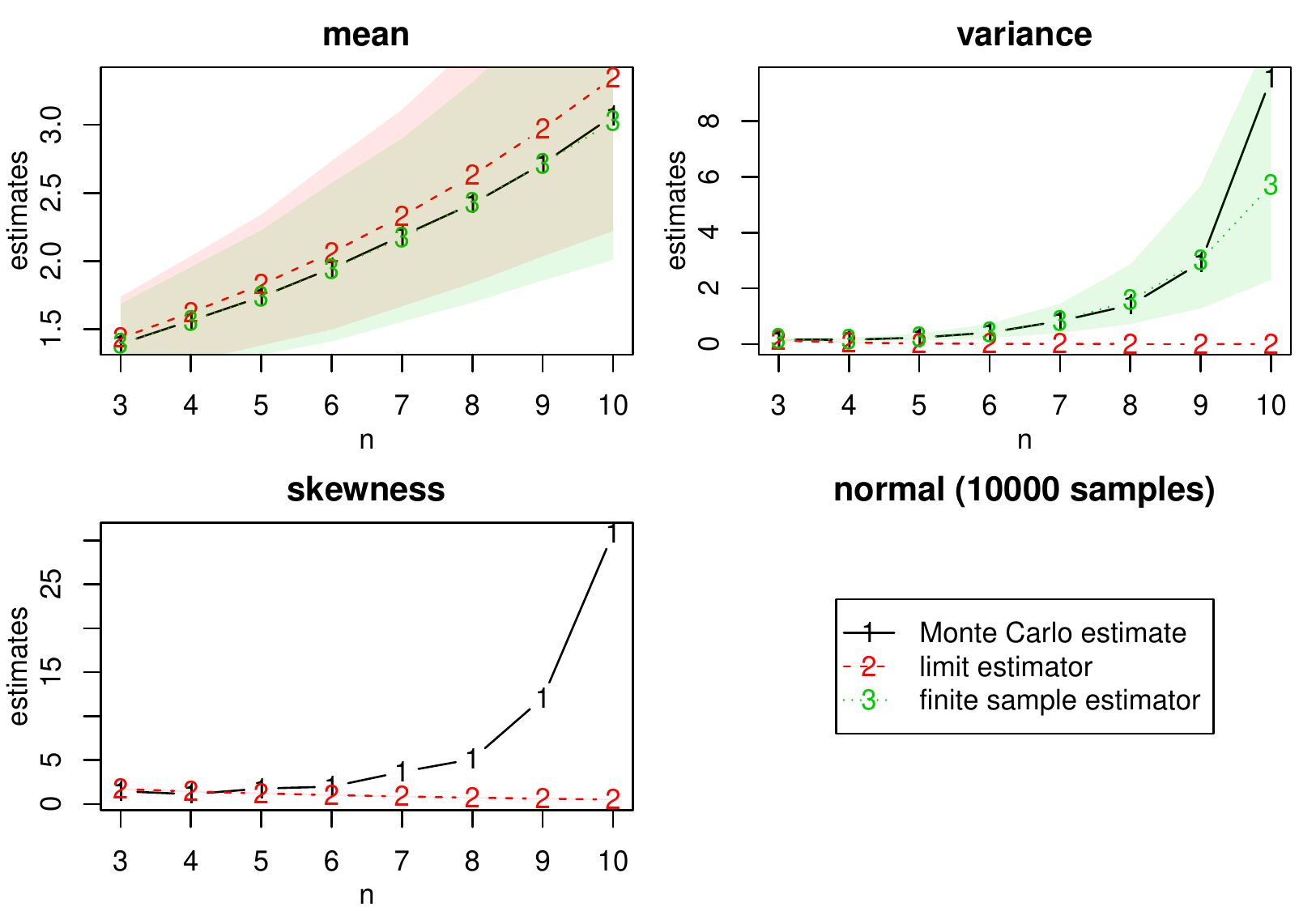}
	\caption{Moment estimation for normal marginals for fixed $N=100$
          (Ex.\ \ref{ex:jointmom}).}
	\label{fig:jointmom-1}
\end{figure}

\begin{example}[Distribution of the finite sample estimator]
\label{ex:not_gQF} Most of the methods of Section \ref{sec:tests} use the finite
sample estimators as if the test statistic
is distributed as a Gaussian quadratic form. Although the other examples show
that this works very well, one should be aware of the fact that in general the
distribution is not that of a Gaussian quadratic form. As an extreme example
consider the case of Bernoulli marginals. In this case the distribution of
$N\cdot \hN M^2(\vect{X}^{(1)},\ldots, \vect{X}^{(N)})$ is a discrete
distribution taking only finitely many values, e.g., for $N=10$ there are only
35 different values (realized in 10000 samples). See Figure 
\ref{fig:not_gQF-1} for the empirical counting density, which is very
irregular.
\end{example}

\begin{figure}[H]
    \centering
    \includegraphics[width = 0.7\textwidth]{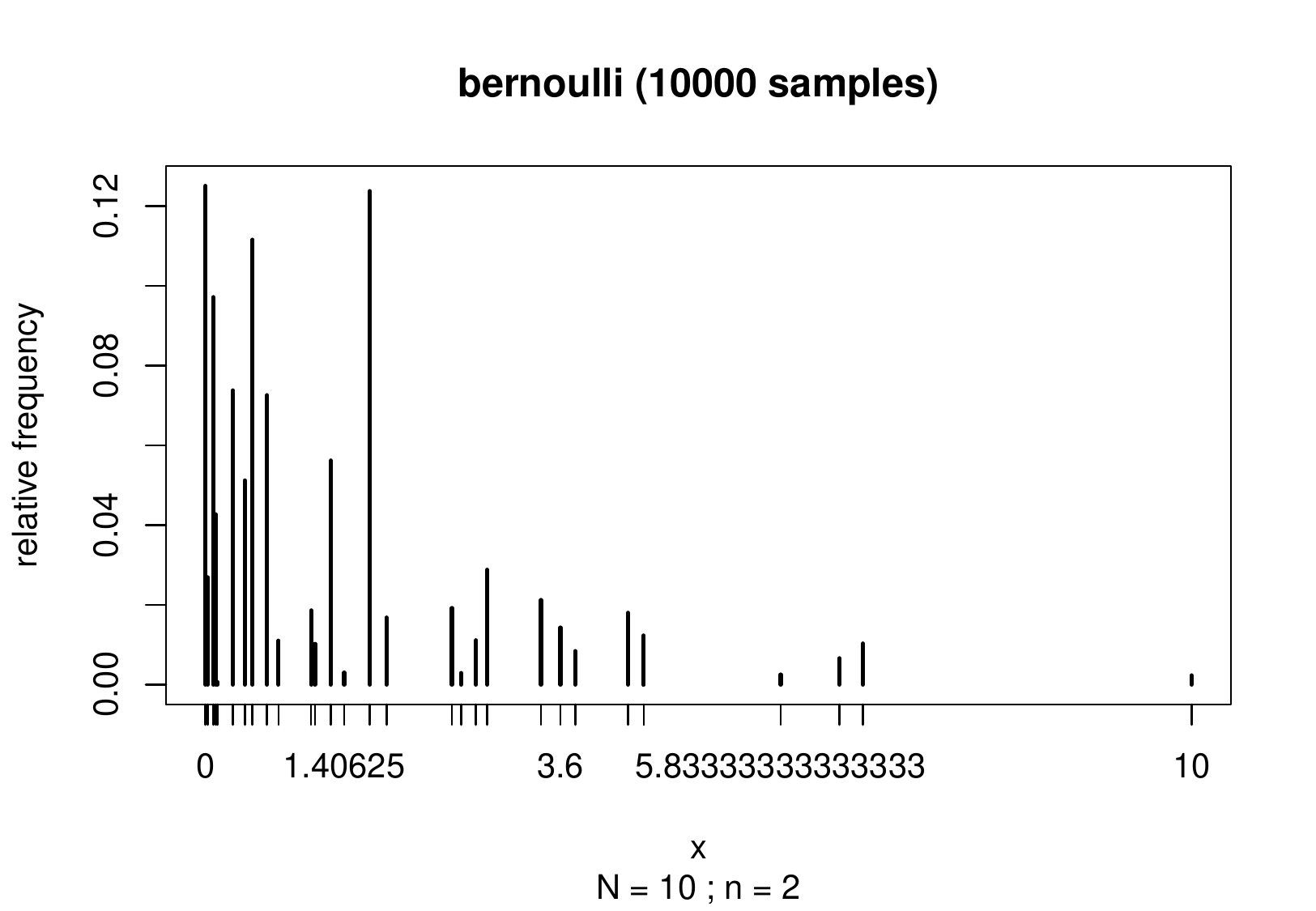}
    \caption{Counting density (based on 10000 samples) of sample distance
      multivariance with Bernoulli marginals for small samples $N= 10$ (Ex.\
      \ref{ex:not_gQF}).}
    \label{fig:not_gQF-1}
\end{figure}

\begin{example}[Robustness] \label{ex:robust} Here we consider a case
where the assumptions of the tests, i.e., the moment condition given in
Remark \ref{rem:momcond} and the implicit assumptions on the existence of
the parameters by each method, are violated.

Let $X_i$ be random variables with Student's t-distribution with 1
degree of freedom, thus their expectation does not exist. Figure
\ref{fig:robust-marginal-mom} shows that therefore the moment estimates are
problematic. In Figure \ref{fig:robust-empirical-size} the empirical size (for
independent $X_i$) is shown, which looks reasonable -- but shows already
strange behavior for $n=3$. Moreover, we also consider a dependent
sample, similarly to \cite[Example~1(b)]{SzekRizzBaki2007}: let
$(Y_1,\ldots,Y_{10})$ be multivariate t-distributed with 1 degree of freedom
and the scale matrix being the identity plus a block matrix with four $5
\times 5$ blocks with values 0, 0.1, 0.1, 0, respectively. In Figure
\ref{fig:robust-power} the power of the tests is depicted. Note that the
methods appear to be much more powerful than the benchmark. In Figure
\ref{fig:robust-power-normalized} the same example is computed for normalized
multivariance, here only Pearson's method is liberal.
\end{example}

\begin{figure}[H]
    \centering
    \includegraphics[width = 0.7\textwidth]{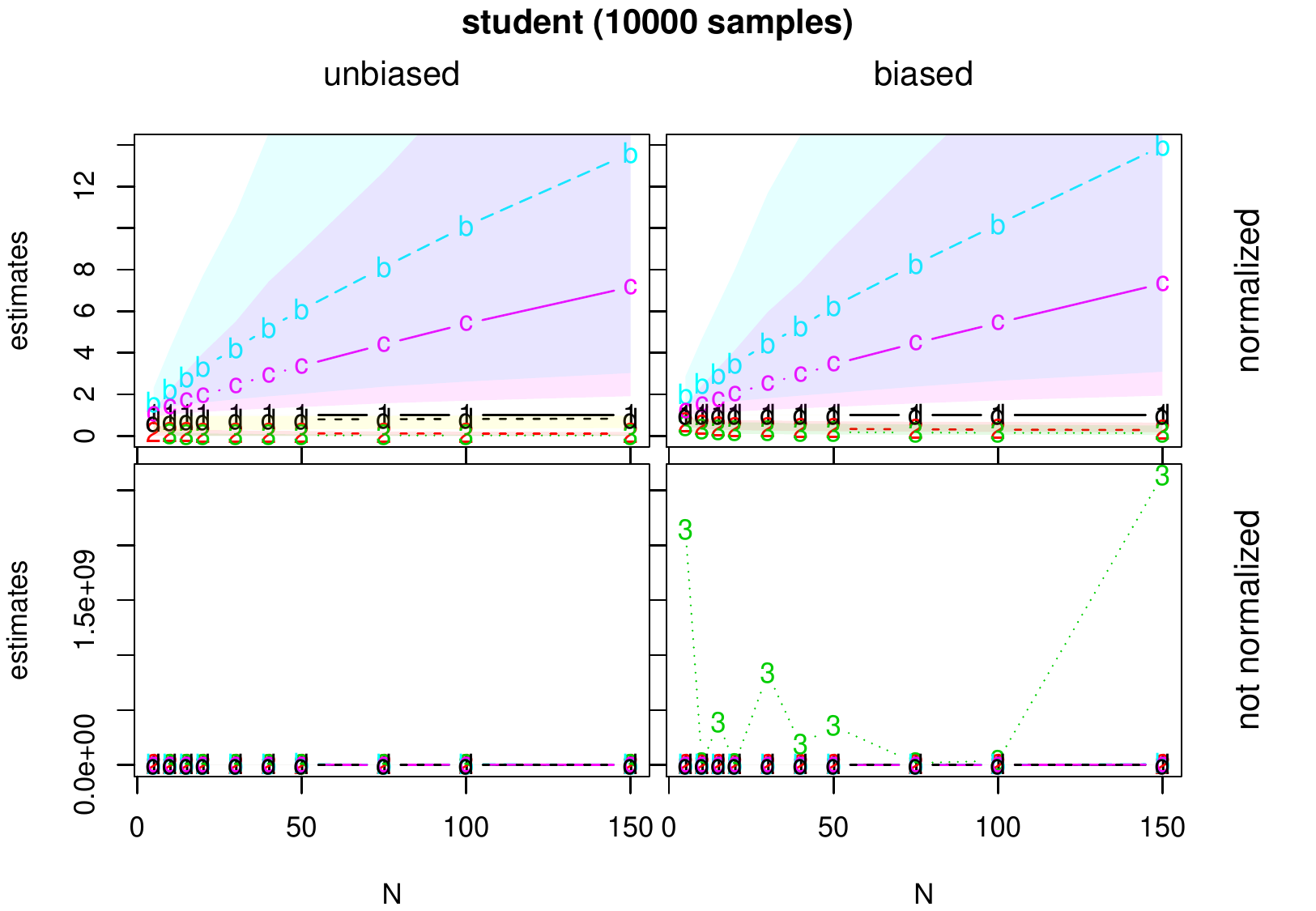}
    \caption{Estimation of the marginal moments for Student-t distributed
      variates, which violate the basic assumptions (Ex.\ \ref{ex:robust}).}
    \label{fig:robust-marginal-mom}
\end{figure}

\begin{figure}[H]
  \centering
  \includegraphics[width = 0.45\textwidth]{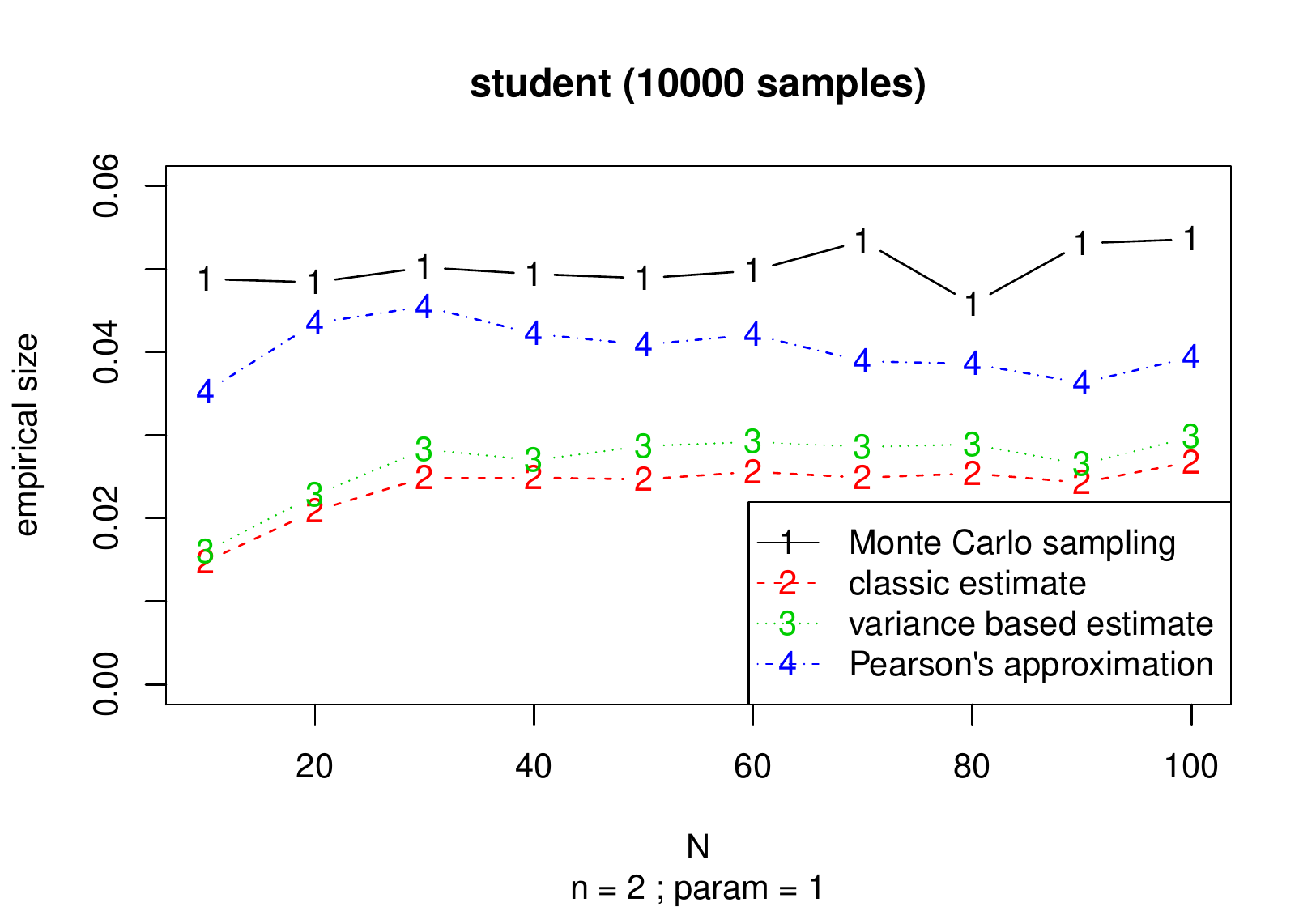}
  \includegraphics[width = 0.45\textwidth]{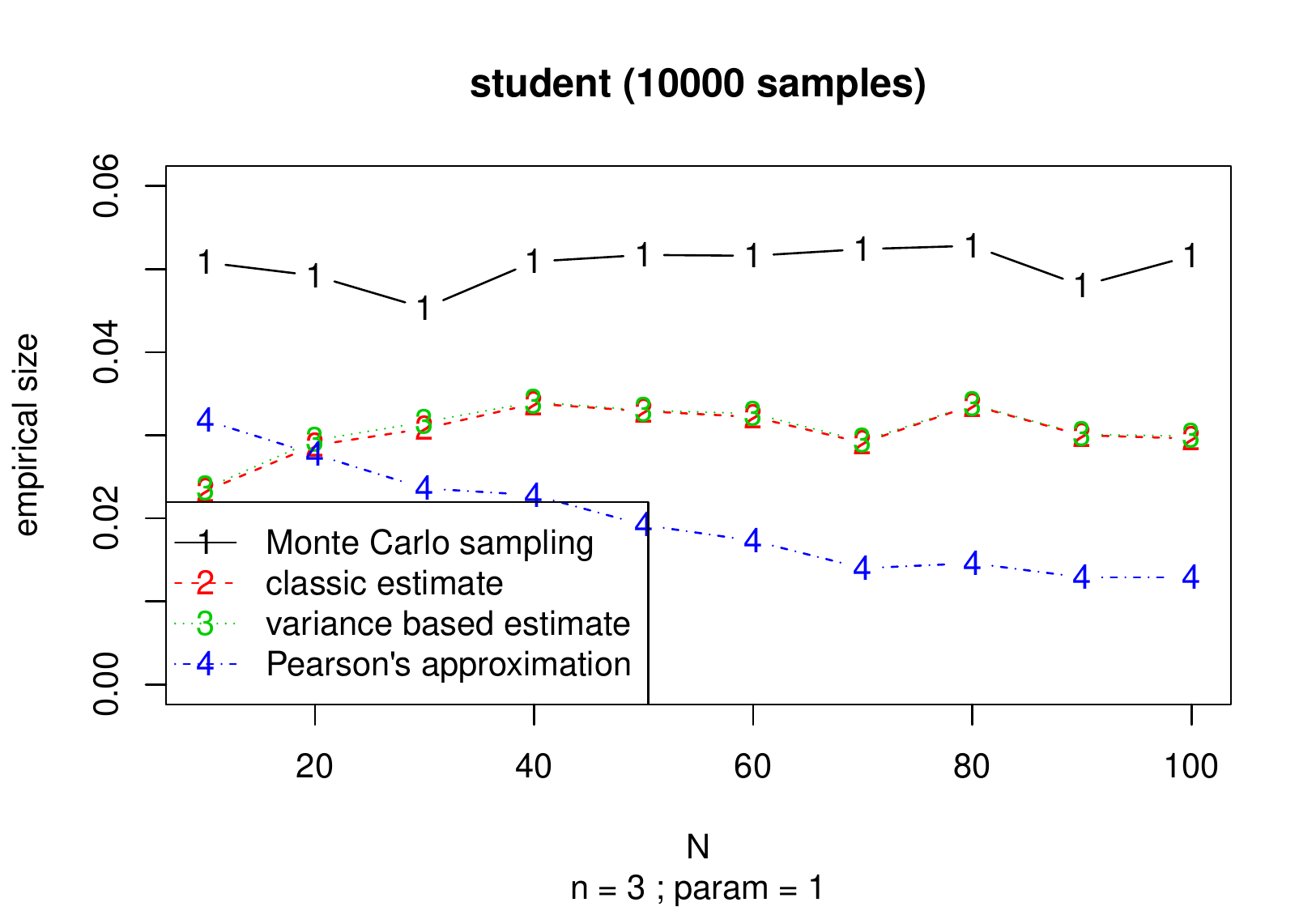}
  \caption{Empirical size of the moment methods for Student-t distributed
    variates, which violate the basic assumptions (Ex.\ \ref{ex:robust}).}
  \label{fig:robust-empirical-size}
\end{figure}

\begin{figure}[H]
    \centering
    \includegraphics[width = 0.7\textwidth]{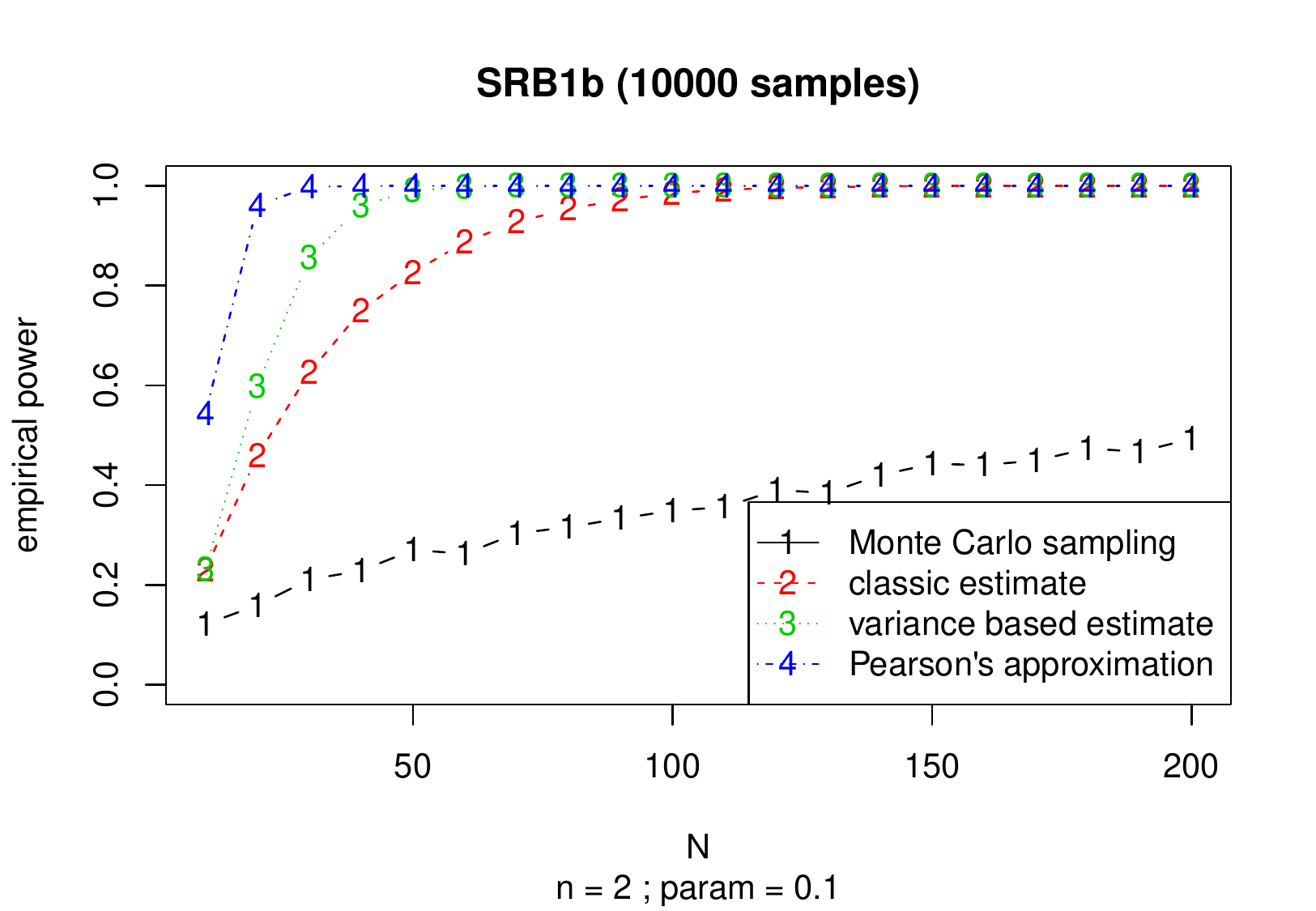}
    \caption{Power of the moment methods for dependent multivariate Student-t
      distributed marginals using distance multivariance (Ex.\
      \ref{ex:robust}).}
    \label{fig:robust-power}
  \end{figure}

\begin{figure}[H]
  \centering
  \includegraphics[width = 0.7\textwidth]{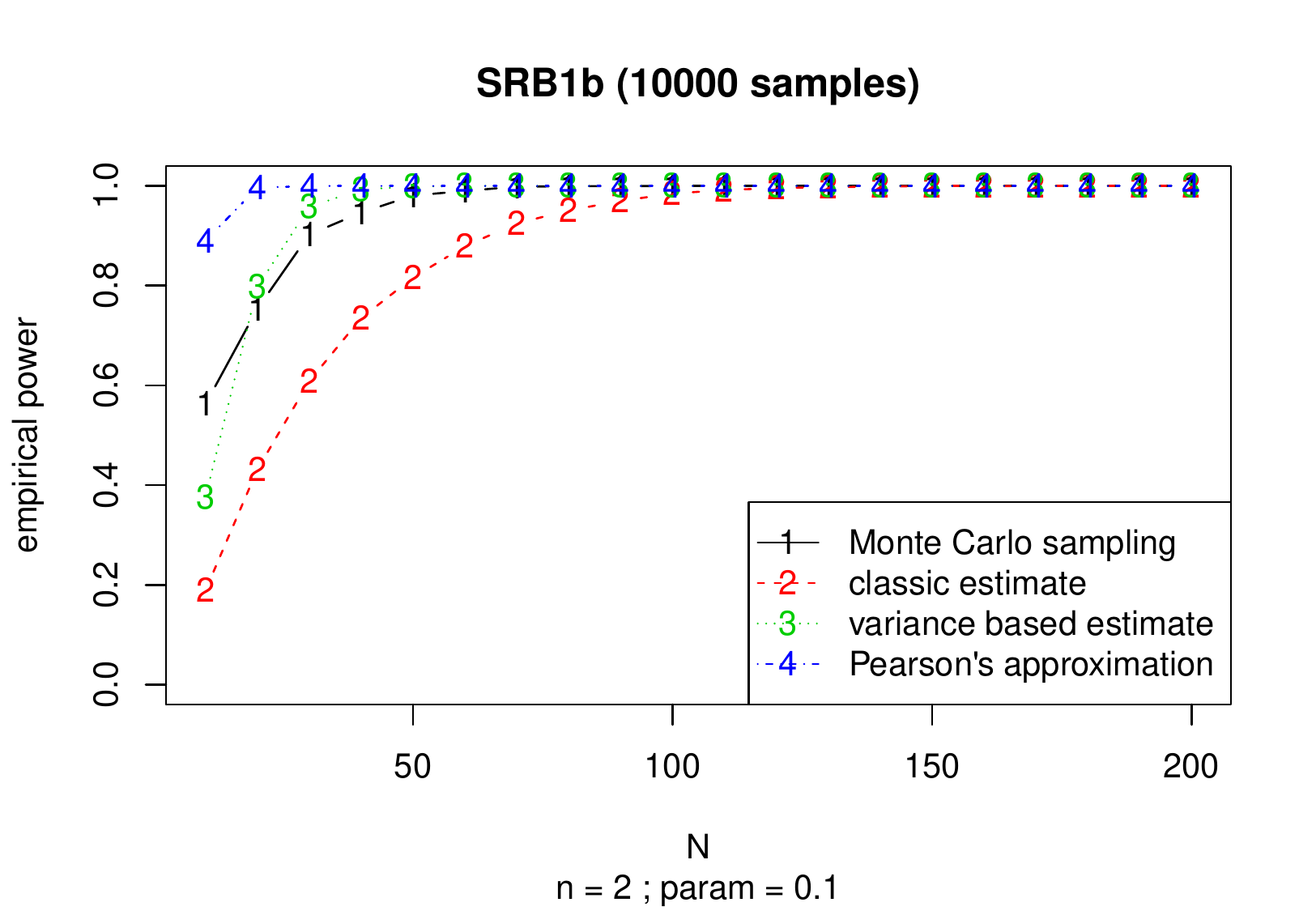}
  \caption{Power of the moment methods for dependent multivariate Student-t
    distributed marginals using normalized distance multivariance (Ex.\
    \ref{ex:robust}).}
  \label{fig:robust-power-normalized}
\end{figure}

\begin{example}[All methods, all examples] \label{ex:all} 
To give recommendations for the application of our methods we performed a (huge)
study. Here we give a brief description, more details can be found in
Section  \ref{sec:study} in the Appendix.

We considered the previous examples and many other examples discussed in
\cite{Boet2019,SzekRizzBaki2007,YaoZhanShao2017} (see Section \ref{sec:study}
in the Appendix for more details). For each (that is also for each specific
parameter choices in the examples, e.g., sample sizes, dimensions,
correlations parameters) we computed (for 10000 samples) the benchmark
p-values and our estimates. Hereto
we used various combinations of the moment estimations and p-value estimation
as discussed in Section \ref{sec:tests}.

Then the relative mean squared error of the estimates in comparison to the
benchmark was computed (considering only the relevant tail). Moreover, also
for each case it was noted if the p-value was conservative in comparison to
the benchmark (plus some margin of error). The results are summarized in
Figure \ref{fig:all} and details are given in Section \ref{sec:study} in the
Appendix.

For the methods described in Section \ref{sec:tests} we use in the figures
 the abbreviations: \texttt{c1} classical estimate \eqref{eq:c1},
\texttt{cv} variance based estimate \eqref{eq:cv},
\texttt{pe} Pearson's approximation \eqref{eq:pearson} and \texttt{clt}
central limit theorem method \eqref{eq:clt}. For the moment estimators we
append to the methods name: limit \texttt{l} or finite sample \texttt{N},
biased \texttt{b} or unbiased \texttt{u}. Finally we appended \texttt{.no} if
the normalized multivariance was used. Recall that we did not derive a finite
sample estimator for the skewness, thus it is always estimated by its limit
estimator.

Before deriving any conclusions note that by definition the results of this
comparison depend on the choice of the utility function (i.e., here we prefer
methods which 'are not liberal and have small relative mean squared error')
and on the choice of examples. Nevertheless also with other utility functions
and for other subsets of the data (see Section \ref{sec:study} in the Appendix
for more details) the following observations seem essential: 
\begin{itemize}
\item Pearson's estimate \eqref{eq:pearson} used with the unbiased finite
  sample estimators for normalized multivariance performs very good. It has
  the lowest relative mean squared error for the estimates, thus it is closest
  to the benchmark power/empirical size. 
  
  In the case of dependent multivariate marginals it appears for very small
  samples to be (too) liberal, i.e., more powerful than the benchmark --
   this requires further future investigation.
\item The variance estimate \eqref{eq:cv} (in particular, used with the
  unbiased finite sample estimators for normalized multivariance) and the
  classical estimate \eqref{eq:c1} are conservative. Theoretically the
  classical estimate is more conservative and this also shows in the
  dependence examples, see also Example \ref{ex:normal_tetrahedron}. If the
  variance is large then both methods coincide.
\item For $m$-multivariance with identically distributed marginals and $n$ not
  too small (e.g., $n>10$) also the central limit theorem method
  \eqref{eq:clt} shows good performance.
\item There are some deviations which are notable: 
\begin{itemize}
\item In some examples the relative mean squared error increases with
  decreasing sample size (e.g., for $N<30$) -- which seems somehow natural
  given the greater variability of estimates based on smaller samples. Anyway,
  in this setting the resampling method could be considered as an alternative,
  in particular, since for small $N$ the speed advantage of the moment methods
  is less pronounced (Example \ref{ex:speed}).
\item With increasing dimension ($n$ large, e.g., $n>30$) also p-values of
  Pearson's estimate \eqref{eq:pearson} become conservative for standard
  multivariance and total multivariance.
\item Due to its erratic behavior we excluded the example already discussed in
  Example \ref{ex:robust}, i.e., dependent Student distributed random
  variables with 1 degree of freedom which do not satisfy the moment
  conditions given in Remark \ref{rem:momcond}, i.e., we exclude examples
  failing the theoretic prerequisites for the application of distance
  multivariance.
\end{itemize}
\end{itemize}

\end{example}

\begin{figure}[H]
    \centering
	\includegraphics[width = 0.7\textwidth]{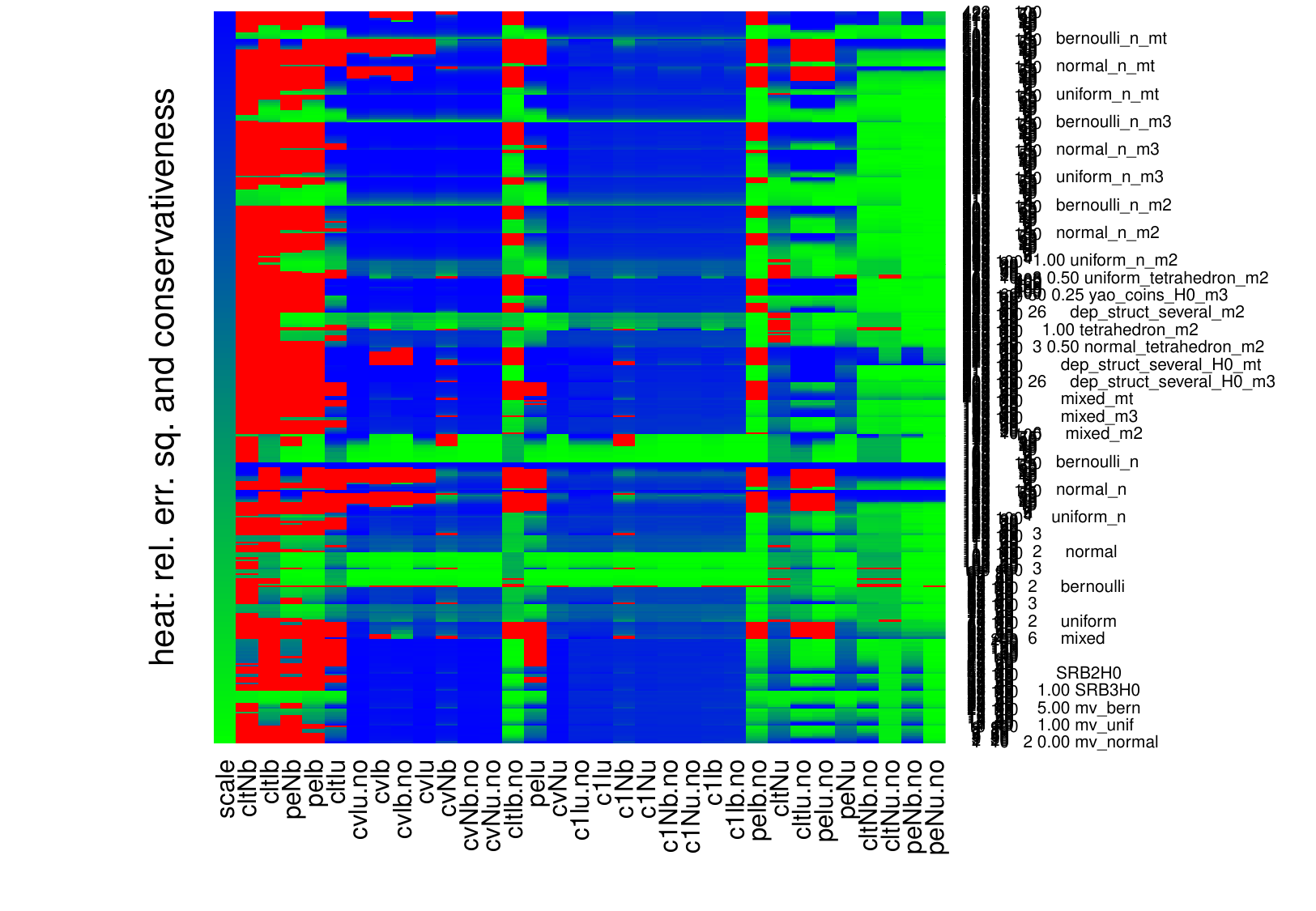}
   \includegraphics[width = 0.7\textwidth]{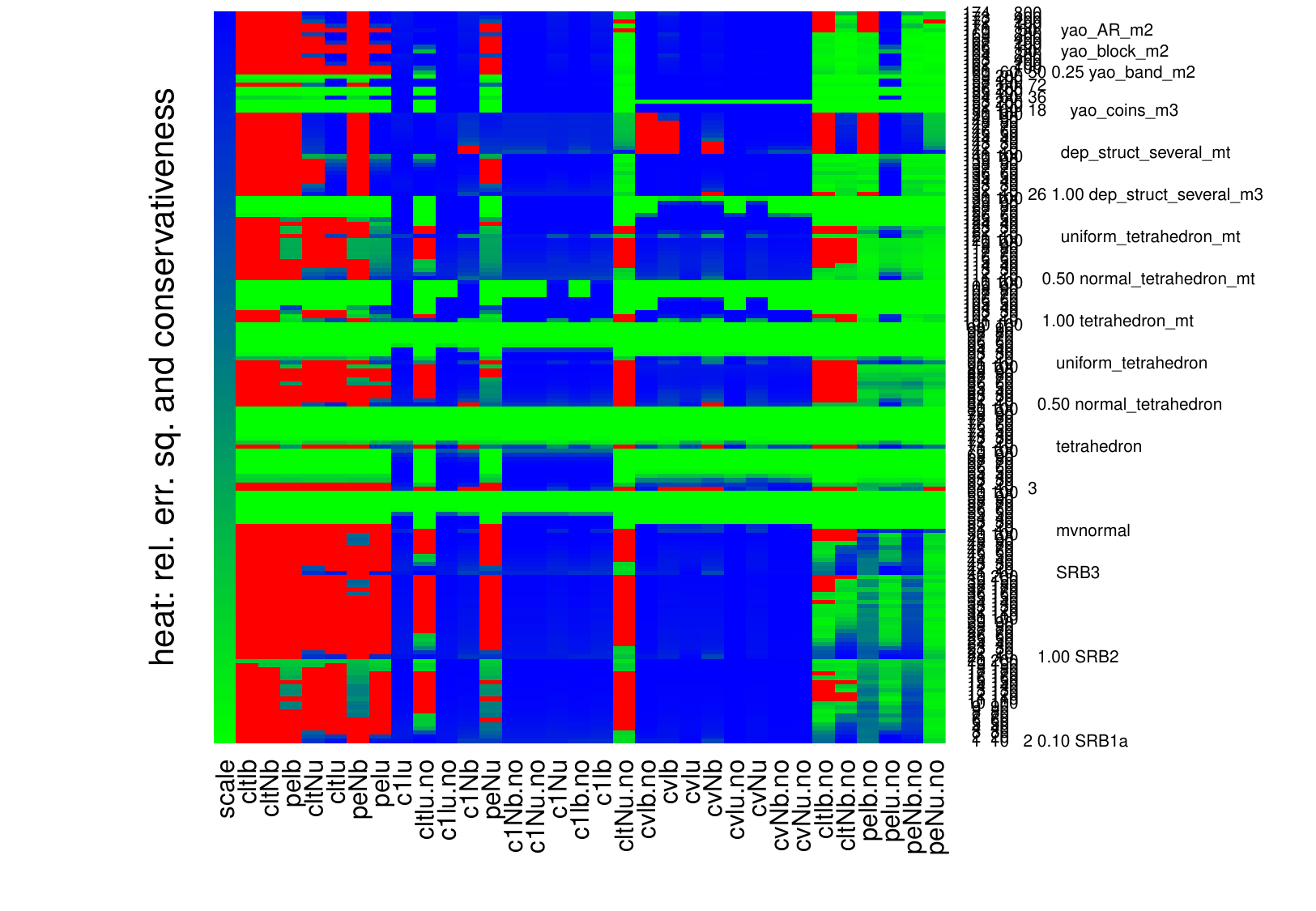}
	\caption{Heatmaps of the performance of the methods (columns) for
          various examples (rows), see Ex.\ \ref{ex:all}. In the first plot
          are $H_0$ examples in the second are dependence examples, for
          details see Section \ref{sec:study} in the Appendix.  The relative
          mean squared error to the benchmark (with cutoff at 1) is colored
          green to blue, and the heat map is overlayed with red for cases
          where the p-value estimates were not conservative. The columns are
          ordered with decreasing error from left to right (treating 'not
          conservative' as error of size 2).}
	\label{fig:all}
\end{figure}


%% file: appendix.tex
\subsection{Details of the comparative study (Example \ref{ex:all})} \label{sec:study}

To complement the brief description in Example \ref{ex:all} we provide here
some more details. First the examples are explained, thereafter some aspects
are discussed.

We consider the following examples (using the labels also used in the figures,
e.g., Figure \ref{fig:all}):

\begin{description}
\item[multivariance:]\hfill
\begin{description}
\item[$H_0$ examples:] \hfill\\
\texttt{mv\_bern}, \texttt{mv\_unif}, \texttt{mv\_normal}:
Bivariate multivariance with independent multivariate marginals of dimension
$5$ (with independent components). \texttt{bern} denotes the Bernoulli
distribution with success probability $\frac{1}{2}$, \texttt{unif} denotes the
uniform distribution on $[0,1]$ and \texttt{normal} denotes the standard
normal distribution. If not stated otherwise, these abbreviations have the
same meaning also in the other examples.\\[3pt]
\texttt{SRB2H0}, \texttt{SRB3H0}: The $H_0$ examples corresponding to
\texttt{SRB2} and \texttt{SRB3} (see below).\\[3pt]
\texttt{student}, \texttt{bernoulli}, \texttt{uniform}, \texttt{normal}: With
the named marginals for $N = 10,20,\dots,100$ and $n=2,3$\\[3pt]
\texttt{bernoulli\_n}, \texttt{uniform\_n}, \texttt{normal\_n}: With the named
marginals for $N= 100$ and $n=3,4,\dots,9,10,15,20,25,30,40,50,75,100$.\\
\texttt{mixed}: $n=6$ with marginals of the exponential (with parameter 1),
normal, Bernoulli, uniform, Poisson (with parameter 1) and binomial
distribution (with parameters 10 and $\frac{1}{2}$).
\item[dependence examples:] \hfill\\
\texttt{mvnormal}: Bivariate standard normal distribution  with scale matrix
$\Sigma = \left(\begin{array}{cc} 1 & 0.1 \\0.1 &1 \end{array}\right).$\\[3pt]
\texttt{SRB1a}: Example 1.(a)  of \cite{SzekRizzBaki2007}. The same as
\texttt{mvnormal} but with multidimensional marginals with dimension 5, the
covariance matrix then coincides with the matrix described in Example
\ref{ex:robust}.\\[3pt]
\texttt{SRB2}: Example 2 of \cite{SzekRizzBaki2007} (multiplicative
dependence):\\  $Y_1,\dots,Y_5,Z_1,\dots,Z_5 \sim N(0,1)$ independent and
$X_1:= (Y_1,\dots,Y_5),$ $X_2 := (Y_1Z_1, \dots, Y_5Z_5)$.\\[3pt]
\texttt{SRB3}: Example 3 of \cite{SzekRizzBaki2007} (non-linear functional
dependence):\\ $Y_1,\dots,Y_5 \sim N(0,1)$ independent and $X_1:=
(Y_1,\dots,Y_5),$ $X_2 := (\log (Y_1)^2, \dots, \log (Y_5)^2)$.\\[3pt]
\texttt{normal\_tetrahedron}, \texttt{uniform\_tetrahedron},
\texttt{tetrahedron}:\\ The normal tetrahedron is defined in Example
\ref{ex:normal_tetrahedron}. Replacing therein $Z_i$ with uniform random
variables or with constant 0 variables yields the other two.\\[3pt]
\end{description}
\end{description}

\begin{description}
\item[$m$-multivariance:] Here the suffix \texttt{\_m2} indicates that
  2-multivariance is considered, analogously for \texttt{\_m3}.\hfill
\begin{description}
\item[$H_0$ examples:]\hfill\\
{}\hspace*{-.75em}\texttt{normal\_tetrahedron\_m2},
\texttt{uniform\_tetrahedron\_m2},\texttt{tetrahedron\_m2},\\
{}\hspace*{-.5em}\texttt{mixed\_m3}, \texttt{mixed\_m2},
\texttt{bernoulli\_n\_m3}, \texttt{bernoulli\_n\_m2}, \texttt{normal\_n\_m3},
\texttt{normal\_n\_m2}, \texttt{uniform\_n\_m3}, \texttt{uniform\_n\_m2} :
Same as the examples having the same name without the suffix.\\[3pt]
\texttt{dep\_struct\_several\_m2}: This is an example with higher order
dependence structure (dependence structure with several disjoint dependence
clusters) discussed in \cite[Example 7.5]{Boet2019}. There are 26 variables of
which are 25 Bernoulli random variables, which are pairwise independent but
have dependences of higher order, and one is an independent standard normal
random variable.\\[3pt]
\texttt{dep\_struct\_several\_H0\_m3}: This is the $H_0$ example to the above
for 3-multivariance, i.e., here the random variables have the same marginal
distributions as above and they are independent.\\[3pt]
\texttt{yao\_H0\_m2}, \texttt{yao\_coins\_H0\_m3}: These are the $H_0$
examples to those with the corresponding names below.
\item[dependence examples:] \hfill\\
\texttt{yao\_AR\_m2}, \texttt{yao\_block\_m2}, \texttt{yao\_band\_m2},
\texttt{yao\_coins\_m3}: These are examples of Yao, Zhang and Shao
\cite{YaoZhanShao2017} which are also discussed in \cite[Example
7.18]{Boet2019}. For the first three $N=60$ and $n = 50,100,200,400,800$ where
the random variables are jointly multivariate normally distributed and their
dependence is auto-regressive (\texttt{AR}), given by a block structure
(\texttt{block}) or a band structure (\texttt{band}). In the \texttt{coins}
example the random variables are pairwise independent Bernoulli random
variables, but (some) triples are dependent. For this the parameters $n =
18,36,72$ and $N = 60,100,200$ are used.\\[3pt]
\texttt{dep\_struct\_several\_m3}: See \texttt{dep\_struct\_several\_H0\_m3}.
\end{description}
\end{description}

\begin{description}
\item[total multivariance:] Here the suffix \texttt{\_mt} indicates that total
  multivariance is considered. All these examples are already described above,
  we list them here just for convenience. \hfill
\begin{description}
\item[$H_0$ examples:] \hfill\\
\texttt{bernoulli\_n\_mt}, \texttt{normal\_n\_mt}, \texttt{uniform\_n\_mt},\\
\texttt{dep\_struct\_several\_H0\_mt}, \texttt{mixed\_mt}.
\item[dependence examples:] \hfill\\
\texttt{uniform\_tetrahedron\_mt}, \texttt{normal\_tetrahedron\_mt},\\
\texttt{tetrahedron\_mt}, \texttt{dep\_struct\_several\_mt}.
\end{description}
\end{description}

From these examples samples were generated (10000 samples for each parameter
setting). Then the methods described in Section \ref{sec:tests} were used to
estimate the p-value of each sample.

Thereafter the relative mean squared error of these p-values to the benchmark
was computed. Here the benchmark is the p-value computed by the empirical
distribution of (total, $m$-)multivariance (with and without normalization) of
10000 Monte Carlo samples satisfying $H_0$. The normalized samples were used
for the methods for normalized multivariance. Note that clearly the choice of
the benchmark influences/determines the outcome, but also using the normalized
multivariance as overall benchmark yielded similar results (a direct
comparison showed that for the dependence examples the power of the normalized
benchmark was always equal or higher than the power of the benchmark without
normalization). 

Since the classical estimate \eqref{eq:c1} and the variance based estimate
\eqref{eq:cv} are only tail estimates we skipped those samples where the
p-value and the benchmark were above 0.21 (cf.\ Remark
\ref{rem:tail}.\ref{rem:tail:x0}). If both, the estimate and the benchmark
p-value, were below 0.001 the error was
set to 0, since below this the value of the benchmark (based on a sample of
size 10000) would rely on less than 10 samples. This also seems reasonable
since in applications usually significance levels between 0.1 and 0.001 are
used. For values below 0.001 the exact size is of less (or no) interest.

Moreover, for each parameter setting we noted for how many of the 10000 samples
the estimated p-value was larger than the benchmark plus a margin of
error. This margin was taken to be the minimum of 0.05 and 50\,\% of
the benchmark. The method was then marked to be too liberal (for the given
parameter setting) if more then 30\,\% of the p-values were
overestimated. The choice of these thresholds is somewhat arbitrary. Hereto
note that the benchmark is based on a Monte Carlo sample (of size 10000) and
the relevant tail depends only on a fraction of this. Thus one has to account
for the variability. For the power of the methods the absolute error seems
more relevant and for the empirical size the relative error seems more
important (thus the two different margins).
Furthermore it turned out that also the classical estimate \eqref{eq:c1} --
which is theoretically always sharp or conservative -- sometimes yields
estimates beyond these bounds (see Figure \ref{fig:non-conservative} for the
proportions of non-conservative estimates for all examples and methods). Thus
we opted for allowing a certain percentage of non-conservative p-value
estimates.

Based on the calculated errors and conservativeness the heatmaps
in Figure \ref{fig:all} were computed.

Certainly some of the above is disputable. We are aware of this, and we hope
that our choices are along the interests of the readers. To complement the
study  we briefly discuss some further aspects:
\begin{itemize}
\item One could argue that in applications only tests with a fixed
  significance level $\lambda$ are performed. Thus one would only compare the
  power/empirical size for a fixed significance level and not compare the
  p-values directly. This yields with $\lambda = 0.05$ the results summarized
  in Figure \ref{fig:power-comparison}. In contrast our above method does,
  roughly speaking, a comparison uniformly over all possible $\lambda$, which
  allows the detection of differences which are otherwise lost due to averaging.

\item Instead of our huge comparison one could (and for special situations
  should) concentrate on a subset: For example, see Figure \ref{fig:all-m23}
  for a comparison of $m$-multivariance for cases with $n>10$, it shows that
  the central limit theorem method works well in this setting. It actually
  works also for the partly mixed case of
  \texttt{dep\_struct\_several}. Figure \ref{fig:all-nlarge} shows the
  performance of ($m$-/total) multivariance for cases with $n>30$, note that
  Pearson's estimate only becomes conservative for standard multivariance and
  total multivariance. For $m$-multivariance it works well since by definition
  only multivariances of $m = 2,3$ variables are considered, thus it is
  actually a lower dimensional case.
\end{itemize}

\begin{figure}[H]
    \centering
    \includegraphics[width = 0.7\textwidth]{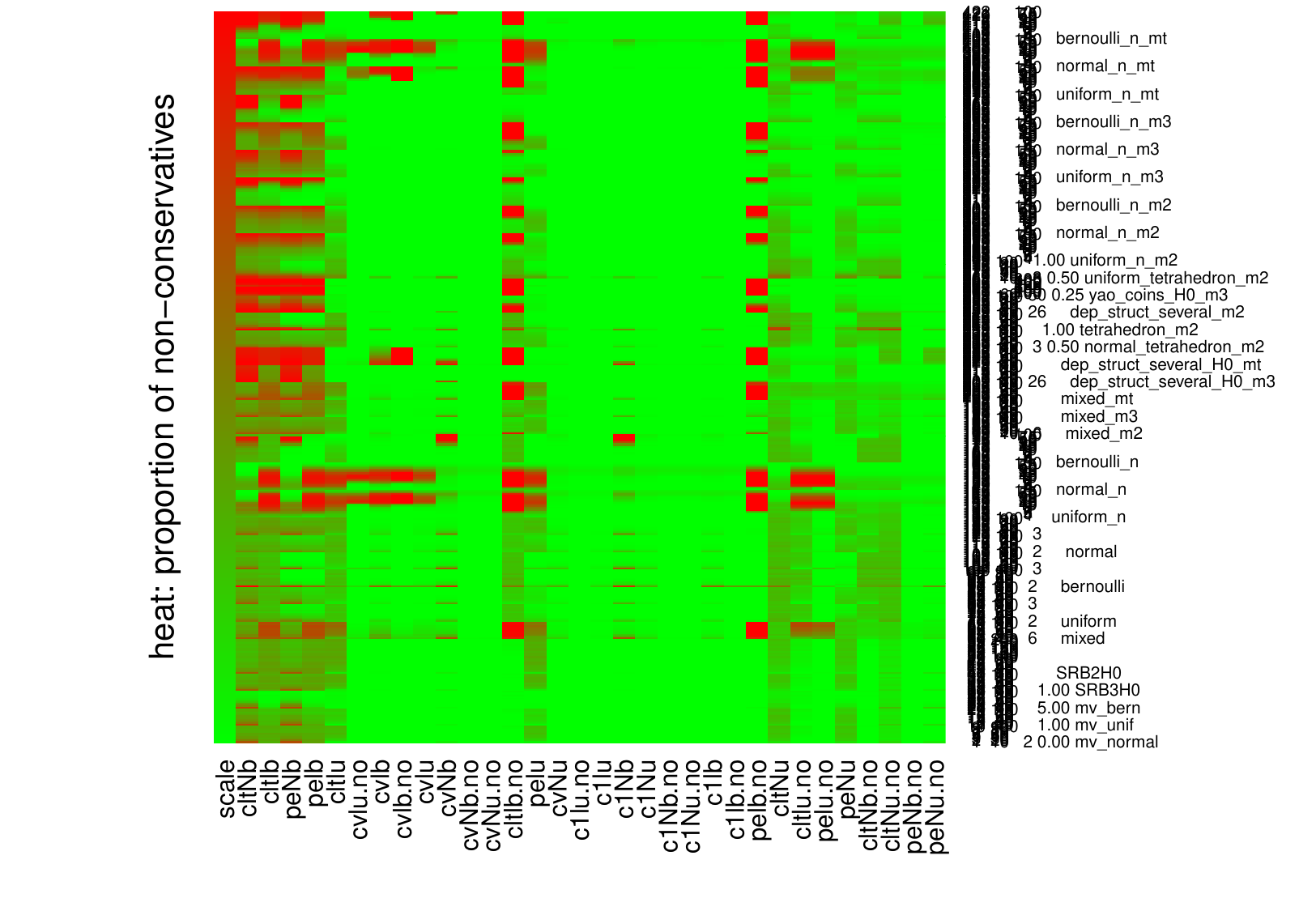}
    \includegraphics[width = 0.7\textwidth]{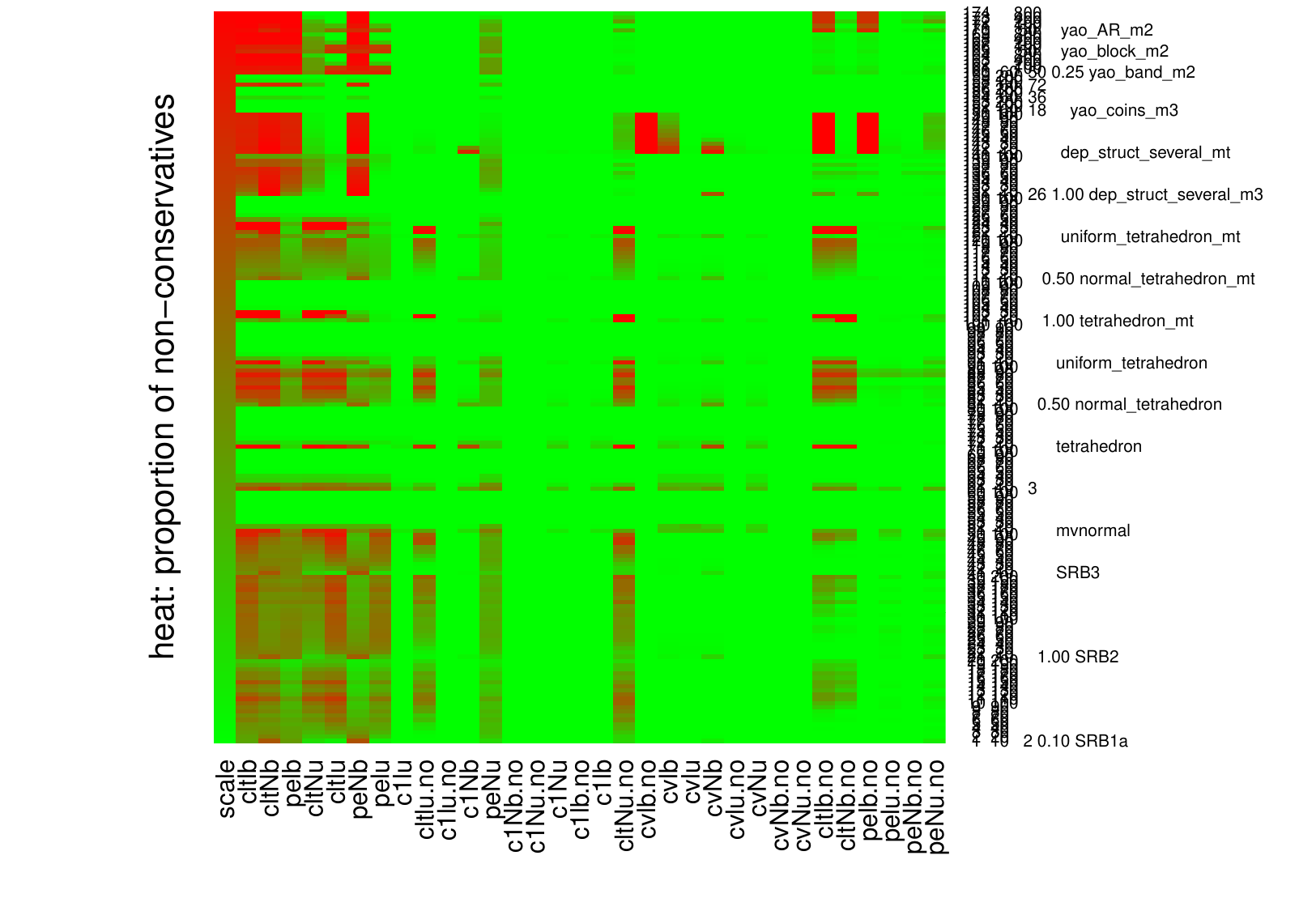}
    \caption{Heatmaps of the same examples in the same order as in Figure
      \ref{fig:all}, but here the rate of non-conservative estimates is
      shown. The scale is from 0 (green) by 1 (red). }
    \label{fig:non-conservative}
\end{figure}

\begin{figure}[H]
    \centering
    \includegraphics[width = 0.7\textwidth]{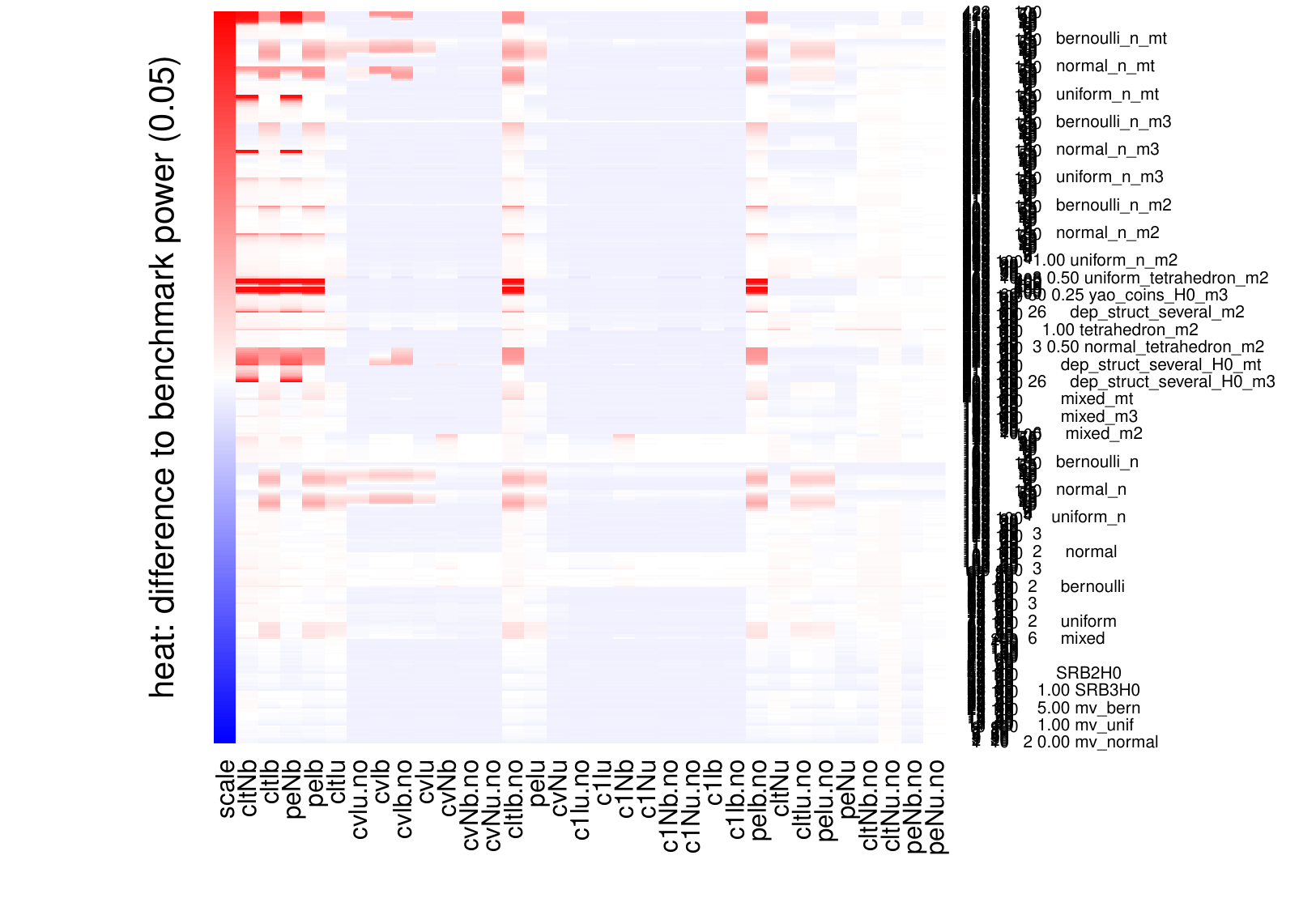}
    \includegraphics[width = 0.7\textwidth]{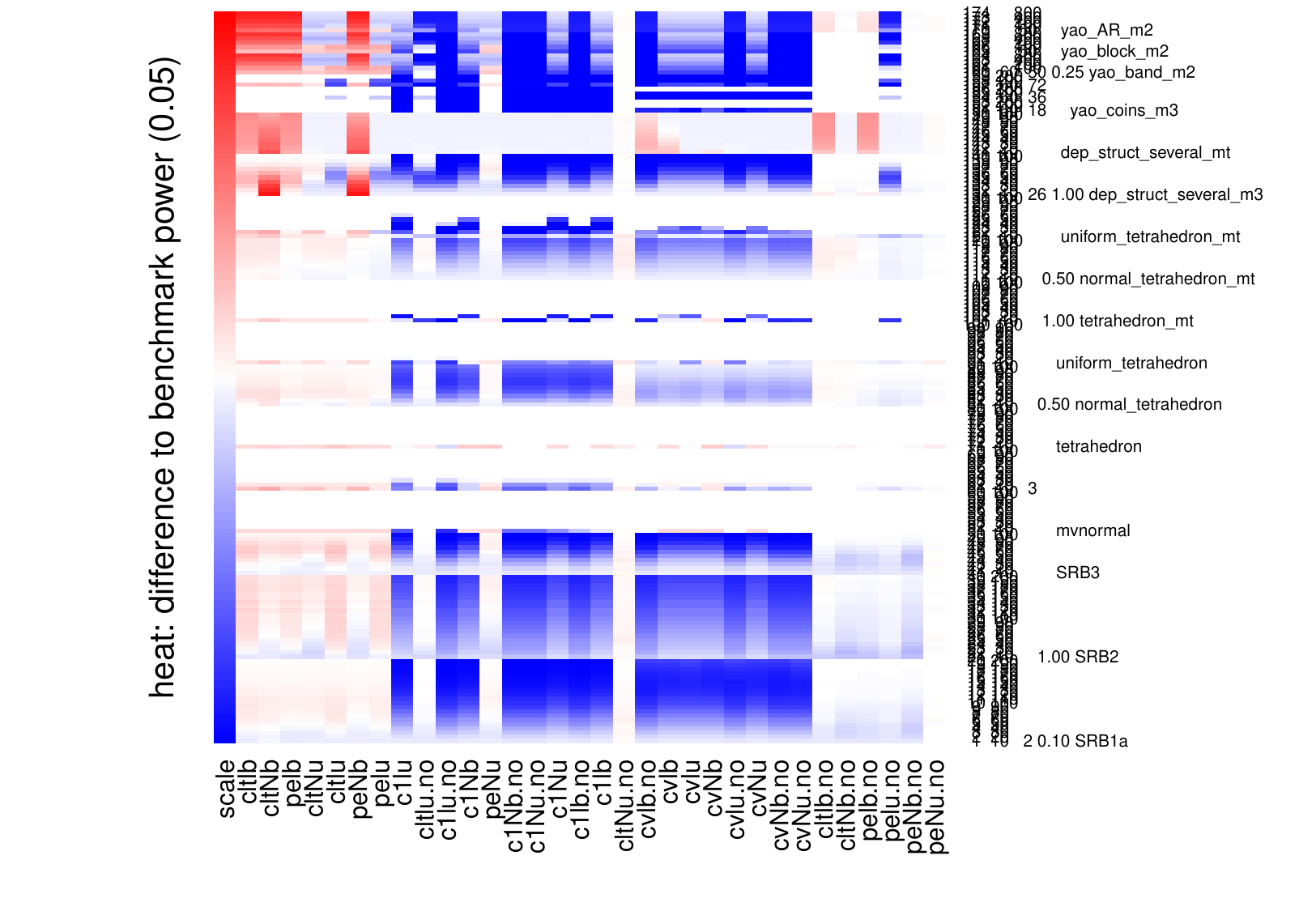}
    \caption{Heatmaps of the same examples in the same order as in Figure
      \ref{fig:all}, but here the differences of the empirical size/power
      of the method and the benchmark is depicted (without overlay). The
      scale is from -1 (blue, conservative) by 0 (white, zero difference)
      to 1 (red, liberal). }
    \label{fig:power-comparison}
\end{figure}

\begin{figure}[H]
    \centering
    \includegraphics[width = 0.7\textwidth]{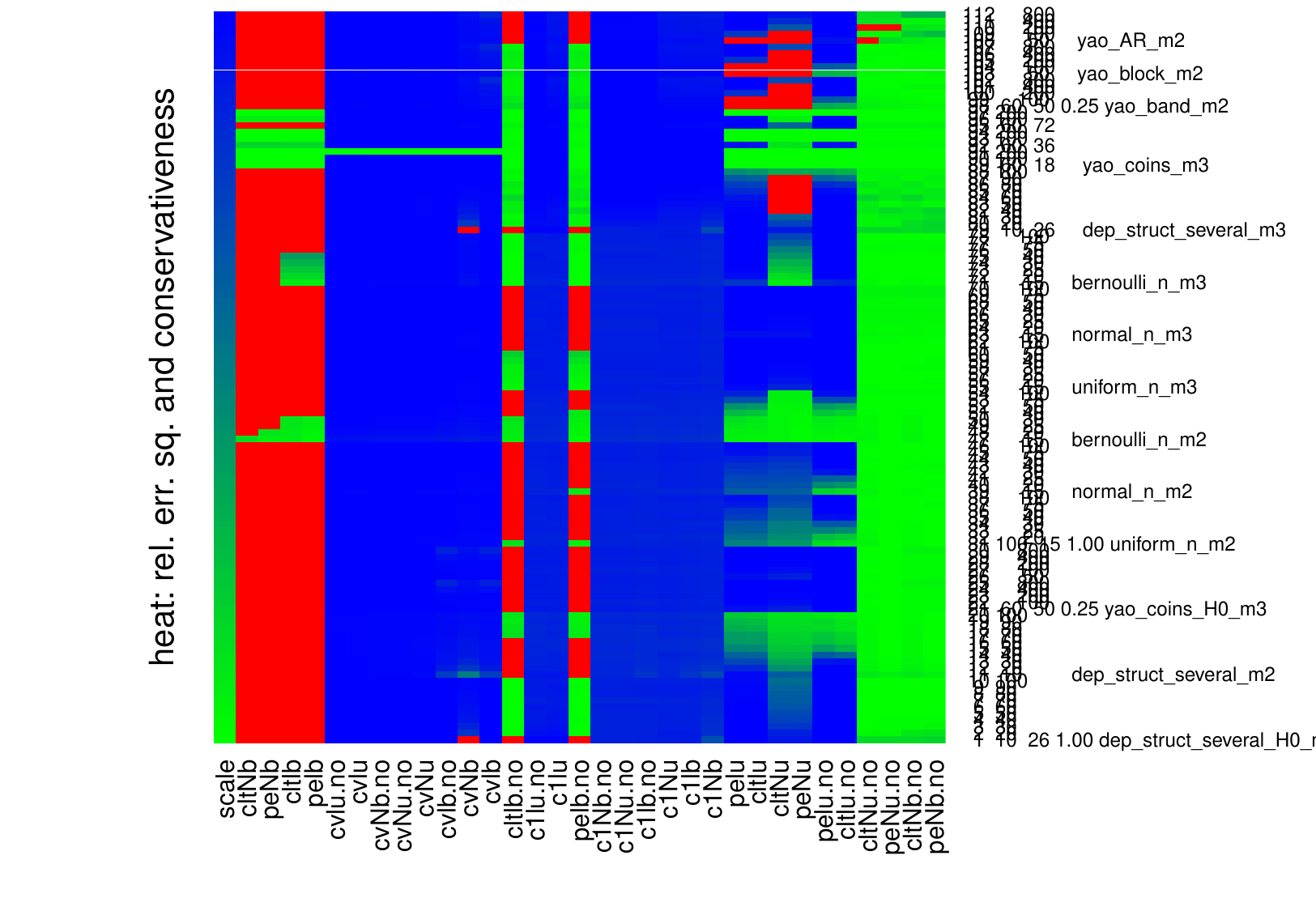}
    \caption{Heatmap of the performance of the methods (columns) for
      various examples (rows). Here: all examples with $m$-multivariance
      and $n>10$. The relative mean squared error to the benchmark (with
      cutoff at 1) is colored green to blue, and the heat map is overlayed
      with red for cases where the p-value estimates were not
      conservative. The columns are ordered with decreasing error from
      left to right (treating 'not conservative' as error of size 2).}
    \label{fig:all-m23}
\end{figure}
  
\begin{figure}[H]
    \centering
    \includegraphics[width = 0.7\textwidth]{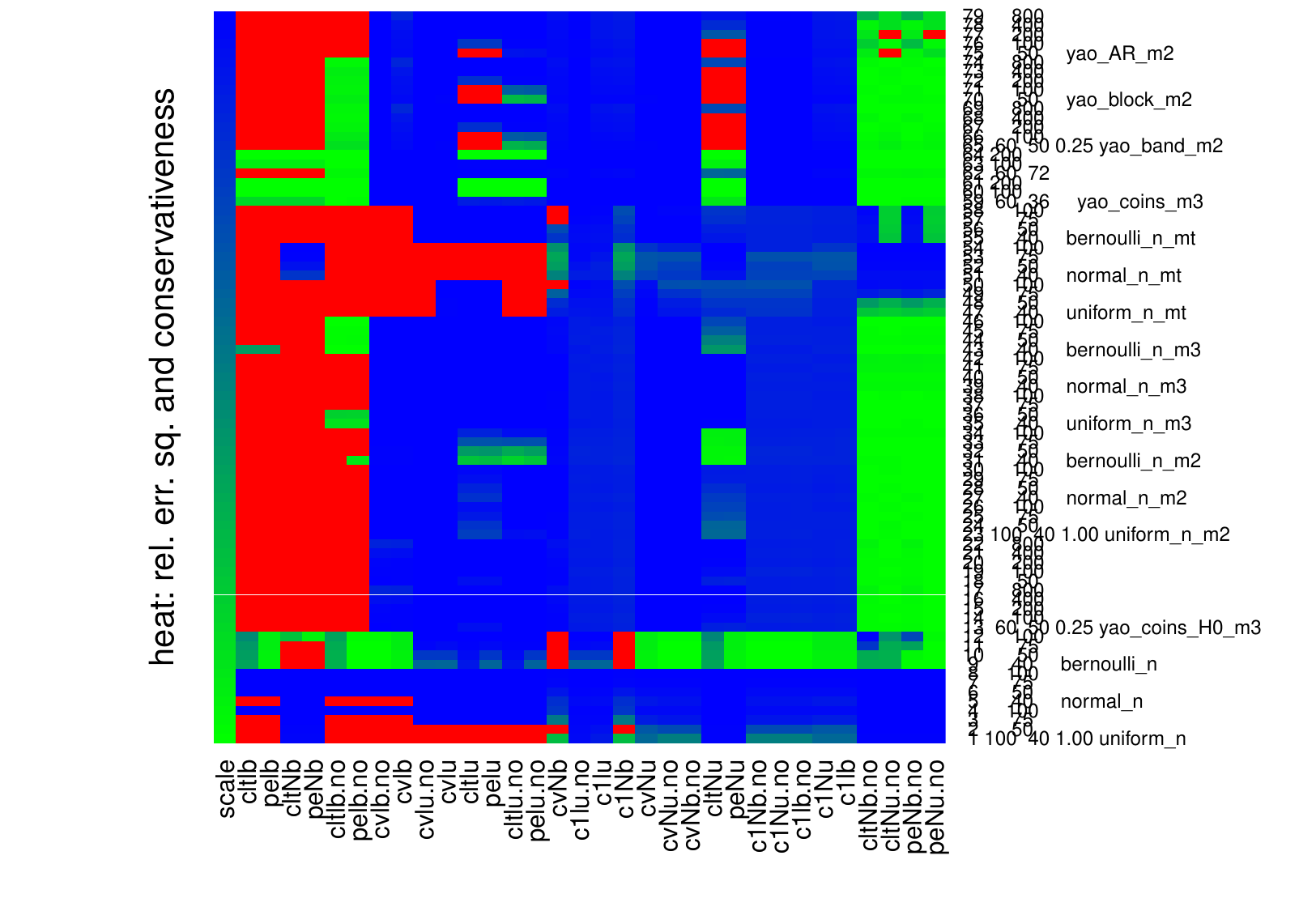}
    \caption{Heatmap of the performance of the methods (columns) for
      various examples (rows). Here: all examples with $n>30$. The
      relative mean squared error to the benchmark (with cutoff at 1) is
      colored green to blue, and the heat map is overlayed with red for
      cases where the p-value estimates were not conservative. The columns
      are ordered with decreasing error from left to right (treating 'not
      conservative' as error of size 2).}
    \label{fig:all-nlarge}
\end{figure}

\subsection{Notes on $x_0$ in the quadratic form estimate \eqref{eq:tail}} \label{sec:x0}

We are interested in an explicit upper bound for $x_0$ in Theorem
\ref{thm:tail}. Hereto recall the setting in the final steps of its proof: Let
$\alpha \in (0,1]$ and $x_0(\alpha)$ be such that
\begin{equation*}
0< \Prob\left( \frac{1}{\lceil 1 / \alpha \rceil }Y_{\lceil
    \frac{1}{\alpha}\rceil} \leq x_0(\alpha)\right) =
\Prob\left(\frac{1}{\lceil 1/\alpha\rceil +1 }Y_{\lceil \frac{1}{\alpha}\rceil
    +1} \leq x_0(\alpha)\right) < 1.
\end{equation*}
By \cite[Prop. 1]{SzekBaki2003} the value of $x_0(\alpha)$ is unique, the
function $\alpha \mapsto x_0(\alpha)$ is increasing on $(0,1]$ and bounded by
$x_0(1)\approx 1.536404$ (see Figure \ref{fig:x0-1} for a plot of
$x_0(\alpha)$ and Figure \ref{fig:x0alphaprob-1} for the corresponding
probabilities). Furthermore, by \cite[Prop. 1', p. 189]{SzekBaki2003},
\begin{equation*}
\Prob(\textstyle \frac{1}{\lceil 1/\beta \rceil} Y_{\lceil 1/\beta \rceil}
\leq x) = \inf_{n\in\N, \frac{1}{n}\leq \beta} \Prob(\textstyle \tfrac{1}{n}
Y_n \leq x)
\end{equation*}
for all $x \geq x_0(\beta)$. By this the statement of Theorem \ref{thm:tail}
was reduced in the proof to the inequality
\begin{equation} \label{eq:finalx0-appendix}
\min\left\{\Prob(\textstyle \frac{1}{\lceil 1/\beta \rceil} Y_{\lceil 1/\beta
    \rceil} \leq x),\ \Prob(\beta Y_{\lfloor \frac{1}{\beta}\rfloor}^{(1)} +
  (1-\beta \lfloor \frac{1}{\beta}\rfloor ) Y_1^{(2)} \leq x) \right\} \geq
\Prob(\alpha Y_\frac{1}{\alpha} \leq x)
\end{equation}
for all $x \geq x_0(\alpha)$ and all $\beta \leq \alpha$. By direct
implementation $x_0(\alpha,\beta)$ -- the smallest $x$ for which
\eqref{eq:finalx0-appendix} holds -- can be computed and it is less than
$x_0(\alpha)$, see Figure \ref{fig:x0-1}.

To illustrate the difficulty of an analytic approach note that we have to
prove for $\beta \leq \alpha$ the two inequalities
\begin{align}
\label{eq:x0-first}\Prob(\textstyle \frac{1}{\lceil 1/\beta \rceil} Y_{\lceil
  1/\beta \rceil} \leq x) \geq &\ \Prob(\alpha Y_\frac{1}{\alpha} \leq x),\\
\Prob(\beta Y_{\lfloor \frac{1}{\beta}\rfloor}^{(1)} + (1-\beta \lfloor
\frac{1}{\beta}\rfloor ) Y_1^{(2)} \leq x) \geq &\ \Prob(\alpha
Y_\frac{1}{\alpha} \leq x).
\end{align}
The first inequality holds since (using results of numerical computations) for
$x\geq x_0(\alpha)$ the function $r \mapsto \Prob(\frac{1}{r} Y_r \leq x)$ is
either increasing or convex with dominant value on the integers (i.e., it
reaches on an interval $[i,i+1]$ the maximum on the right end point). Thus an
analytic proof for \eqref{eq:x0-first} would 'just' amount to a calculation of
the extreme values of this function -- although the density is known explicitly
it seems to be intangible (or at least very technical). For the second
inequality it becomes even more difficult.
\begin{figure}[H]
    \centering
    \includegraphics[width = 0.7\textwidth]{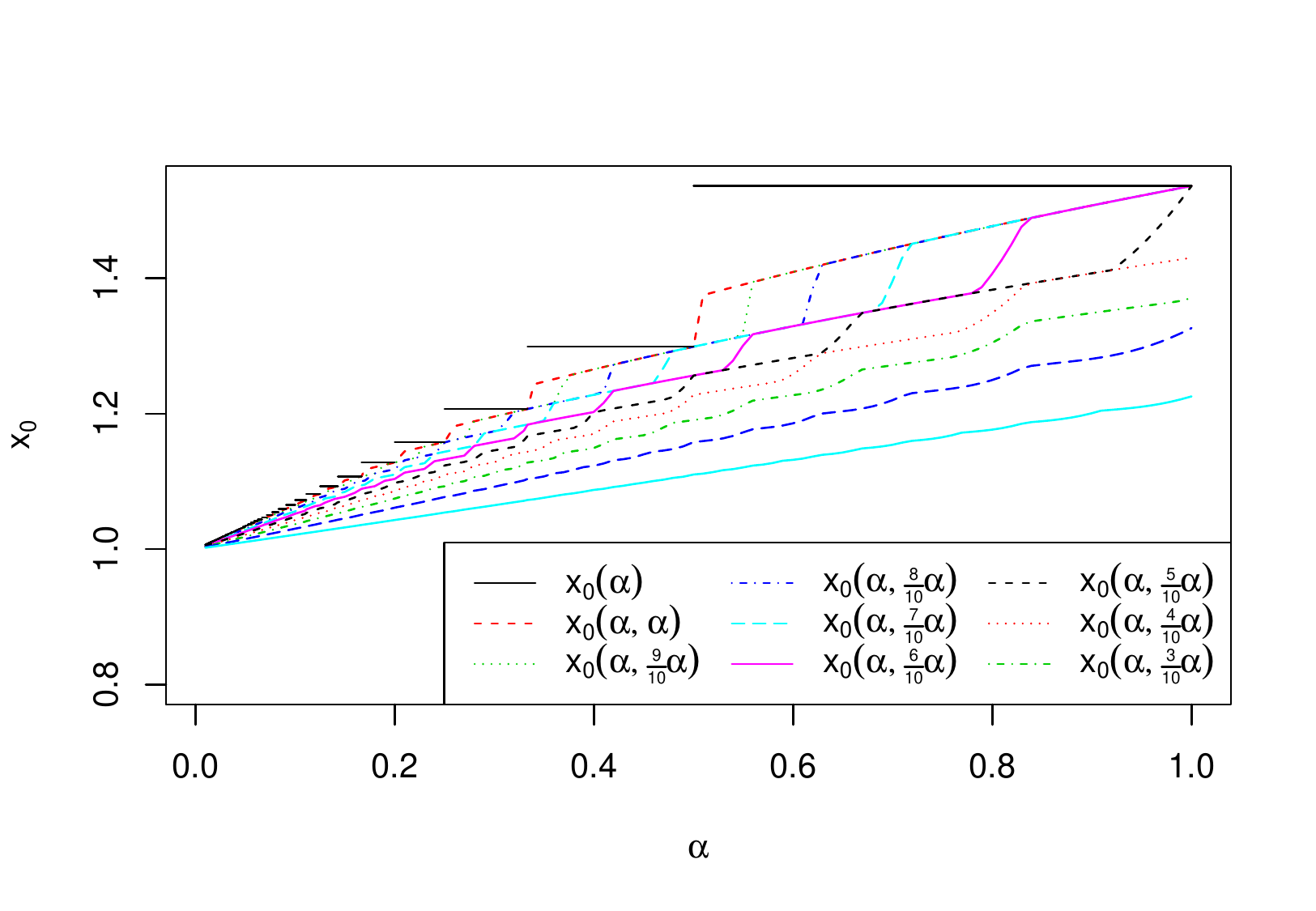}
    \caption{$x_0(\alpha)$ and $x_0(\alpha,\beta)$
    (Section~\ref{sec:x0}).}
    \label{fig:x0-1}
\end{figure}

\begin{figure}[H]
  \centering
  \includegraphics[width = 0.7\textwidth]{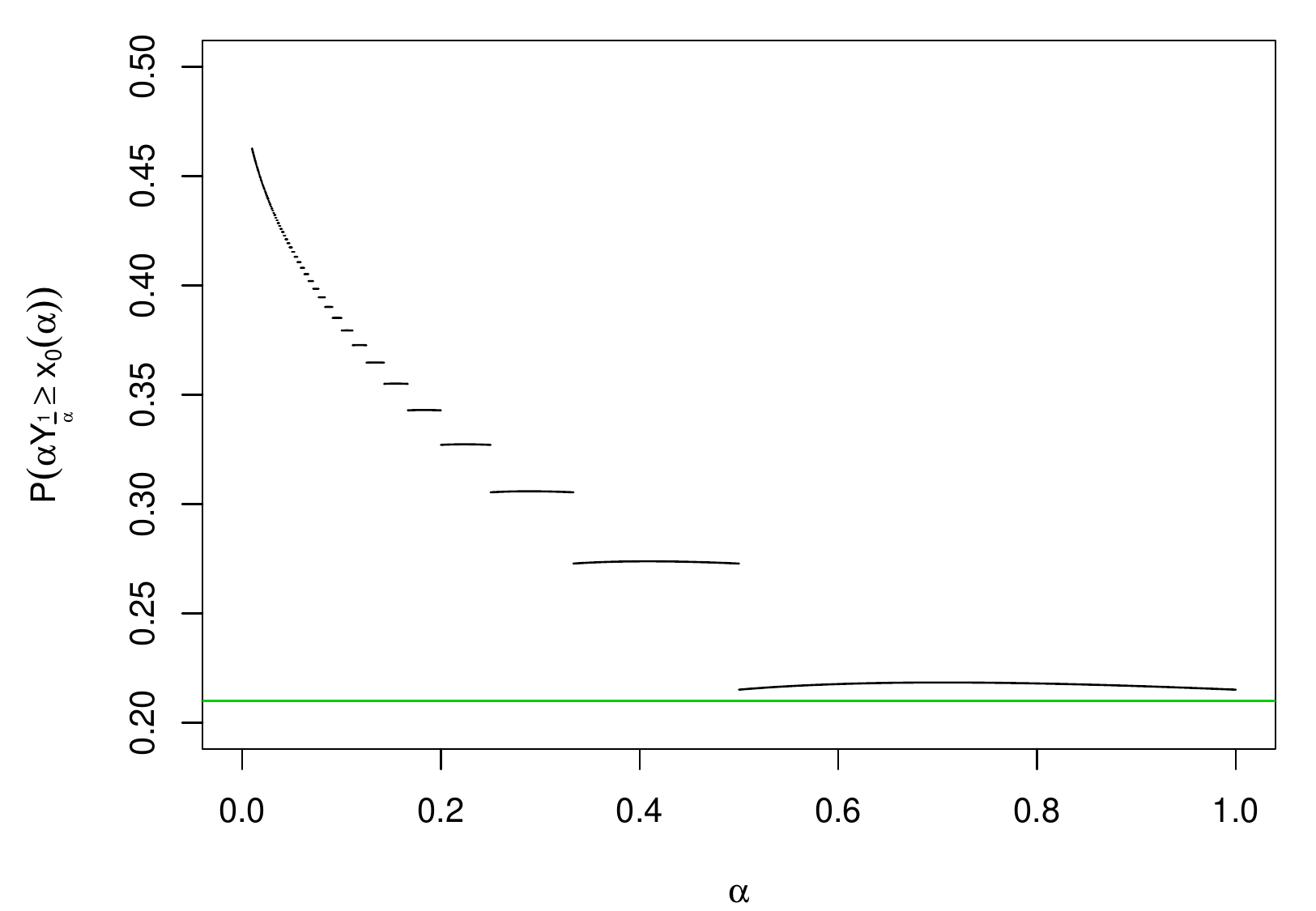}
  \caption{The value of $\alpha \mapsto 1-\Prob(\alpha
    Y_\frac{1}{\alpha} \geq x_0(\alpha))$ with level 0.21 as horizontal
    line (Section~\ref{sec:x0}, see also Remark
    \ref{rem:tail}.\ref{rem:tail:x0}).}
  \label{fig:x0alphaprob-1}
\end{figure}

\subsection{Pearson Type III distribution} \label{sec:pearson}

To relate our results to the Pearson Type III distribution
$P_{\text{III}}(a,b,c)$ ($a,b >0$ and $c\in\R$) recall the density,
characteristic function and moments \cite[formula 26.1.31, p.\
930]{Abra1972}
\begin{align*}
p(x) &= \frac{(x-c)^{a-1}}{b^a \Gamma(a)} e^{-\frac{x-c}{b}}\One_{(c,\infty)}(x) \\
f(t) &= e^{ict} (1-ibt)^{-a}\\
&\text{mean: }c+ab,\text{ variance: }ab^2,\text{ skewness: }\frac{2}{\sqrt{a}}
\end{align*}
The distribution $P_{\text{III}}(a,b,0)$ is commonly also known as
$(a,b)$-gamma distribution, e.g., \cite[Section 6.9]{Zwil2000}. Using our
notation note that for $Y_r \sim \chi^2(r)$
\begin{equation*}
\frac{b}{2} Y_{2a} + c \sim P_{\text{III}}(a,b,c).
\end{equation*}
The right hand side of Pearson's estimate \eqref{eq:pearson} can be rewritten as
\begin{equation*}
\Prob\left(\frac{\sqrt{\Var{Q}}}{\sqrt{2}} \beta
  Y_\frac{1}{\beta^2}-\frac{\sqrt{\Var{Q}}}{\beta \sqrt{2}} + \E(Q) \geq
  x\right),
\end{equation*}
where $\beta = \frac{\skw(Q)}{\sqrt{8}}.$ Thus this is the upper tail
distribution function of the distribution $P_{\text{III}}(a,b,c)$ with $a =
\frac{1}{2\beta^2}, b= 2 \frac{\sqrt{\Var{Q}}}{\sqrt{2}} \beta$ and $c =
-\frac{\sqrt{\Var{Q}}}{\beta \sqrt{2}} + \E(Q)$. Now calculating the moments,
one finds that this is the Pearson Type III distribution with the mean,
variance and skewness of $Q$.

\subsection{Proof of Lemma~\ref{lem:Q4}}\label{sec:proof:Q4}

\begin{proof}[Proof of Lemma~\ref{lem:Q4}]
Since $Z_i \sim N(0,1)$ implies $Z_i^2\sim\chi^2(1)$ and thus
\begin{equation*}
  \Expect{Z_i^2} = 1,\quad \Var{Z_i^2} = 2,\quad
  \Expect{(Z_i^2-\Expect{Z_i^2})^3} = 8,\quad
  \Expect{(Z_i^2-\Expect{Z_i^2})^4}=60.
\end{equation*}
Furthermore $\Expect{Q} = \sum_{i\in\N} \alpha_i$ implies, using independence,
\begin{align*}
  \Var{Q} = \Expect{(Q-\Expect{Q})^2} &= \Expect{\Biggl(\sum_{i\in\N}
    \alpha(Z_i^2-1)\Biggr)^2} \\
  &= \sum_{i,k\in\N} \alpha_i\alpha_k \cov{Z^2_i}{Z^2_k} =  \sum_{i\in\N}
  \alpha_i^2\Var{Z_i^2} = 2 \sum_{i\in\N} \alpha_i^2.
\end{align*}
Analogously, for the third central moment one finds
\begin{align*}
  \Expect{(Q-\Expect{Q})^3} &= \Expect{\Bigl(\sum_{i\in\N}
    \alpha_i(Z_i^2-1)\Bigr)^3} \\
  &= \sum_{i,k,l\in\N} \alpha_i\alpha_k\alpha_l
  \Expect{(Z_i^2-1)(Z_k^2-1)(Z_l^2-1)}
  \\&= \sum_{i\in\N} \alpha_i^3 \Expect{(Z_i^2-1)^3} =
  8\sum_{i\in\N} \alpha_i^3. 
\end{align*}
For the fourth central moment we find
\begin{align*}\MoveEqLeft[3] 
\Expect{(Q-\Expect{Q})^4} =\Expect{\Bigl(\sum_{i\in\N}
    \alpha_i(Z_i^2-1)\Bigr)^4}\\
  &= \sum_{i\in\N} \alpha_i^4 \Expect{(Z_i^2-1)^4} + 3\sum_{i\in\N} \sum_{k\neq
    i}
  \alpha_i^2\alpha_k^2\Expect{(Z_i^2-1)^2}\Expect{(Z_k^2-1)^2} \\
  &=60\sum_{i\in\N} \alpha_i^4 + 12\sum_{i\in\N} \sum_{k\neq
    i} \alpha_i^2\alpha_k^2= 48\sum_{i\in\N} \alpha_i^4 + 12\sum_{i,k\in\N}
  \alpha_i^2\alpha_k^2\\
  &= 48\sum_{i\in\N} \alpha_i^4 + 3\Var{Q}^2.
\end{align*}
Finally, let $\alpha_i \geq 0$ then $\sum_{i\in \N}\alpha_i<\infty$ implies
$\sum_{i\in \N}\alpha_i^k<\infty$ for all $k\in\N$.
\end{proof}

\subsection{Complex normal distribution}\label{sec:complexnormal}

A random vector $Z$ with values in $\C^n$ is \textbf{complex normally
  distributed} with expectation vector $\mu := \E(Z)$, covariance matrix
$\Gamma := \E((Z-\mu)(Z-\mu)^*)$ and pseudo-covariance matrix $C :=
\E((Z-\mu)(Z-\mu)^T)$ if
\begin{equation}\label{eq:multgauss}
\begin{pmatrix} 
  \Re Z \\ \Im Z
\end{pmatrix}
\sim N_{2n}\left(
\begin{pmatrix}
\Re\mu \\ \Im \mu
\end{pmatrix},\frac{1}{2}
\begin{pmatrix} \Re(\Gamma + C) & \Im(-\Gamma + C)\\ \Im(\Gamma + C)&
  \Re(\Gamma - C) 
\end{pmatrix}\right).
\end{equation}
Here, $N_{2n}(\nu,\Sigma)$ denotes the multivariate normal distribution in
$\R^{2n}$ with expectation vector $\nu$ and covariance matrix $\Sigma$. We
write $Z\sim CN_n(\mu,\Gamma,C)$. Instead of using the three-parameter family
$(\mu,\Gamma,C)$ one can also use the two-parameter family of the associated
multivariate normal distribution in \eqref{eq:multgauss} to characterize
complex normal distributions.

Note that for the complex case the knowledge of both the covariance and
pseudo-covariance is needed: Indeed, e.g., starting from $Z\sim
CN_1(0,\gamma,c)$ with $\gamma,c\in\R$, $c\neq 0$, we find that $\Re Z$ and
$\Im Z$ are independent random variables. Choosing $\vartheta\in\R$ such that
$\sin 2\vartheta\neq 0$, the new complex random variable $Z':=\ee^{\ii\vartheta}Z$
has the same (co)variance $\gamma$ as $Z$ but
\begin{equation*}
  \cov{\Re Z'}{\Im Z'} = \frac{c}{2}\sin 2\vartheta\neq 0,
\end{equation*}
i.e., $Z'\sim CN_1(0,\gamma,\ee^{2\ii \vartheta}c)$ has dependent real and
imaginary parts.

Beware that some authors require the pseudo-covariance to vanish
in the definition of the complex normal distribution (similarly, complex
Gaussian random fields have pseudo-covariance kernel $C\equiv
0$); depending on the context such random fields are then called proper
or circular, see \cite{Duchetal2016} and the references within for more
details. 

Finally, as in the real-valued case, for a complex Gaussian vector $(X,Y)$ it
holds that $X$ and $Y$ are independent if and only if $X$ and $Y$ are
uncorrelated, i.e., if both the covariance and the pseudo-covariance are zero.

\subsection{Proof of Proposition \ref{prop:replimit}} \label{sec:proof:replimit}
\begin{proof}[Proof of Proposition \ref{prop:replimit}]
Note that due to the symmetry of $\rho_i$
\begin{equation*}
\begin{split}
\int 1 - \ee^{\ii x_i\cdot t_i} \,\rho_i(dt_i) &= \int 1 - \ee^{- \ii x_i\cdot t_i}
\,\rho_i(dt_i) = \frac{1}{2} \int 2- \ee^{\ii x_i\cdot t_i} - \ee^{-\ii x_i\cdot t_i}
\,\rho_i(dt_i) \\ 
&= \int 1- \cos(x_i\cdot t_i) \,\rho_i(dt_i) = \psi_i(x_i),
\end{split}
\end{equation*}
as well as,
\begin{equation*}
\begin{split}
\int& \bigl(\ee^{\ii x_i\cdot t_i} - f_{X_i}(t_i)\bigr)\bigl(\ee^{-\ii
  y_i\cdot t_i}-f_{X_i}(-t_i)\bigr)\,\rho_i(dt_i) \\
  &=  \int \bigl(\ee^{-\ii
  x_i\cdot t_i} - f_{X_i}(-t_i)\bigr)\bigl(\ee^{\ii y_i\cdot
  t_i}-f_{X_i}(t_i)\bigr)\,\rho_i(dt_i)\\
&= \int \conj{\bigl(\ee^{\ii x_i\cdot t_i} - f_{X_i}(t_i)\bigr)\bigl(\ee^{-\ii
    y_i\cdot t_i}-f_{X_i}(-t_i)\bigr))}\,\rho_i(dt_i).
\end{split}
\end{equation*}
Hence, we find
\begin{equation*}
\begin{split}
\int& \bigl(\ee^{\ii x_i\cdot t_i} - f_{X_i}(t_i)\bigr)\bigl(\ee^{-\ii
  y_i\cdot t_i}-f_{X_i}(-t_i)\bigr))\,\rho_i(dt_i) \\
&=  \int \Re\Bigl[\bigl(\ee^{\ii
  x_i\cdot t_i} - f_{X_i}(t_i)\bigr)\bigl(\ee^{-\ii
  y_i\cdot t_i}-f_{X_i}(-t_i)\bigr)\Bigr]\,\rho_i(dt_i) \\ 
&= \iiint  \Re\Bigl[\bigl(\ee^{\ii
  x_i\cdot t_i} - \ee^{\ii u_i\cdot t_i}\bigr)\bigl(\ee^{-\ii
  y_i\cdot t_i}- \ee^{-\ii v_i\cdot
  t_i}\bigr)\Bigr]\,dF_{X_i}(u_i)\,dF_{X_i}(v_i)\,\rho_i(dt_i)\\
  &= \iiint
  \bigl[\cos((x_i-y_i)\cdot t_i)-\cos((x_i-v_i)\cdot t_i)\\
&\hspace{1.8cm} -\cos((u_i-y_i)\cdot t_i)+\cos((u_i-v_i)\cdot t_i)\bigr]
 \, dF_{X_i}(u_i)\,dF_{X_i}(v_i)\,\rho_i(dt_i)\\
&\overset{(*)}{=} \iiint
  \bigl[\cos((x_i-y_i)\cdot t_i)-1+1-\cos((x_i-v_i)\cdot t_i)
  -\cos((u_i-y_i)\cdot t_i)\\
&\hspace{3cm} +1-1+\cos((u_i-v_i)\cdot
t_i)\bigr]\,\rho_i(dt_i)\,dF_{X_i}(u_i)\,dF_{X_i}(v_i)\\
&=-\psi_i(x_i-y_i) + \E(\psi_i(X_i-y_i)) + \E(\psi_i(x_i-X_i')) -
\E(\psi_i(X_i-X_i'))\\
 &= \Psi_i(x_i,y_i).  
\end{split}
\end{equation*}
The moment condition $\E(\psi_i(X_i))<\infty$ for $1\leq i\leq n$
allows to apply Fubini's theorem in the penultimate step $(*)$ since
\begin{align*}
\MoveEqLeft[0.5]    \iiint
  \bigl|\cos((x_i-y_i)\cdot t_i)-\cos((x_i-v_i)\cdot t_i)-\cos((u_i-y_i)\cdot
  t_i)\\
  &\hspace{4cm} +\cos((u_i-v_i)\cdot t_i)\bigr|\,\rho_i(dt_i)\,
  dF_{X_i}(u_i)\,dF_{X_i}(v_i)\\ 
  &\leq \iiint \Bigl(\abs{\cos((x_i-y_i)\cdot t_i)-1} +
  \abs{1-\cos((x_i-v_i)\cdot t_i)}\\
  &\hspace{2.5cm}
  +\abs{1-\cos((u_i-y_i)\cdot t_i)}+\abs{\cos((u_i-v_i)\cdot t_i)-1}\Bigr)\\
  &\hspace{7.2cm}\rho_i(dt_i)\,dF_{X_i}(u_i)\,dF_{X_i}(v_i)\\
  &= \iiint \Bigl(1-\cos((x_i-y_i)\cdot t_i) + 1-\cos((x_i-v_i)\cdot t_i)\\
  &\hspace{2.5cm}
  +1-\cos((u_i-y_i)\cdot t_i) + 1-\cos((u_i-v_i)\cdot t_i)\Bigr)\\
  &\hspace{7.2cm} \rho_i(dt_i)\,dF_{X_i}(u_i)\,dF_{X_i}(v_i)\\
  &= \psi_i(x_i-y_i) + \int \psi_i(x_i-v_i)\,dF_{X_i}(v_i) + \int
  \psi_i(u_i-y_i)\,dF_{X_i}(u_i) \\
&\hspace{5cm}{}  + \iint \psi_i(u_i-v_i)\,dF_{X_i}(u_i)\,dF_{X_i}(v_i)\\
&\leq \psi_i(x_i-y_i) + 2\psi_i(x_i)  + 4 \int \psi_i(v_i)\,dF_{X_i}(v_i) +
2\psi_i(y_i) + 4\int \psi_i(u)\,dF_{X_i}(u_i)\\
&= \psi_i(x_i-y_i) + 2\psi_i(x_i)  + 2\psi_i(y_i)  + 8
\Expect{\psi_i(X_i)}<\infty,
\end{align*}
where we used the generalized triangle inequality for the negative definite
function $\psi_i$ (e.g., \cite[Equation~(66)]{Boet2019}) in the
penultimate line. Thus,
\begin{align*}
\norm{\pG_S}^2 &= \int \left |\int \prod_{i\in S} \left(\ee^{\ii x_i\cdot
      t_i}-f_{X_i}(t_i)\right)\,d\W(\vect{F}_\vect{X}(x))\right|^2
\,\rho(dt)\\
& = \int \left[\int \prod_{i\in S} \left(\ee^{\ii x_i\cdot
      t_i}-f_{X_i}(t_i)\right)\,d\W(\vect{F}_\vect{X}(x))\right]\cdot \\
&\hspace{3cm}\left[\int \prod_{i\in S} \left(\ee^{-\ii y_i\cdot
      t_i}-f_{X_i}(-t_i)\right)\,d\W(\vect{F}_\vect{X}(y))\right] \,\rho(dt)\\
& \overset{(\star)}{=} \iiint \prod_{i\in S} \Bigl[\left(\ee^{\ii x_i\cdot
      t_i}-f_{X_i}(t_i)\right)\cdot \\
  &\hspace{2.5cm}\left(\ee^{-\ii y_i\cdot
      t_i}-f_{X_i}(-t_i)\right)\Bigr]\,\rho(dt)
  \,d\W(\vect{F}_\vect{X}(x))\,d\W(\vect{F}_\vect{X}(y))
  \\
& = \iint \prod_{i\in S}\int  \Bigl[\left(\ee^{\ii x_i\cdot
    t_i}-f_{X_i}(t_i)\right)\cdot\\
&\hspace{2.5cm}\left(\ee^{-\ii y_i\cdot
    t_i}-f_{X_i}(-t_i)\right)\Bigr]\,\rho_i(dt_i)
\,d\W(\vect{F}_\vect{X}(x))\,d\W(\vect{F}_\vect{X}(y))
\\
& = \iint \prod_{i\in S} \Psi_i(x_i,y_i)
\,d\W(\vect{F}_\vect{X}(x))\,d\W(\vect{F}_\vect{X}(y)).
\end{align*}
For the application of a stochastic Fubini theorem in $(\star)$, a
sufficient condition -- see, e.g., \cite{Vera2012} -- is finiteness of
\begin{multline*}
  C:= \int\Biggl( \iint \biggl| \prod_{i\in S} \Bigl[\left(\ee^{\ii x_i\cdot
        t_i}-f_{X_i}(t_i)\right)\cdot \\ \left(\ee^{-\ii y_i\cdot
        t_i}-f_{X_i}(-t_i)\right)\Bigr] \biggr|^2
  \,d\vect{F}_\vect{X}(x)\,d\vect{F}_\vect{X}(y)\Biggr)^{1/2}\,\rho(dt).
\end{multline*}
Indeed, from  \cite[Eqs.\ (39) ff.]{BoetKellSchi2018} we obtain
\begin{align*}
C &=  \int \Biggl( \prod_{i\in S} \left[ \E\left(\left|\ee^{\ii X_i\cdot
        t_i}-f_{X_i}(t_i)\right|^2\right)\cdot \E\left(\left|\ee^{-\ii
        X_i\cdot t_i}-f_{X_i}(-t_i)\right|^2\right)\right]
\Biggr)^{1/2}\,\rho(dt)\\
&=  \int \Biggl(  \prod_{i\in S} \left[ \E\left(\left|\ee^{\ii X_i\cdot
        t_i}-f_{X_i}(t_i)\right|^2\right)^2\right] \Biggr)^{1/2}\,\rho(dt)\\
&=  \int \prod_{i\in S}  \E\left(\left|\ee^{\ii X_i\cdot
      t_i}-f_{X_i}(t_i)\right|^2\right) \,\rho(dt) = \prod_{i \in S}
\E(\psi_i(X_i - X_i')) \\
&\leq \prod_{i \in S} (4 \E(\psi_i(X_i))) <\infty. \hspace{7.5cm} \qedhere
\end{align*}
\end{proof}

\subsection{On the moment requirements for the approximation of
  $\mu_i^{(k)}$} \label{app:relaxmom}
In Section \ref{sec:multivariance} several representations of $\mu_i^{(k)}$
and the corresponding estimators were developed. Here we show that in fact all
estimators converge if $\E(\psi_i(X_i)) <\infty$ for all $1\leq i\leq n$
(thus no moments of higher order are required!). The basic idea is to
use representations which require only $\E(\psi_i(X_i)) <\infty$, these exist
since $\E(\psi_i(X_i)) <\infty$ implies $\mu_i^{(1)} <\infty$ and thus
$\mu_i^{(k)}<\infty$ for all $k\in\N$ by Proposition
\ref{prop:pnormintegral}. Based on this the convergence under the stronger
moment condition can be extended via an approximation argument using
\eqref{eq:epsmu}, below. Hereto we have to define several new objects:

As in the proof of Theorem \ref{thm:ZNBMlimit} we denote by $\hN f_i$ the
empirical characteristic function of $x_i^{(1)},\dots, x_i^{(N)}$ and define
$\hN K_i(t_i,t_i') =  \hN f_i(t_i-t_i') - \hN f_i(t_i)\hN f_i(-t_i').$ Then
the following representations of the empirical estimators for $\mu_i^{(k)}$
are natural (cf.\ Remark \ref{rem:mu} with the explicit kernels given by
Proposition \ref{pro:moments-ZNlimit}). Analogous to the population versions,
it can be shown that these are just different representations for the $\hN
\mu_i^{(k)}$ defined in Corollary \ref{cor:estlimitmom}:
\begin{equation*}
  \hN \mu_i^{(k)} 
  = \int_{(\R^{d_i})^k} \prod_{j=1}^{k-1}
  \hN K_i(t_i^{(k-j+1)},t_i^{(k-j)})\cdot
  \hN K_i(t_i^{(1)},t_i^{(k)})\,\rho_i^{\otimes k}(dt_i^{(1)},\dotsc,dt_i^{(k)}).
\end{equation*}
Since $\hN K_i$ is positive definite the $\hN \mu_i^{(k)}$ are non-negative,
and  $\hN \mu_i^{(1)} < \infty$ implies $\hN \mu_i^{(k)}<\infty$ for all $k$
(cf.\ Proposition \ref{prop:pnormintegral}). Using the generalized triangle
inequality for continuous negative definite functions \cite[Equation
(66)]{Boet2019} one obtains the bound
\begin{align}\label{eq:mu1bound}
\begin{split}
\hN \mu_i^{(1)} &= \int \hN K_i(t_i,t_i)\,\rho_i(dt_i)\\
& = \E(\psi_i(\hN  X_i - \hN  X_i')) \leq 2 \E(\psi_i(\hN  X_i)) = 2
\frac{1}{N} \sum_{j = 1}^{N} \psi_i(x_i^{(j)}),
\end{split}
\end{align}
where $\hN X_i$ and $\hN X_i'$ are i.i.d. random variables with the given
empirical distribution. We attach one of the indices $\varepsilon$ and
$\overline{\varepsilon}$ to $\hN\mu_i^{(k)}$ if $\rho_i(.)$ is replaced by
\begin{equation*}
  \rho_{i,\varepsilon}(.) := \rho_i(. \cap \{t_s : |t_i|>\varepsilon\}) \text{
    and }\rho_{i,\overline{\varepsilon}}(.) := \rho_i(. \cap \{t_s : |t_i|\leq
  \varepsilon\}),
\end{equation*}
respectively. Note that the corresponding (cf.\
\eqref{def:sm-cdms} for the correspondence) $\psi_{i,\varepsilon}$ and
$\psi_{i,\overline{\varepsilon}}$ are well defined, and $\psi_{i,\varepsilon}
\leq \psi_i$ and $\psi_{i,\overline{\varepsilon}} \leq \psi_i$.  Moreover,
$\psi_{i,\varepsilon}$ is bounded and thus $\E(\psi_{i,\varepsilon}(X_i)^k) <
\infty $ for all $k\in\N$.

Now the following estimate is the key to relax the moment conditions:
\begin{equation} \label{eq:epsmu}
|\mu_i^{(k)}-\hN \mu_i^{(k)}| \leq |\mu_i^{(k)}-
\mu_{i,\varepsilon}^{(k)}|+|\mu_{i,\varepsilon}^{(k)}-\hN
\mu_{i,\varepsilon}^{(k)}|+|\hN \mu_{i,\varepsilon}^{(k)}-\hN \mu_i^{(k)}|.
\end{equation}
We treat the three terms on the right hand side separately, and assume
henceforth that $\E(\psi_i(X_i))<\infty$ for all $1\leq i\leq n$.
The first term converges to 0 by dominated convergence for $\varepsilon \to 0$
(a finite bound exists since $\mu_i^{(1)} <\infty$ by Corollary
\ref{cor:momG}), the second term converges (for fixed $\varepsilon$ as $N \to
\infty$) by the results of Section \ref{sec:distance-multivariance} since in
this case any moment condition is satisfied. It remains to prove the uniform
convergence of the last term:
\begin{equation}
\lim_{\varepsilon\to 0} \limsup_{N\to \infty} |\hN
\mu_{i,\varepsilon}^{(k)}-\hN \mu_i^{(k)}| = 0.
\end{equation}
Hereto note that \eqref{eq:mu1bound} and the strong law of large numbers yields
\begin{equation*}
\begin{split}
\limsup_{N\to \infty} |\hN \mu_{i,\varepsilon}^{(k)}-\hN \mu_i^{(k)}| &=
\limsup_{N\to \infty} \hN \mu_{i,\overline{\varepsilon}}^{(k)} \\
&\leq 2\limsup_{N\to \infty} \frac{1}{N}\sum_{j =1}^N
\psi_{i,\overline{\varepsilon}}(X_{i}^{(j)})\\
& = 2 \E(\psi_{i,\overline{\varepsilon}}(X_i)) \leq  2 \E(\psi_i(X_i)),
\end{split}
\end{equation*}
where the last term is finite by assumption, and thus the penultimate term
converges by dominated convergence to 0 as $\varepsilon \to 0$.

By an analogous argument (e.g., introducing $\hN M_{\rho,\varepsilon},
b_{i,\varepsilon}, c_{i,\varepsilon},d_{i,\varepsilon},\hN h_{i,\varepsilon},
b_{i,N,\varepsilon},$ $c_{i,N,\varepsilon},d_{i,N,\varepsilon}$) one could try
to show that also the approximations for the second finite sample moment
(e.g., in Theorem \ref{thm:finitemoms}) require only $\E(\psi_i(X_i))<\infty$
for all $1\leq i\leq n$. But hereto it is unclear whether the moment
itself is finite (although this seems somehow natural considering the above, the
coefficients in Table \ref{tab:coef} seem to indicate the opposite). One would
have to show that $\E(\psi_i(X_i))<\infty$ implies that $\E(\hN
\Psi_i(X_i^{(j)},X_i^{(k)})\cdot \hN \Psi_i(X_i^{(l)},X_i^{(m)})) <\infty $
for all $j,k,l,m \in \{1,\dots,N\}$, cf.\ the proof of Theorem
\ref{thm:finitemoms}. We leave this as an open problem.

\subsection{Derivation of the unbiased estimator for $\mu_i^{(3)}$}
\label{app:unbiased}

In this section we give a more detailed exposition on the derivation of
unbiased estimators for the summands in the representation \eqref{eq:repk3}
of $\mu_i^{(3)}$, cf.\ Remark \ref{rem:unbiased}. For the sake of simplicity,
we omit in the notation the dependence on the marginal (i.e., we drop the
index $i$) and define $m:=\Expect{\psi(X-X')}$, $d:=m^2$,
\begin{align*}
  c &:=
  \Expect{\psi(X-X')\psi(X'-X'')}, &b &:=  \Expect{\psi(X-X')^2},\\
  e &:= \Expect{\psi(X-X')\psi(X'-X'')\psi(X''-X)},
  &g &:= \Expect{\psi(X-X')^3},\\ 
   f &:= \Expect{\psi(X-X')\psi(X'-X'')\psi(X''-X''')},
   &u &:= m^3,\\
   v &:= \Expect{\psi(X-X')\psi(X-X'')\psi(X-X''')}, &w &:= b\cdot m,\\
   h &:= \Expect{\psi(X-X')^2\psi(X-X'')},&y &:= c\cdot m.
\end{align*}
Then \eqref{eq:repk3} reads $\mu^{(3)} = -e + 3f - 3y + u$.
Given a sample $x = (x^{(1)},\dots,x^{(N)})$ of $X$ and
$B:=(\psi(x^{(j)}-x^{(k)}))_{j,k=1,\dots,N}$ it is straightforward to find
the following unbiased estimators for $m,e,g$ and $h$:
\begin{align*}
  \frac{1}{N (N-1)} |B| \quad \xrightarrow{N\to \infty } \quad &m,\\
   \frac{1}{N (N-1) (N-2)} |B^2\circ B| \quad
\xrightarrow{N\to \infty } \quad &e,\\
\frac{1}{N(N-1)}  \abs{B\circ B\circ B} \quad
\xrightarrow{N\to \infty } \quad &g,\\
 \frac{1}{N(N-1)(N-2)}\bigl(\abs{(B\circ B)\cdot B} -\abs{B\circ B\circ
    B}\bigr)\quad
\xrightarrow{N\to \infty } \quad &h.
\end{align*}
In order to estimate $f$ we start from the obvious (biased) estimator
\begin{equation*}
  \hat f = \frac{1}{N^4} \sum_{j,k,\ell,m=1}^N
  \psi(x^{(j)}-x^{(k)})\psi(x^{(k)}-x^{(\ell)})\psi(x^{(\ell)}-x^{(m)}) =
  \frac{1}{N^4} |B^3|
\end{equation*}
and obtain (for independent copies $X^{(\ell)}$, $1\leq\ell\leq N$, of $X$)
\begin{align*}
    N^4\Expect[big]{\hat f} &= \sum_{j=1}^N \sum_{k\neq j} \Bigl[
  \Expect{\psi(X^{(j)}-X^{(k)})^3} \\
&\hspace{1.2cm} {} +  \sum_{m\neq j,k}
  \Expect{\psi(X^{(j)}-X^{(k)})^2\psi(X^{(j)}-X^{(m)})}  \\ 
&\hspace{1.2cm} {} + \sum_{\ell\neq j,k} \Bigl[
  \Expect{\psi(X^{(j)}-X^{(k)})\psi(X^{(k)}-X^{(\ell)})\psi(X^{(\ell)}-X^{(j)})}
  \\
&\hspace{1.5cm} {} +
  \Expect{\psi(X^{(j)}-X^{(k)})\psi(X^{(k)}-X^{(\ell)})^2} + \\
&\hspace{1.5cm}\sum_{m\neq j,k,\ell}
\Expect{\psi(X^{(j)}-X^{(k)})\psi(X^{(k)}-X^{(\ell)})\psi(X^{(\ell)}-X^{(m)})}\Bigr]
\Bigr] \\
&= N(N-1)g + 2N(N-1)(N-2)h + N(N-1)(N-2)e \\ & \hspace{1cm} + N(N-1)(N-2)(N-3)f.
\end{align*}
Inserting the previously determined unbiased estimators for $e,g$ and $h$ we
arrive at \eqref{eq:approxfunbiased}, i.e.,
\begin{equation*}
  \frac{1}{N (N-1) (N-2) (N-3)}
\bigl(|B^3|-|B^2\circ B| -2|(B\circ B)\cdot B|+|B\circ B\circ B|\bigr)
\quad\xrightarrow{N\to \infty } f.
\end{equation*}
Similarly the estimators for the auxiliary variables $v$ and $w$ are obtained
\begin{align*}
\frac{1}{N (N-1) (N-2) (N-3)}\bigl(|\colsum(B)\circ\colsum(B)\circ\colsum(B)|
\hspace{1cm} \\
{} + 2 \abs{B\circ B\circ B} -  3 \abs{(B\circ B)\cdot
  B}\bigr)\quad&\xrightarrow{N\to \infty } \quad v,\\
  \frac{1}{N (N-1) (N-2) (N-3)}\bigl(\abs{B\circ B}\cdot\abs{B} - 4 \abs{(B\circ
  B)\cdot B} \hspace{1cm} &\\ {} + 2 \abs{B\circ B \circ B} \bigr)
\quad&\xrightarrow{N\to \infty } \quad w,
\end{align*}
where $\colsum(B)$ denotes the vector of the column sums of $B$.

To find an unbiased estimator for $y=c\cdot m$ we start from
\begin{align*}
  \hat y = \frac{1}{N^5} \sum_{j,k,\ell,m,n=1}^N
\psi(x^{(j)}-x^{(k)})\psi(x^{(k)}-x^{(\ell)})\psi(x^{(m)}-x^{(n)}) = \frac{1}{N^5}
\abs{B^2}\cdot\abs{B}
\end{align*}
and obtain
\begin{align*}
  &N^5\Expect{\hat y} \\
  &= \sum_j \sum_{k\neq j} \Bigl[
  2\Expect{\psi(X^{(j)}-X^{(k)})^3} + 4 \sum_{n\neq j,k}
  \Expect{\psi(X^{(j)}-X^{(k)})^2\psi(X^{(j)}-X^{(n)})} \\
&\hspace{.5cm} {} + \sum_{m\neq j,k}\sum_{n\neq
    j,k,m} \Expect{\psi(X^{(j)}-X^{(k)})^2\psi(X^{(m)}-X^{(n)})} \\
&\hspace{.75cm} {}+ \sum_{\ell\neq k,j} \Bigl[
2\Expect{\psi(X^{(j)}-X^{(k)})^2\psi(X^{(k)}-X^{(\ell)})} \\
&\hspace{.75cm} {} +
2\Expect{\psi(X^{(j)}-X^{(k)})\psi(X^{(k)}-X^{(\ell)})^2}  \\
& \hspace{1cm} {} +
2\Expect{\psi(X^{(j)}-X^{(k)})\psi(X^{(k)}-X^{(\ell)})\psi(X^{(\ell)}-X^{(j)})}\\
&\hspace{1cm} {} + 2\sum_{n\neq j,k,\ell}
  \Expect{\psi(X^{(j)}-X^{(k)})\psi(X^{(k)}-X^{(\ell)})\psi(X^{(k)}-X^{(n)})}\\ 
  &\hspace{1cm} {} + 2\sum_{n\neq j,k,\ell}
  \Expect{\psi(X^{(j)}-X^{(k)})\psi(X^{(k)}-X^{(\ell)})\psi(X^{(j)}-X^{(n)})}\\
   &\hspace{1cm} {} +   2\sum_{n\neq
    j,k,\ell}
  \Expect{\psi(X^{(j)}-X^{(k)})\psi(X^{(k)}-X^{(\ell)})\psi(X^{(\ell)}-X^{(n)})}
  \\
&\hspace{1cm} + \sum_{m\neq j,k,\ell} \sum_{n\neq j,k,\ell,m}
\Expect[big]{\psi(X^{(j)}-X^{(k)})\psi(X^{(k)}-X^{(\ell)})
  \psi(X^{(m)}-X^{(n)})}\Bigr]\Bigr]\\
&= 2N(N-1)g + 4 N(N-1)(N-2)h + N(N-1)(N-2)(N-3)w  \\
&\hspace{.25cm} + 4N(N-1)(N-2)h + 2 N(N-1)(N-2)e  +2 N(N-1)(N-2)(N-3)v \\
&\hspace{.5cm} + 4N(N-1)(N-2)(N-3)f + N(N-1)(N-2)(N-3)(N-4) y.
\end{align*}
Inserting all previously determined unbiased estimators we arrive at
\eqref{eq:approxmcunbiased}, i.e.,
\begin{multline*}
  \frac{(N-5)!}{N!}
 \bigl(|B^2|\cdot|B|-|B\circ
 B|\cdot|B|-2|\colsum(B)\circ\colsum(B)\circ\colsum(B)| \\
  {} - 4 |B\circ  B\circ B|-4|B^3|+ 2  |B^2\circ B|+ 10|(B\circ B)\cdot
  B|\bigr) \quad \xrightarrow{N\to \infty } \quad y=m\cdot c.
\end{multline*}
Finally, for $u=m^3$ we start from
\begin{equation*}
  \hat u = \frac{1}{N^6} \sum_{j,k,\ell,m,n,o=1}^N
  \psi(x^{(j)}-x^{(k)})\psi(x^{(\ell)}-x^{(m)})\psi(x^{(n)}-x^{(o)}) =
  \frac{1}{N^6}\abs{B}^3
\end{equation*}
to obtain
{\allowdisplaybreaks\begin{align*}
  &N^6\Expect{\hat u} \\ &= \sum_{j=1}^N \sum_{k\neq j}
  \Bigl[4\Expect{\psi(X^{(j)}-X^{(k)})^3}  \\
&\hspace{.5cm} {} + 2
  \sum_{n\neq j,k}
  \Bigl[ 4 \Expect{\psi(X^{(j)}-X^{(k)})^2\psi(X^{(n)}-X^{(j)})} \\
&\hspace{.75cm} {} + \sum_{o\neq j,k,n}
  \Expect{\psi(X^{(j)}-X^{(k)})^2\psi(X^{(n)}-X^{(o)})}\Bigr] \\ 
&\hspace{.5cm} {} + 2 \sum_{m\neq j,k} \Bigl[
  4\Expect{\psi(X^{(j)}-X^{(k)})^2\psi(X^{(j)}-X^{(m)})} \\
&\hspace{.75cm} {} +
2\Expect{\psi(X^{(j)}-X^{(k)})\psi(X^{(j)}-X^{(m)})\psi(X^{(m)}-X^{(k)})}  \\
&\hspace{.75cm} {}  + 2
\sum_{o\neq j,k,m}
\Expect{\psi(X^{(j)}-X^{(k)})\psi(X^{(j)}-X^{(m)})\psi(X^{(k)}-X^{(o)})}  \\
&\hspace{.75cm} {} +  2 \sum_{o\neq
  j,k,m} \Expect{\psi(X^{(j)}-X^{(k)})\psi(X^{(j)}-X^{(m)})\psi(X^{(m)}-X^{(o)})} \\
&\hspace{.75cm} {} + 2 \sum_{o\neq j,k,m}
\Expect{\psi(X^{(j)}-X^{(k)})\psi(X^{(j)}-X^{(m)})\psi(X^{(j)}-X^{(o)})} \\
&\hspace{.75cm} {}+ \sum_{n\neq
  j,k,m} \sum_{o\neq j,k,m,n}
\Expect{\psi(X^{(j)}-X^{(k)})\psi(X^{(j)}-X^{(m)})\psi(X^{(n)}-X^{(o)})}\Bigr]\\
&\hspace{.5cm} {} +  2\sum_{\ell\neq j,k} \Bigr[ 4
\Expect{\psi(X^{(j)}-X^{(k)})^2\psi(X^{(j)}-X^{(\ell)})} \\
&\hspace{.75cm} {} +
2\Expect{\psi(X^{(j)}-X^{(k)})\psi(X^{(j)}-X^{(\ell)})\psi(X^{(k)}-X^{(\ell)})} \\
&\hspace{.75cm} {} + 4 \sum_{o\neq j,k,\ell}
\Expect{\psi(X^{(j)}-X^{(k)})\psi(X^{(j)}-X^{(\ell)})\psi(X^{(\ell)}-X^{(o)})}
\\
&\hspace{.75cm} {} + 2 \sum_{o\neq j,k,\ell}
\Expect{\psi(X^{(j)}-X^{(k)})\psi(X^{(j)}-X^{(\ell)})\psi(X^{(j)}-X^{(o)})} \\
&\hspace{.75cm} {} + \sum_{n\neq j,k,\ell} \sum_{o\neq j,k,\ell,n}
\Expect{\psi(X^{(j)}-X^{(k)})\psi(X^{(j)}-X^{(\ell)})\psi(X^{(n)}-X^{(o)})}\Bigr]\\
&\hspace{.5cm} {} + \sum_{m\neq j,k,\ell}
\Bigl[ 4\Expect{\psi(X^{(j)}-X^{(k)})^2\psi(X^{(\ell)}-X^{(m)})} \\
&\hspace{.75cm} {} + 8
\Expect{\psi(X^{(j)}-X^{(k)})\psi(X^{(\ell)}-X^{(m)})\psi(X^{(j)}-X^{(\ell)})}\\
&\hspace{.75cm} {} + 8 \sum_{o\neq j,k,\ell,m}
\Expect{\psi(X^{(j)}-X^{(k)})\psi(X^{(\ell)}-X^{(m)})\psi(X^{(j)}-X^{(o)})} \\
&\hspace{.75cm} {} +
\sum_{n\neq j,k,\ell,m} \sum_{o\neq j,k,\ell,m,n}
\!\!\Expect{\psi(X^{(j)}-X^{(k)})\psi(X^{(\ell)}-X^{(m)})
  \psi(X^{(n)}-X^{(o)})}\Bigr]\Bigr]\\
&= N(N-1)\Bigl(4g + 2(N-2)\bigl( 4h+(N-3)w\bigr) \\
&\hspace{.5cm} {} + 4(N-2)\bigl(4h +2e + 4(N-3)f + 2 (N-3)v +
(N-3)(N-4)y\bigr)\\
&\hspace{2cm} {}+ (N-2)(N-3)\bigl(4w + 8f +
8(N-4)y + (N-4)(N-5)u\bigr)\Bigr)\\
&= 4N(N-1) g + 4N (N-1) (N-2) \bigl(6 h + 2e) \\
&\hspace{.5cm} {} + N (N-1) (N-2) (N-3) (6w + 24f+8v)\\
&\hspace{1cm} {} + 12N (N-1) (N-2) (N-3) (N-4) y \\
&\hspace{1.5cm} {} + N(N-1)(N-2)(N-3)(N-4)(N-5)u.
\end{align*}}
Inserting all previously determined unbiased estimators we arrive at
\eqref{eq:approxm3unbiased}, i.e.,
\begin{multline*}
  \frac{(N-6)!}{N!}  \bigl(|B|^3+ 16 |B\circ  B\circ B|-48 |(B\circ B)\cdot
  B|-  8 |B^2\circ B| \\ +6 |B|\cdot|B\circ B|  +24|B^3|  \\ {}
  +16|\colsum(B)\circ\colsum(B)\circ\colsum(B)|   -12|B^2|\cdot|B|\bigr) \quad
  \xrightarrow{N\to \infty }  u=m^3.
\end{multline*}


%% file: sample-distance-multivariance-arXiv.bbl
\begin{thebibliography}{10}

\bibitem{Abra1972}
M.~Abramowitz and I.~A. Stegun, editors.
\newblock {\em Handbook of Mathematical Functions with Formulas, Graphs, and
  Mathematical Tables. 10th Printing, with corr.}
\newblock National Bureau of Standards, 1972.

\bibitem{Atkinson2005}
K.~Atkinson and W.~Han.
\newblock {\em Theoretical Numerical Analysis. {A} Functional Analysis
  Framework}, volume~39 of {\em Texts in Applied Mathematics}.
\newblock Springer, New York, second edition, 2005.

\bibitem{BiloNang2017}
M.~Bilodeau and A.~Guetsop~Nangue.
\newblock Tests of mutual or serial independence of random vectors with
  applications.
\newblock {\em The Journal of Machine Learning Research}, 18(1):2518--2557,
  2017.

\bibitem{Boet2019}
B.~B{\"o}ttcher.
\newblock Dependence and dependence structures: estimation and visualization
  using distance multivariance.
\newblock arXiv: 1712.06532v3, 2019.

\bibitem{Boet2019R-2.0.0}
B.~B\"ottcher.
\newblock {\em multivariance: Measuring Multivariate Dependence Using Distance
  Multivariance}, 2019.
\newblock R package version 2.0.0.

\bibitem{BoetKellSchi2018}
B.~B{\"o}ttcher, M.~Keller-Ressel, and R.~L. Schilling.
\newblock Detecting independence of random vectors: {G}eneralized distance
  covariance and {G}aussian covariance.
\newblock {\em Modern Stochastics: Theory and Applications}, 5(3):353--383,
  2018.

\bibitem{BoetKellSchi2018a}
B.~B{\"o}ttcher, M.~Keller-Ressel, and R.~L. Schilling.
\newblock {D}istance multivariance: New dependence measures for random vectors.
\newblock Accepted for publication in Annals of Statistics. arXiv:1711.07775,
  2018.

\bibitem{BoetKellSchi2018a-supp}
B.~B{\"o}ttcher, M.~Keller-Ressel, and R.~L. Schilling.
\newblock Supplement to ``{D}istance multivariance: New dependence measures for
  random vectors''.
\newblock Accepted for publication in Annals of Statistics, 2018.

\bibitem{Davies1980}
R.~B. {Davies}.
\newblock {The distribution of a linear combination of chi-squared random
  variables. (Algorithm AS 155).}
\newblock {\em {J. R. Stat. Soc., Ser. C}}, 29:323--333, 1980.

\bibitem{Duchetal2016}
G.~R. Ducharme, P.~Lafaye~de Micheaux, and B.~Marchina.
\newblock The complex multinormal distribution, quadratic forms in complex
  random vectors and an omnibus goodness-of-fit test for the complex normal
  distribution.
\newblock {\em Ann. Inst. Statist. Math.}, 68(1):77--104, 2016.

\bibitem{DuchdeM2010}
P.~Duchesne and P.~L. {de Micheaux}.
\newblock Computing the distribution of quadratic forms: Further comparisons
  between the {L}iu-{T}ang-{Z}hang approximation and exact methods.
\newblock {\em Computational Statistics \& Data Analysis}, 54:858--862, 2010.

\bibitem{FanMichPeneSalo2017}
Y.~Fan, P.~L. de~Micheaux, S.~Penev, and D.~Salopek.
\newblock Multivariate nonparametric test of independence.
\newblock {\em Journal of Multivariate Analysis}, 153:189--210, 2017.

\bibitem{Farebrother1984}
F.~W. Farebrother.
\newblock Algorithm {AS} 204: {T}he distribution of a positive linear
  combination of $\chi^2$ random variables.
\newblock {\em {J. R. Stat. Soc., Ser. C}}, 33(3):332--339, 1984.

\bibitem{GretFukuTeoSongScho2008}
A.~Gretton, K.~Fukumizu, C.~H. Teo, L.~Song, B.~Sch{\"o}lkopf, and A.~J. Smola.
\newblock A kernel statistical test of independence.
\newblock In D.~Koller, D.~Schuurmans, Y.~Bengio, and L.~Bottou, editors, {\em
  Advances in Neural Information Processing Systems 21}, pages 585--592. Curran
  Associates Inc, 2008.

\bibitem{Guet2017}
A.~Guetsop~Nangue.
\newblock {\em Tests de permutation d'ind{\'e}pendance en analyse
  multivari{\'e}e}.
\newblock PhD thesis, Universit\'e de Montr\'eal, 2017.

\bibitem{Imhof1961}
J.~P. Imhof.
\newblock Computing the distribution of quadratic forms in normal variables.
\newblock {\em Biometrika}, 48:419--426, 1961.

\bibitem{Jaschke2004}
S.~Jaschke, C.~Kl\"uppelberg, and A.~Lindner.
\newblock Asymptotic behavior of tails and quantiles of quadratic forms of
  {G}aussian vectors.
\newblock {\em J. Multivariate Anal.}, 88(2):252--273, 2004.

\bibitem{Kank1995}
A.~Kankainen.
\newblock {\em Consistent testing of total independence based on the empirical
  characteristic function.}
\newblock PhD thesis, University of Jyv\"askyl\"a, 1995.

\bibitem{KoroBoro1994}
V.~S. Korolyuk and Y.~V. Borovskich.
\newblock {\em Theory of U-statistics}, volume 273.
\newblock Springer Science \& Business Media, 1994.

\bibitem{Kotz1967}
S.~{Kotz}, N.~{Johnson}, and D.~{Boyd}.
\newblock {Series representations of distributions of quadratic forms in normal
  variables. I: Central case. II: Non-central case.}
\newblock {\em {Ann. Math. Statist.}}, 38:823--837, 838--848, 1967.

\bibitem{Kuo1975}
H.~H. Kuo.
\newblock {\em Gaussian Measures in {B}anach spaces}.
\newblock Lecture Notes in Mathematics, Vol. 463. Springer, Berlin-New York,
  1975.

\bibitem{LiuTangZhan2009}
H.~Liu, Y.~Tang, and H.~H. Zhang.
\newblock A new chi-square approximation to the distribution of non-negative
  definite quadratic forms in non-central normal variables.
\newblock {\em Computational Statistics \& Data Analysis}, 53(4):853--856,
  2009.

\bibitem{Lyons2013}
R.~Lyons.
\newblock Distance covariance in metric spaces.
\newblock {\em Ann. Probab.}, 41(5):3284--3305, 2013.

\bibitem{Mathai1992}
A.~{Mathai} and S.~B. {Provost}.
\newblock {\em Quadratic Forms in Random Variables. {T}heory and Applications.}
\newblock New York: Marcel Dekker, 1992.

\bibitem{ReedSimon1980}
M.~Reed and B.~Simon.
\newblock {\em Methods of Modern Mathematical Physics. {I}. Functional
  Analysis}.
\newblock Academic Press, Inc. [Harcourt Brace Jovanovich, Publishers], New
  York, second edition, 1980.

\bibitem{Rube1962}
H.~Ruben.
\newblock Probability content of regions under spherical normal distributions.
  {IV}. {T}he distribution of homogeneous and non-homogeneous quadratic
  functions of normal variables.
\newblock {\em Ann. Math. Statist.}, 33:542--570, 1962.

\bibitem{Sas2013}
Z.~Sasv\'ari.
\newblock {\em Multivariate Characteristic and Correlation Functions},
  volume~50 of {\em De Gruyter Studies in Mathematics}.
\newblock Walter de Gruyter \& Co., Berlin, 2013.

\bibitem{SottViit2015}
T.~Sottinen and L.~Viitasaari.
\newblock Fredholm representation of multiparameter {G}aussian processes with
  applications to equivalence in law and series expansions.
\newblock {\em Modern Stochastics: Theory and Applications}, 2(3):287--295,
  2015.

\bibitem{SteiSco2012}
I.~Steinwart and C.~Scovel.
\newblock Mercer's theorem on general domains: {O}n the interaction between
  measures, kernels, and {RKHS}s.
\newblock {\em Constr. Approx.}, 35:363--417, 2012.

\bibitem{SzekBaki2003}
G.~J. Sz{\'e}kely and N.~K. Bakirov.
\newblock {E}xtremal probabilities for {G}aussian quadratic forms.
\newblock {\em Probability Theory and Related Fields}, 126(2):184--202, 2003.

\bibitem{SzekRizzBaki2007}
G.~J. Sz{\'e}kely, M.~L. Rizzo, and N.~K. Bakirov.
\newblock {M}easuring and testing dependence by correlation of distances.
\newblock {\em The Annals of Statistics}, 35(6):2769--2794, 2007.

\bibitem{Tzir1987}
G.~Tziritas.
\newblock On the distribution of positive-definite {G}aussian quadratic forms.
\newblock {\em IEEE Transactions on Information Theory}, 33(6):895--906, 1987.

\bibitem{Vera2012}
M.~Veraar.
\newblock The stochastic {F}ubini theorem revisited.
\newblock {\em Stochastics An International Journal of Probability and
  Stochastic Processes}, 84(4):543--551, 2012.

\bibitem{YaoZhanShao2017}
S.~Yao, X.~Zhang, and X.~Shao.
\newblock Testing mutual independence in high dimension via distance
  covariance.
\newblock {\em {J. R. Stat. Soc., Ser. B}}, 2017.

\bibitem{Zolo1961}
V.~Zolotarev.
\newblock {C}oncerning a certain probability problem.
\newblock {\em Theory of Probability \& Its Applications}, 6(2):201--204, 1961.

\bibitem{Zwil2000}
K.~Zwillinger.
\newblock {\em {S}tandard probability and statistics tables and formulae}.
\newblock CRC, 2000.

\end{thebibliography}
